\documentclass[a4paper,oneside,11pt]{article}
\usepackage{a4wide, dsfont}
\usepackage{graphicx, color}
\usepackage[latin1]{inputenc}
\usepackage[pagebackref,pdftex,bookmarks,colorlinks,breaklinks]{hyperref}
\usepackage{amsmath}%
\usepackage{amssymb}%
\usepackage{amsthm}
\usepackage{verbatim}
\usepackage{booktabs}
\usepackage{subfigure}

\newcommand{\R}{\ensuremath{\mathbb{R}}}%

\newcommand{\norm}[1]{\Vert #1 \Vert}
\newcommand{\sqnorm}[1]{\Vert #1 \Vert_2^2}
\newcommand{\abs}[1]{\vert #1 \vert}

\DeclareMathOperator*{\EDP}{\reply{EDP}}
\DeclareMathOperator*{\DP}{DP}
\newcommand{\EDPm}{\reply{\mathrm{EDP}_m}}
\newcommand{\DPm}{\mathrm{DP}_m}
\newcommand{\alphaDP}{\hat{\alpha}_{\EDP}}
\DeclareMathOperator*{\mspe}{\reply{MSPE}}
\newcommand{\mspem}{\reply{\mathrm{MSPE}_m}}
\newcommand{\mseem}{\reply{\mathrm{MSEE}_m}}
\DeclareMathOperator*{\psure}{\reply{PSURE}}
\newcommand{\psurem}{\reply{\mathrm{PSURE}_m}}
\newcommand{\alphaPSURE}{\hat \alpha_{\mspe}}
\newcommand{\alphaHatPSURE}{\hat \alpha_{\psure}}
\DeclareMathOperator*{\msee}{\reply{MSEE}}
\DeclareMathOperator*{\gsure}{\reply{SURE}}
\newcommand{\gsurem}{\reply{\mathrm{SURE}_m}}
\newcommand{\alphaGSURE}{\hat \alpha_{\gsure}^*}
\newcommand{\alphaHatGSURE}{\hat \alpha_{\gsure}}
\newcommand{\alphaHatDP}{\hat \alpha_{\DP}}
\newcommand{\Indicator}{\mathds{1}}
\newcommand{\xML}{x_{\rm \tiny{ML}}}

\DeclareMathOperator*{\argmin}{argmin}
\DeclareMathOperator*{\cond}{cond}

\DeclareMathOperator*{\rank}{rank}

\newcommand{\Exp}{\mathbb{E}}

\newcommand{\diag}{\text{diag}}
\newcommand{\trace}{\text{tr}}
\newcommand{\rmd}{\mathrm{d}}
\DeclareMathOperator*{\sign}{sign}

\newtheorem{Theorem}{Theorem}
\newtheorem{Proposition}{Proposition}
\newtheorem{Assumption}{Assumption}
\newtheorem{Lemma}{Lemma}
\newtheorem{Corollary}{Corollary}
\newtheorem{Remark}{Remark}
\newtheorem{Algorithm}{Algorithm}
\newtheorem{Definition}{Definition}

\newcommand{\reply}[1]{ {#1}}

\begin{document}

\title{ Risk Estimators for Choosing Regularization Parameters in Ill-Posed Problems - Properties and Limitations}
\author{Felix Lucka \thanks{Centre for Medical Image Computing, University College London, WC1E 6BT London, UK
email: f.lucka@ucl.ac.uk} \and Katharina Proksch\thanks{Institut f\"ur Mathematische Stochastik, Georg-August-Universit\"at G\"ottingen, Goldschmidtstrasse 7, 37077 G\"ottingen, Germany, e-mail:	kproksc@uni-goettingen.de} \and Christoph Brune\thanks{Department of Applied Mathematics,  University of Twente, P.O. Box 217, 7500 AE Enschede, The Netherlands, e-mail:c.brune@utwente.nl} \and  Nicolai Bissantz\thanks{Fakult\"at f\"ur Mathematik, Ruhr-Universit\"at Bochum, 44780 Bochum, Germany, e-mail: nicolai.bissantz@ruhr-uni-bochum.de} \and Martin Burger\thanks{Institut f\"ur Numerische und Angewandte Mathematik, Westf\"alische Wilhelms-Universit\"at (WWU) M\"unster. Einsteinstr. 62, D 48149 M\"unster, Germany. e-mail: martin.burger@wwu.de  } \and Holger Dette\thanks{Fakult\"at f\"ur Mathematik, Ruhr-Universit\"at Bochum, 44780 Bochum, Germany, e-mail: holger.dette@ruhr-uni-bochum.de} \and Frank W\"ubbeling\thanks{Institut f\"ur Numerische und Angewandte Mathematik, Westf\"alische Wilhelms-Universit\"at (WWU) M\"unster. Einsteinstr. 62, D 48149 M\"unster, Germany. e-mail: frank.wuebbeling@wwu.de  }  }
\maketitle
%
\begin{abstract}
\reply{
This paper discusses the properties of certain risk estimators that recently regained popularity for choosing regularization parameters in ill-posed problems, in particular for sparsity regularization. They apply Stein's unbiased risk estimator (SURE) to estimate the risk in either the space of the unknown variables or in the data space, which we call PSURE
in order to distinguish the two different risk functions. It seems intuitive that SURE is more appropriate for ill-posed problems, since the properties in the data space do not tell much about the quality of the reconstruction. We provide theoretical studies of both approaches for linear Tikhonov regularization in a finite dimensional setting and estimate the quality of the risk estimators, which also leads to asymptotic convergence results as the dimension of the problem tends to infinity. Unlike previous works which studied single realizations of image processing problems with a very low degree of ill-posedness, we are interested in the statistical behaviour of the risk estimators for increasing ill-posedness. Interestingly, our theoretical results indicate that the quality of the SURE risk can deteriorate asymptotically for ill-posed problems, which is confirmed by an extensive numerical study. The latter shows that in many cases the SURE estimator leads to extremely small regularization parameters, which obviously cannot stabilize the reconstruction. Similar but less severe issues with respect to robustness also appear for the PSURE estimator, which in comparison to the rather conservative discrepancy principle leads to the conclusion that regularization parameter choice based on unbiased risk estimation is not a reliable procedure for ill-posed problems. A similar numerical study for sparsity regularization demonstrates that the same issue appears in non-linear variational regularization approaches.
}

\end{abstract}

%
%
%
\section{Introduction} \label{sec:Intro}

Choosing suitable regularization parameters is a problem as old as regularization theory, which has seen a variety of approaches both from deterministic (e.g. L-curve criteria, \cite{Hansen1993,Hansen1992}) or statistical perspectives (e.g. Lepskij principles, \cite{Bauer2005,Lepskii1991}), respectively in between (e.g. discrepancy principles motivated by deterministic bounds or noise variance, cf. \cite{vainikko1982discrepancy,blanchard2012discrepancy}).
{While the particular class of statistical parameter choice rules based on unbiased risk estimation (URE) was used for linear inverse reconstruction techniques early on \cite{rice1986choice,thompson1991study,galatsanos1992methods}, there is a renewed interest in these approaches for iterative, non-linear inverse reconstruction techniques, in particular in the context of sparsity constraints, see e.g., \cite{Ramani2008,Eldar2009,VanDeVille2009,Giryes2011,Luisier2011,VanDeVille2011,Deledalle2012,Deledalle2012a,Dossal2013,Vaiter2013,Wang2013,Deledalle2014,Vaiter2014}). These works are based on extending Stein's general construction of an unbiased risk estimator \cite{Stein1981} to the inverse problems setting. Compared to approaches that measure the risk in the data space, the classical SURE and a generalized version (GSURE, \cite{Eldar2009,Giryes2011,Deledalle2012a,Vaiter2013}) measure the risk in the space of the unknown which seems more appropriate for ill-posed problems. Previous investigations show that the performance of such parameter choice rules is reasonable in many different settings (cf. \cite{HaghshenasLari2014,Weller2014,Candes2013,Almeida2013,Ramani2012,Pesquet2009,Eldar2009}). However, most of the problems considered in these works are very mildly ill-posed (which we will define more precisely below), the interplay between ill-posedness and the performance of the risk estimators is not studied explicitly and the inherent statistical nature of the selected regularization parameter is ignored as only single realizations of noise are typically considered. }

\reply{
Therefore, a first motivation of this paper is to further study the properties of SURE in Tikhonov-type regularization methods from a statistical perspective and systematically in dependence of the ill-posedness of the problem.}
For this purpose we provide a theoretical analysis of the quality of unbiased risk estimators in  the case of linear Tikhonov regularization. In addition, we carry out extensive numerical investigations on appropriate model problems. While in very mildly ill-posed settings the performances of the parameter choice rules under consideration are reasonable and comparable, our investigations yield various interesting results and insights in ill-posed settings. For instance, we demonstrate that \reply{SURE} shows a rather erratic behaviour as the degree of ill-posedness increases. The observed effects are so strong that the meaning of a parameter chosen according to this particular criterion is unclear. 

A second motivation of this paper is to study the discrepancy principle as a reference method and as we shall see it can indeed be put in a very similar context and analysed by the same techniques. Although the popularity of the discrepancy principle is decreasing recently in favour of choices using more statistical details, our findings show that it is still more robust for ill-posed problems than risk-based parameter choices. The conservative choice by the discrepancy principle is well-known to rather overestimate the optimal parameter, but on the other hand it avoids to choose too small regularization as risk-based methods often do. In the latter case the reconstruction results are completely deteriorated, while the discrepancy principle yields a reliable, though not optimal, reconstruction.

\reply{
\paragraph{Formal Introduction}
We consider a discrete inverse problem of the form }
\begin{equation}\label{model}
  y=Ax^* + \varepsilon,
\end{equation}
where $y\in\R^m$ is a vector of observations, $A\in\R^{m\times n}$ is a known matrix, and $\varepsilon\in\R^m$ is a noise vector. We assume that $\varepsilon$ consists of independent and identically distributed (\textit{i.i.d.}) Gaussian errors, i.e., $\varepsilon\sim\mathcal{N}(0,\sigma^2 I_m)$. The vector $x^*\in\R^n$ denotes the (unknown) exact solution to be reconstructed from the observations. \reply{There are two potential difficulties for this: If $A$ has a non-trivial kernel, e.g. for $n > m$, we simply cannot observe certain aspects of $x^*$ and regularization has to interpolate them from the observed features in some way. This, however, typicallyis not the ill-posedness we are interested, in practice we know what we miss and we consider these problems only "mildly ill-posed". The second difficulty is more subtle: The singular values of $A$ might decay very fast, which means that certain aspects of $x^*$ are barely measurable and even small additional noise $\varepsilon$ can render their recovery unstable. This is the main difficulty we are interested in here, so we will measure the degree of ill-posedness of \eqref{model} by the condition of $A$ restricted to its co-kernel, i.e. the ratio between largest and smallest non-zero singular value. Note that this definition deviates from the classical definition of ill-posedness for continuous problems by Hadamard \cite{Ha23}, which leads to a binary classification of problems as either well- or ill-posed but is not very useful for practical applications.}
In order to find an estimate $\hat x(y)$ of $x^*$ from \eqref{model}, we apply a variational regularization method:
\begin{equation} \label{eq:VarReg}
  \hat x_{\alpha}(y)=\argmin_{x\in\R^n} \; \frac{1}{2}\|Ax-y\|_2^2+\alpha R(x),
\end{equation}
where $R$ is assumed convex and such that the minimizer is unique for positive regularization parameter \reply{$\alpha > 0$}.
In what follows the dependence of $\hat x_{\alpha}(y)$ on $\alpha$ and the data $y$ may be dropped where it is clear without ambiguity that $\hat x=\hat x_{\alpha}(y).$

In practice there are two choices to be made: First, a regularization functional $R$ needs to be specified in order to appropriately represent a-priori knowledge about solutions and second, a regularization parameter $\alpha$ needs to be chosen in dependence of the data $y$. The ideal parameter choice would minimize a difference between $\hat x_{\alpha}(y)$ and $x^*$ over all $\alpha$, which obviously cannot be computed and is hence replaced by a parameter choice rule that tries to minimize a worst-case or average error to the unknown solution, which can be referred to as a risk minimization. In the practical case of having a single observation only, the risk based on average error needs to be replaced by an estimate as well, and unbiased risk estimators that will be detailed in the following are a natural choice.

For the sake of a clearer presentation of methods and results  we first focus on linear Tikhonov regularization, i.e.,
\begin{equation*}
	  R(x)=\frac{1}2 \|x\|_2^2,
\end{equation*}
leading to the explicit Tikhonov estimator
\begin{equation} \label{eq:ExplTikhonov}
	\hat x_{\alpha}(y) = T_\alpha y := (A^* A + \alpha I)^{-1} A^* y .
\end{equation}
In this setting, a natural distance for measuring the error of $\hat x_{\alpha}(y)$ is given by its $\ell_2$-distance to $x^*$. Thus, we define 
\begin{equation}
\alpha^* := \argmin_{\alpha \geqslant 0} \; \|\hat x_{\alpha}(y) - x^* \|_2^2 \label{eq:Optimal}
\end{equation}
as the optimal, but inaccessible, regularization parameter. Many different rules for the choice of the regularization parameter $\alpha$ are discussed in the literature. Here, we focus on strategies that rely on an accurate estimate of the noise variance $\sigma^2$. A classical example of such a rule is given by the \emph{discrepancy principle}: The regularization parameter $\alphaHatDP$ is given as the solution of the equation
\begin{equation} \label{DP}
	  \|A\hat x_{\alpha}(y)-y\|^2_2 = m\sigma^2.
\end{equation}
The discrepancy principle is robust and easy-to-implement for many applications (cf. \cite{blomgren2002modular,jin2012iterative,qu2006discrepancy}) and is based on the heuristic argument, that $\hat x_{\alpha}(y)$ should only explain the data $y$ up to the noise level. 
\reply{
The broader class of unbiased risk estimators accounts for the stochastic nature of $\varepsilon$ by aiming to choose $\alpha$ such that it minimizes certain $\ell_2$-errors between $\hat x_{\alpha}(y)$ and $x^*$ only \emph{in expectation:} We first define the \emph{mean squared prediction error}\emph{(MSPE)} as 
\begin{equation}\label{MSPE}
\mspe(\alpha) := \Exp \left[\sqnorm{A \left(x^* - \hat{x}_\alpha (y) \right) } \right]
\end{equation}
and refer to its minimizer as $\alphaPSURE$. Since $\mspe$ depends on the unknown vector $x^*$, we have to replace it by an unbiased estimate we will call $\psure$ here and define:
\begin{equation}\label{PSURE}
\alphaHatPSURE  \in \argmin_{\alpha \geqslant 0} \; \psure(\alpha,y) := \argmin_{\alpha \geqslant 0} \; \sqnorm{y - A \hat{x}_\alpha(y)} - m \sigma^2 + 2 \sigma^2 \text{df}_\alpha (y)
\end{equation}
with 
\begin{equation*}
\text{df}_\alpha (y) = \trace\left(  \nabla_y \cdot A \hat{x}_\alpha (y) \right).
\end{equation*}
While the classical SURE estimator would try to estimate the expectation of the simple $\ell_2$-error between  $\hat x_{\alpha}(y)$ and $x^*$ like in \eqref{eq:Optimal}, a generalization \cite{Eldar2009,Giryes2011} is often considered in inverse problems where $A$ may have a non-trivial kernel: We define the \emph{mean squared estimation error}\emph{(MSEE)} here as 
\begin{equation}\label{RGSURE}
\msee(\alpha) := \Exp \left[\sqnorm{\Pi (x^* - \hat{x}_{\alpha}(y))} \right],
\end{equation}
where $\Pi := A^+ A $ denotes the orthogonal projector onto the range of $A^*$ ( $M^+$ denotes the Pseudoinverse of $M$), and refer to the minimizer of $\msee(\alpha)$ as $\alphaGSURE$. Again, we replace $ \msee$ by an unbiased estimator to obtain 
\begin{equation}\label{defGSURE}
\alphaHatGSURE {\in} \argmin_{\alpha \geqslant 0} \gsure(\alpha,y) :=  \argmin_{\alpha \geqslant 0} \sqnorm{\xML(y) -  \hat{x}_\alpha(y)} - \sigma^2 \trace\left( (AA^*)^+ \right) + 2 \sigma^2 \text{gdf}_\alpha (y)
\end{equation}
with 
\begin{equation*}
\text{gdf}_\alpha (y) = \trace( (A A^*)^+ \nabla_y  A \hat{x}_\alpha (y)), \qquad \qquad \qquad 
\xML = A^+ y=A^* (AA^*)^+ y.
\end{equation*}
 If $A$ is non-singular, as it will be in the theoretical analysis and numerical experiments in this work, the above definition coincides with the classical one considered by Stein \cite{Stein1981}. 

Note that the main difference between the two risk functions MSPE and MSEE and their corresponding estimators PSURE and SURE is that they measure in image and domain of the ill-conditioned operator $A$, respectively. The second important observation here is that all parameter choice rules depend on the data $y$ and hence on the random errors $\varepsilon_1,\ldots,\varepsilon_m$. Therefore, $\alphaHatDP$, $\alphaHatPSURE$ and $\alphaHatGSURE$ are random variables, described in terms of their probability distributions. In the next section, we first investigate these distributions by a numerical simulation study in a simple inverse problem scenario using quadratic Tikhonov regularization. The results point to several problems of the presented parameter choice rules, in particular of \reply{SURE}, and motivate our further theoretical investigation in Section \ref{sec:Theory}. The theoretical results will be illustrated and supplemented by an exhaustive numerical study in Section \ref{sec:NumStudiesL2}.} Finally we extend the numerical investigation in Section \ref{sec:NumStudiesL1} to a sparsity-promoting LASSO-type regularization, for which we find a similar behaviour. Conclusions are given in Section \ref{sec:Conclusion}.

\section{Risk Estimators for Quadratic Regularization} \label{sec:NumSetup}

In the following we discuss the setup in the case of \reply{the simple} quadratic regularization functional $R(x)=\frac{1}{2} \Vert x \Vert^2$, i.e. we recover the well-known linear Tikhonov regularization scheme. The linearity can be used to simplify arguments and gain analytical insight in the next section. \reply{While the arguments presented can easily be extended to more general quadratic regularizations, this model already contains all important properties.}

\subsection{Singular System and Risk Representations} \label{subsec:SinSysQuadReg}

Considering a quadratic regularization  allows to analyze $\hat{x}_\alpha$ in a singular system of $A$ in a convenient way. Let $r=\rank(A)$, $q=\min(n,m)$. Let
\begin{equation*}
A =  U \Sigma V^*, \,\, \Sigma = \diag\left( \gamma_1,\ldots,\gamma_q \right) \in \R^{m \times n}, \,  \gamma_1  \geq\ldots\geq \gamma_r> 0,\, \gamma_{r+1}\ldots\gamma_m:=0
\end{equation*}
denote a singular value decomposition of $A$ with
$$
U=\left(u_1,\ldots,u_m\right)\in\R^{m\times m}, V=\left(v_1,\ldots,v_n\right)\in\R^{n\times n} \text{ unitary}.
$$

Defining 
\begin{equation}\label{spectral}
y_i = \left\langle u_i , y\right\rangle, \qquad x_i^* = \left\langle v_i , x^*\right\rangle,\qquad \tilde \varepsilon_i=\left\langle u_i,\varepsilon \right\rangle
\end{equation}
we can rewrite model \eqref{model} in its spectral form
\begin{equation}\label{modelSpectral}
	  y_i=\gamma_ix_i^*+\tilde\varepsilon_i,\,i=1\ldots q;\,\, y_i=\tilde\varepsilon_i,\,i=q+1\ldots m,
\end{equation}
where $\tilde\varepsilon_1,\ldots,\tilde\varepsilon_m$ are still \textit{i.i.d.} $\sim\mathcal{N}(0,\sigma^2)$.
\reply{All quantities considered in the following depend on $n$ or $m$. 
In particular, we have $A=A_{n,m}$,	$y=y_m$, $x^*=x^*_n$, $\gamma_i=\gamma_{n,m}$, $x_i^*=x_{i,n,m}^*$ and $\tilde{\varepsilon}_i=\tilde\varepsilon_{i,n,m}$. This dependence is made explicit in the statements of the results and technical assumptions for clarity but is dropped in the main text for ease of notation.
	Increasing $m$ corresponds to sampling from an equation such as \eqref{model} more finely, whereas an increase in $n$ increases the level of discretization of an operator $A_{\infty}$ (see section \ref{subsec:Setting}). In our asymptotic considerations both $n$ and $m$ tend to infinity.}\\
 We will express some more terms in the singular system that are frequently used throughout this paper. In particular, we have for $x_{ML}$, the regularized solution $\hat x_{\alpha}$ (dropping the dependence on $y$ below for notational simplicity) and its norm 
\begin{align}\label{eq:xhat}
x_{ML}&=A^+y=V \Sigma^+ U^* y, \text{ with } \Sigma^+=\diag(\frac 1 \gamma_1,\ldots,\frac 1 \gamma_r,0\ldots 0)\in\R^{n\times m}\notag\\
\hat x_{\alpha} &= (A^* A + \alpha I)^{-1} A^* y =: V \Sigma_\alpha^{+} U^* y, \quad \text{with} \quad \Sigma_\alpha^{+} = \diag\left( \frac{\gamma_i}{\gamma_i^2 + \alpha} \right)\in\R^{n\times m} \notag\\
\sqnorm{\hat{x}_\alpha}   &= \sum_{i=1}^m \frac{\gamma_i^2}{(\gamma_i^2 + \alpha)^2} y_i^2
\end{align}
as well as the residual and distance to the maximum likelihood estimate
\begin{align}
\sqnorm{A \hat{x}_\alpha -  y}  &= \sum_{i=1}^m \frac{\alpha^2}{(\gamma_i^2 + \alpha)^2} y_i^2.\label{dataFidSpectral}\\
\sqnorm{\xML - \hat{x}_\alpha} &= \sqnorm{A^*(A A^*)^+ y - (A^*A + \alpha I)^{-1} A^* y} = \sqnorm{ V \left(\Sigma^+ - \Sigma_\alpha^{+} \right) U^* y} \nonumber \\
&= \sum_{i=1}^r \left( \frac{1}{\gamma_i} - \frac{\gamma_i}{(\gamma_i^2 + \alpha)} \right)^2 y_i^2.\nonumber
\end{align}
Based on the generalized inverse we compute
\begin{align*}
(A A^*)^+ &= U (\Sigma \Sigma^*)^+ U^* = U \diag\left(\frac 1 {\gamma_1^2},\ldots,\frac 1 {\gamma_r^2},0,\ldots,0 \right) U^* \\
A^* (A A^*)^+ A &= V \diag( \underbrace{1,\ldots,1}_{r},\underbrace{0,\ldots,0}_{n-r}) V^*,
\end{align*}
which yields the degrees of freedom and the generalized degrees of freedom
\begin{align*}
\text{df}_\alpha &:= \trace ( \nabla_y \cdot A \hat{x}  ) = \trace\left( A (A^*A + \alpha I)^{-1} A^*\right) = \sum_{i=1}^r \frac{\gamma_i^2}{\gamma_i^2 + \alpha}\\ 
\text{gdf}_\alpha &:= \trace( (A A^*)^+ \nabla_y \cdot A \hat{x} ) = \trace\left((A A^*)^+ A (A^*A + \alpha I)^{-1} A^*\right) \nonumber\\
&= \trace((\Sigma \Sigma^*)^+ \Sigma \Sigma_\alpha^{-1}) = \sum_{i=1}^r \frac 1 {\gamma_i^2} \gamma_i \frac{\gamma_i}{\gamma_i^2 + \alpha} = \sum_{i=1}^r \frac{1}{\gamma_i^2 + \alpha}.\nonumber
\end{align*}
Next, we derive the spectral representations of the parameter choice rules. For the discrepancy principle, we use \eqref{dataFidSpectral} to define 
\begin{equation}
\DP(\alpha,y) := \sum_{i=1}^m\frac{\alpha^2}{(\gamma_i^2+\alpha)^2}y_i^2-m\sigma^2, \label{DPspectral}
\end{equation}
and now, \eqref{DP} can be restated as $\DP(\alphaHatDP,y) = 0$. For \eqref{PSURE} and \eqref{defGSURE}, we find
\begin{align}
\psure(\alpha,y) &= \sum_{i=1}^m \frac{\alpha^2}{(\gamma_i^2 + \alpha)^2} y_i^2 - m \sigma^2  + 2 \sigma^2 \sum_{i=1}^m \frac{\gamma_i^2}{\gamma_i^2 + \alpha} \label{PSUREspectral}\\
\gsure(\alpha,y) &= \sum_{i=1}^r \left( \frac{1}{\gamma_i} - \frac{\gamma_i}{\gamma_i^2 + \alpha} \right)^2 y_i^2 - \sigma^2 \sum_{i=1}^r \frac{1}{\gamma_i^2}  + 2 \sigma^2 \sum_{i=1}^r \frac{1}{\gamma_i^2 + \alpha} \label{GSUREspectral}.
\end{align}

\subsection{An Illustrative Example} \label{subsec:Setting}

We consider a simple imaging scenario which exhibits typical properties of inverse problems. The unknown function $x_\infty^*: [-1/2,1/2] \rightarrow \mathbb{R}$ is mapped to a function $y_\infty:[-1/2,1/2] \rightarrow \mathbb{R}$ by a periodic convolution with a compactly supported kernel of width $l \leq 1/2$:
\begin{equation*}
y_\infty(s) = A_{\infty,l} x_\infty^* := \int_{-\frac{1}{2}}^{\frac{1}{2}} k_l \left( s-t \right) x_\infty^*(t) \, \rmd t, \quad s \in [-1/2,1/2],
\end{equation*}
where  the 1-periodic $C_0^\infty(\R)$ function $k_l(t)$ is defined for $|t|\leq 1/2$ by
\begin{equation*}
k_l(t) := \frac{1}{N_l}
\begin{cases}
 \exp\left(-\frac{1}{1 - {t^2}/{l^2}} \right) &\text{if} \quad |t| < l\\
0 &l\leq |t| \leq 1/2
\end{cases}, \quad \qquad N_l = \int_{-l}^l \exp\left(-\frac{1}{1 - {t^2}/{l^2}} \right) \, \rmd t,
\end{equation*}
and continued periodically for $|t|> 1/2$. Examples of $k_l(t)$ are plotted in Figure \ref{subfig:Kernel}. The normalization ensures that $A_{\infty,l}$ and suitable discretizations thereof have the spectral radius $\gamma_1 = 1$ which simplifies our derivations and the corresponding illustrations. The  $x_\infty^*$ used in the numerical examples is the sum of four delta distributions:
\begin{align*}
x_\infty^*(t) := \sum_{i=1}^4 a_i \delta\left(b_i - \frac{1}{2}\right), \; \text{with} \;\; a = [0.5,1,0.8,0.5], \;\; b = \left[\frac{1}{\sqrt{26}},\frac{1}{\sqrt{11}},\frac{1}{\sqrt{3}},\frac{1}{\sqrt{3/2}}\right].
\end{align*}
The locations of the delta distributions approximate $[-0.3,-0.2,0.1,0.3]$ by irrational numbers which \reply{will simplify the discretization of this continuous problem}. 
\paragraph{Discretization}
For a given number $n \in \mathbb{N} $, let
\begin{equation*}
E_i^n := \left[\frac{i-1}{n}-\frac{1}{2},\frac{i}{n}-\frac{1}{2}\right], \qquad i=1,\ldots,n
\end{equation*}
denote the equidistant partition of $[-1/2,1/2]$ and $\psi_i^n(t) = \sqrt{n} \, \Indicator_{E^n_i}(t)$ an 
orthonormal basis  (ONB) of piecewise constant functions over that partition. If we use $m$ and $n$ degrees of freedom to discretize range and domain of $A_{\infty,l}$, respectively, we arrive at the discrete inverse problem \eqref{model} with
\begin{align}
\left(A_l\right)_{i,j} &= \left\langle \psi_i^m, A_{\infty,l} \psi_j^n \right\rangle = \sqrt{mn}\int_{E^m_i} \int_{E^n_j} k_l \left( s-t \right) \, \rmd t \, \rmd s \label{eq:DefA}\\
x^*_j &= \left\langle \psi_j^n, x_\infty^* \right\rangle = \sqrt{n} \int_{E^n_j} x_\infty^*(t) \, \rmd t = \sqrt{n} \sum_i^4 a_i \Indicator_{E^n_i}\delta\left(b_i - \frac{1}{2}\right)\nonumber
\end{align}
The two dimensional integration in \eqref{eq:DefA} is computed by the trapezoidal rule with equidistant spacing, employing $100 \times 100$ points to partition $E^m_i \times E^n_i$. Note that we drop the subscript $l$ from $A_l$ whenever the dependence on this parameter is not of importance for the argument being carried out. 

\begin{figure}[tb]
   \centering
\subfigure[][]{
\includegraphics[width=0.47 \textwidth]{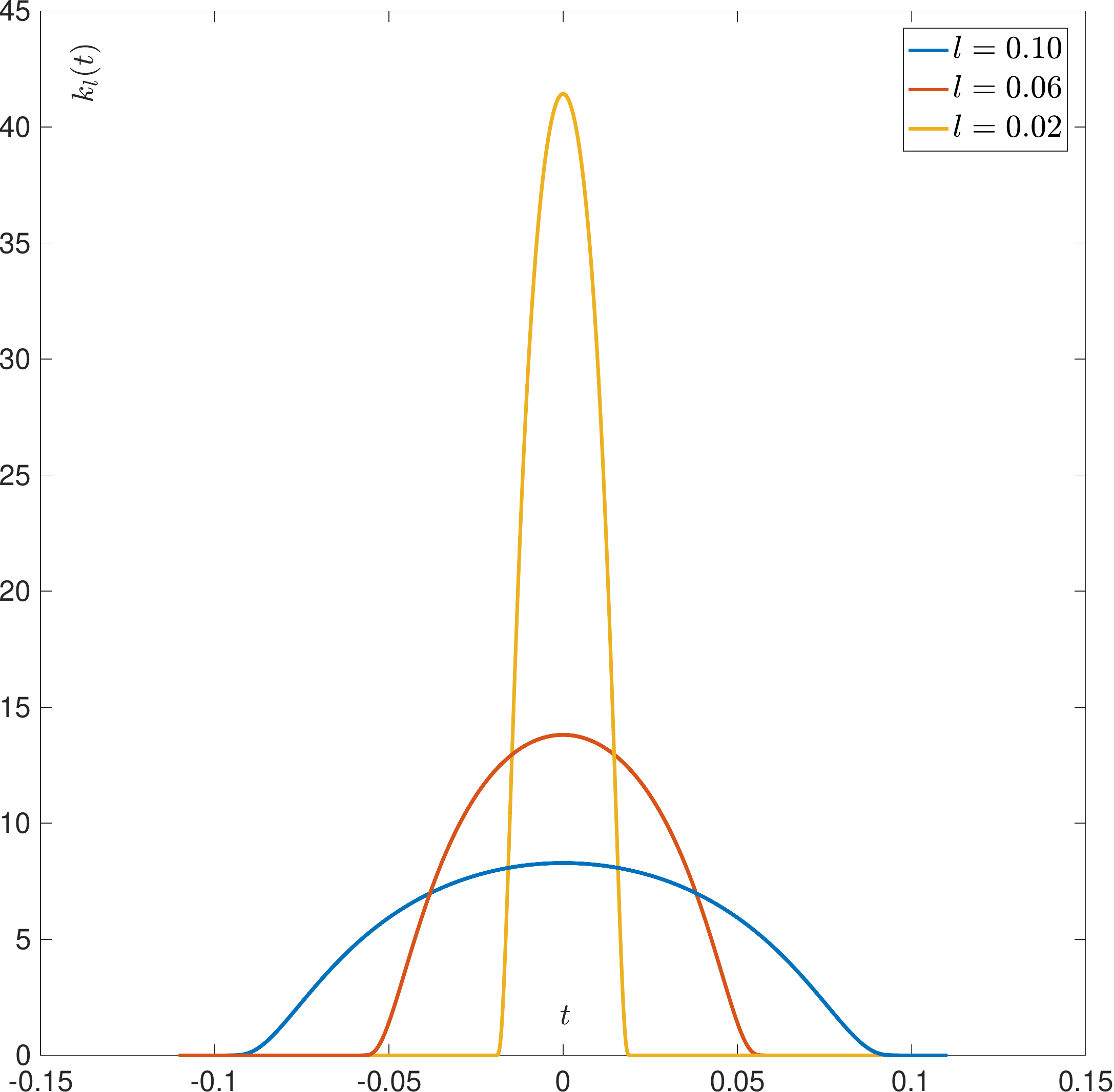} \label{subfig:Kernel}}
\subfigure[][]{
\includegraphics[width=0.46 \textwidth]{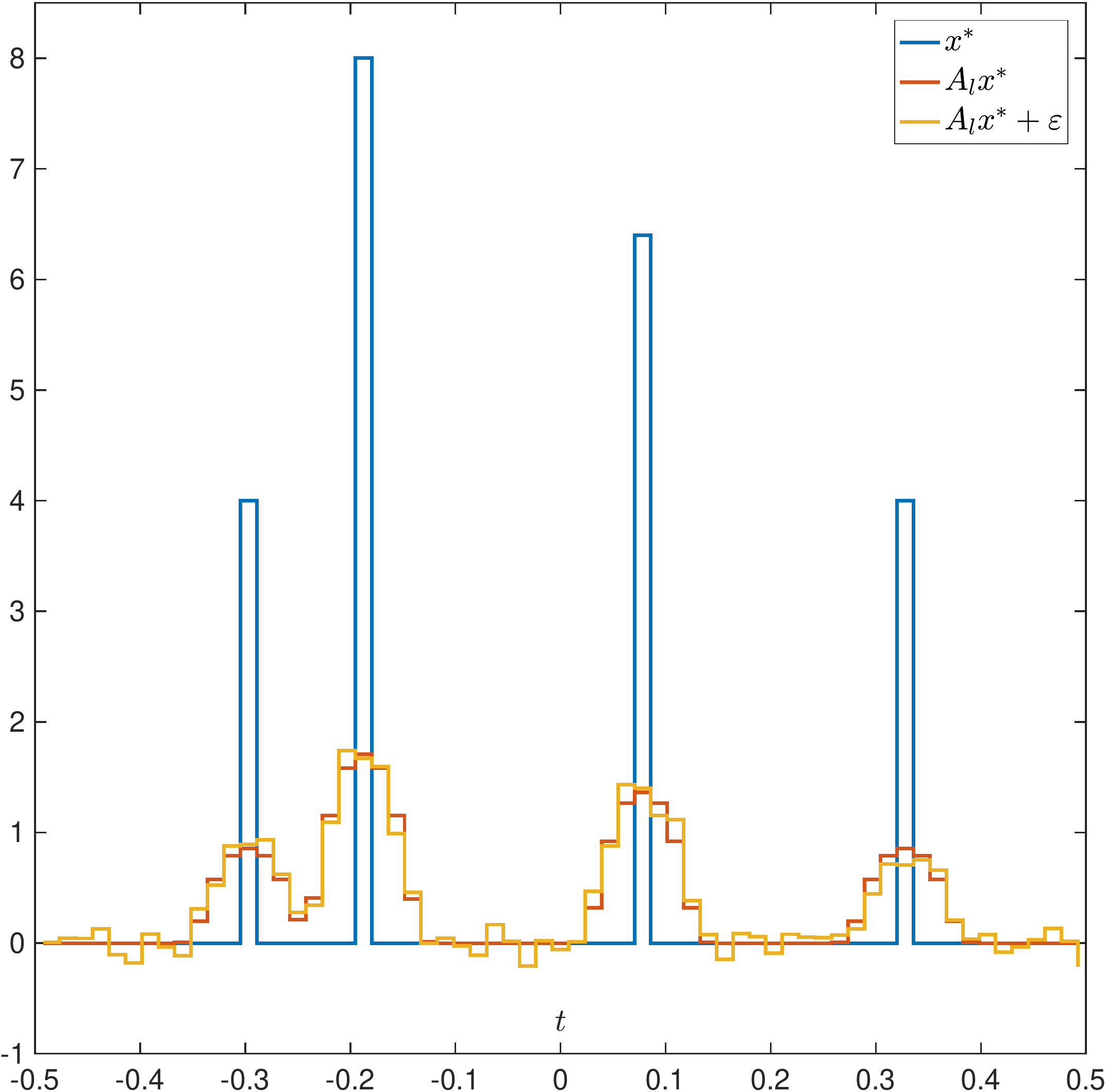}  \label{subfig:Variables}}
\caption{\subref{subfig:Kernel} The convolution kernel $k_l(t)$ for different values of $l$. \subref{subfig:Variables} True solution $x^*$, clean data $A_l x^*$ and noisy data $A_lx^* +\varepsilon$ for $m=n=64$, $l = 0.06$, $\sigma = 0.1$. \label{fig:Setting}
}
\end{figure}

As the convolution kernel $k_l$ has mass $1$ and the discretization was designed to be mass-preserving, we have $\gamma_1 = 1$ and the condition number of $A$ is given by $\cond(A) = 1/\gamma_r$, where $r = \text{rank}(A)$. Figure \ref{fig:SingularValues} shows the decay of the singular values for various parameter settings and Table \ref{tbl:CondA} lists the corresponding condition numbers: \reply{From this, we can see that the degree of ill-posedness of solving \eqref{model} measured in terms of the rate of decay of the singular values and the condition number grows very fast with increasing $m$ and $l$. It is easy to show that in the infinite dimensional setting, the rate of decay would be exponentially fast.} 

\begin{figure}[tb]
   \centering
\includegraphics[width=0.65 \textwidth]{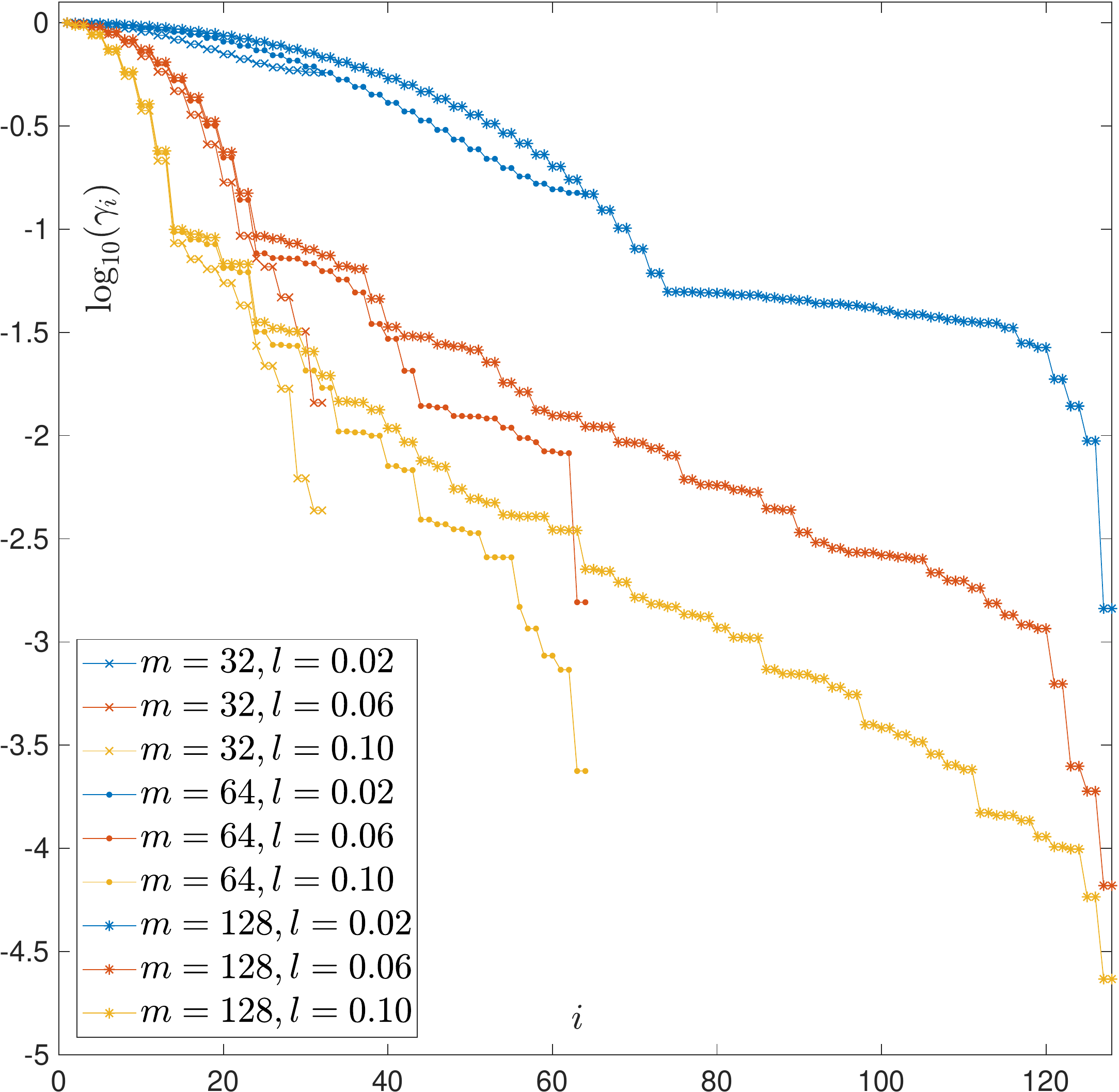}
\caption{Decay of the singular values $\gamma_i$ of $A_l$ for different choices of \reply{$m = n$} and $l$. As expected, increasing the width $l$ of the convolution kernel leads to a faster decay. For a fixed $l$, increasing $m$ corresponds to using a finer discretization and $\gamma_i$ converges to the corresponding singular value of $A_{\infty,l}$, as can be seen for the largest $\gamma_i$, e.g., for $l = 0.02$. \label{fig:SingularValues}}
\end{figure}

\begin{table}[tb]
\caption{Condition of $A_l$ computed different values of $m = n$ and $l$.}
\label{tbl:CondA}
  \centering
  \begin{tabular}{lrrrrr}
  \toprule
	&	$l = 0.02$ & $l=0.04$ & $l=0.06$ & $l=0.08$ & $l=0.1$\\
  \cmidrule{1-6}
$m=16$  & 1.27e+0 & 1.75e+0 & 2.79e+0 & 6.77e+0 & 2.31e+2\\
$m=32$  & 1.75e+0 & 6.77e+0 & 6.94e+1 & 6.88e+2 & 2.30e+2\\
$m=64$  & 6.77e+0 & 6.88e+2 & 6.42e+2 & 1.51e+3 & 4.22e+3\\
$m=128$ & 6.88e+2 & 1.51e+3 & 1.51e+4 & 4.29e+3 & 4.29e+4\\
$m=256$ & 1.70e+3 & 4.70e+4 & 1.87e+6 & 4.07e+6 & 1.79e+6\\
$m=512$ & 4.70e+4 & 1.11e+7 & 1.22e+7 & 2.12e+7 & 3.70e+7\\
\bottomrule
  \end{tabular} 
\end{table}


\paragraph{Empirical Distributions}

Using the above formulas and $m=n=64$, $l = 0.06$, $\sigma = 0.1$, we computed the empirical distributions of \reply{the $\alpha$ values selected} by the different parameter choice rules by evaluating \eqref{DPspectral}, \eqref{PSUREspectral} and \eqref{GSUREspectral} on a fine logarithmical $\alpha$-grid, i.e., $\log_{10}(\alpha_i)$ was increased linearly in between $-40$ and $40$ with a step size of $0.01$. We draw $N_\varepsilon = 10^6$ samples of $\varepsilon$. The results are displayed in Figures \ref{fig:Hist_n64l06N6} and \ref{fig:2DHist_n64l06N6}: In both figures, we use a logarithmic scaling of the empirical probabilities wherein empirical probabilities of $0$ have been set to $1/(2 N_\varepsilon)$. While this presentation complicates the comparison of the distributions as the probability mass is deformed, it facilitates the examination of small values and tails. 

First, we observe in Figure \ref{subfig:Alpha_n64l06N6} that $\alphaHatDP$ typically overestimates the optimal $\alpha^*$. However, it performs robustly and does not cause large $\ell_2$-errors as can be seen in Figure \ref{subfig:Err_n64l06N6}. For $\alphaHatPSURE$ and $\alphaHatGSURE$, the latter is not true: While being closer to $\alpha^*$ than $\alphaHatDP$ most often, and, as can be seen from the joint error histograms in Figure \ref{fig:2DHist_n64l06N6}, producing smaller $\ell_2$-errors \reply{more often (87\%/56\% of the time for PSURE/SURE)}, both distributions show outliers, i.e., occasionally, very small values of $\hat \alpha$ are estimated that cause large $\ell_2$-errors. In the case of $\alphaHatGSURE$, we even observe two clearly separated modes in the distributions. \reply{Table \ref{tbl:L2ErrStats} shows different statistics that summarize the described phenomena.} These findings motivate the theoretical examinations carried out in the following section.
 
\begin{figure}[tb]
   \centering
\subfigure[][]{\includegraphics[width= 0.49\textwidth]{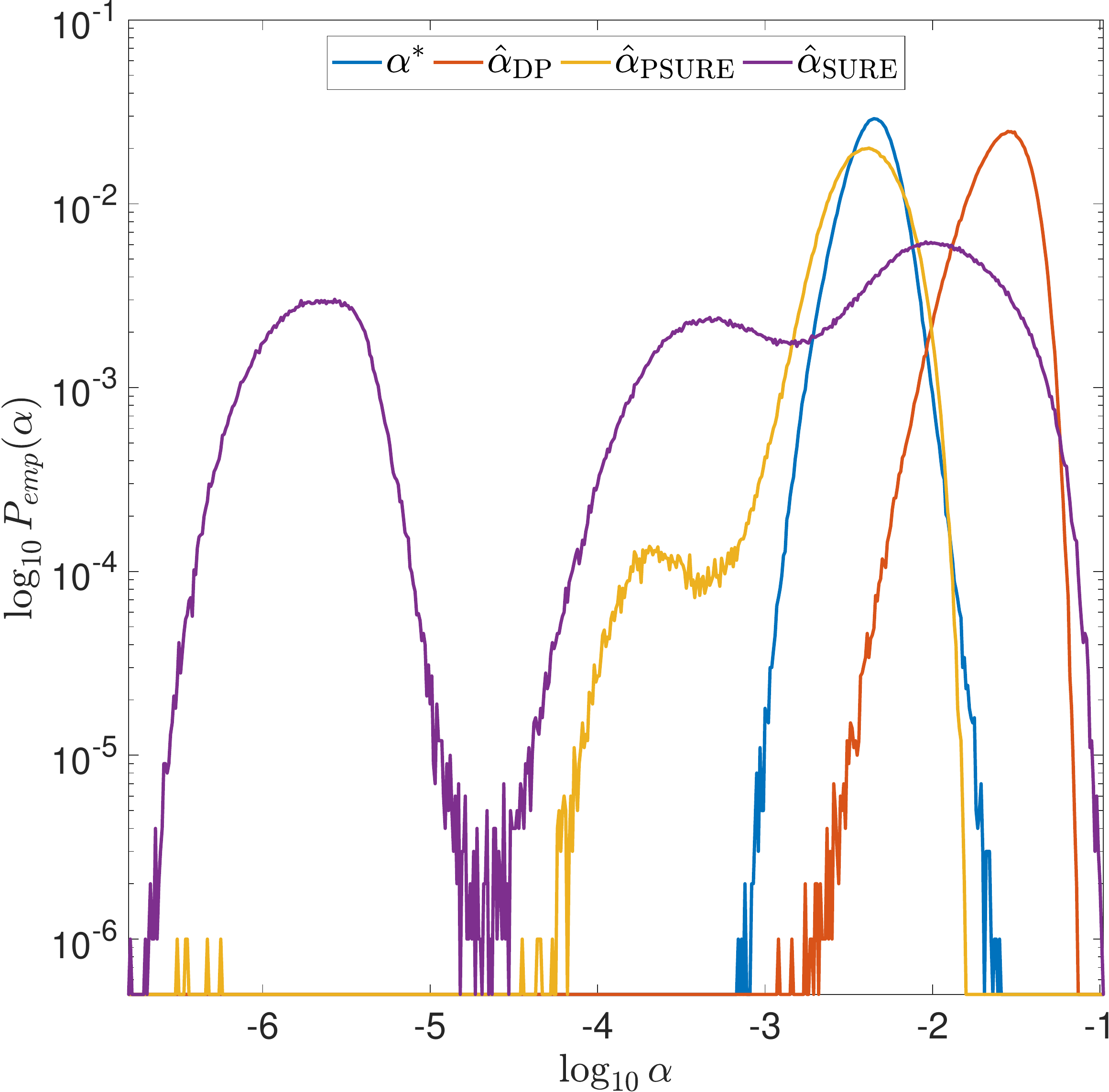}\label{subfig:Alpha_n64l06N6}}
\subfigure[][]{\includegraphics[width= 0.49\textwidth]{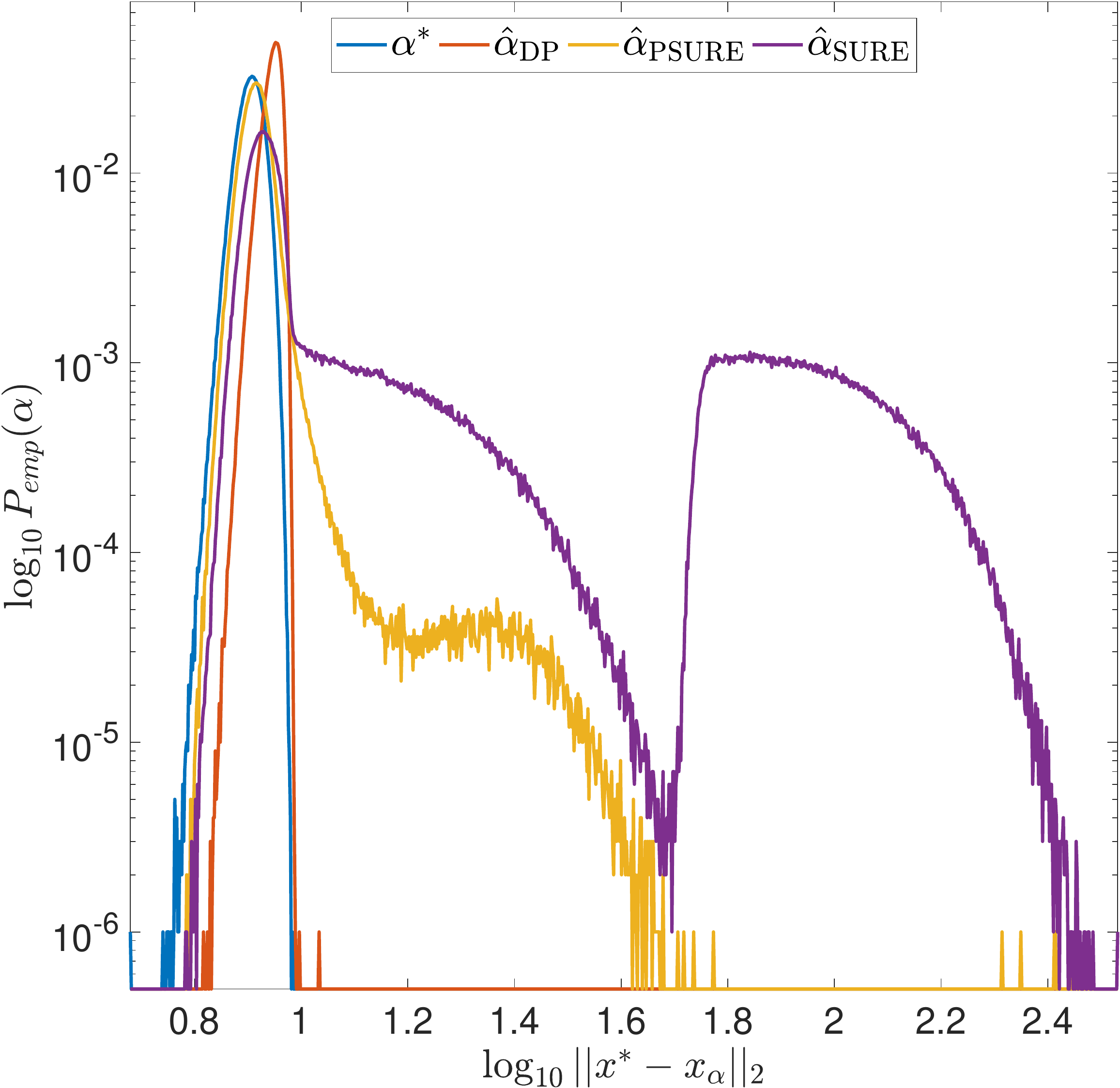}\label{subfig:Err_n64l06N6}}
\caption{Empirical probabilities of \subref{subfig:Alpha_n64l06N6} $\hat \alpha$ and \subref{subfig:Err_n64l06N6} the corresponding $\ell_2$-error for different parameter choice rules using $m=n=64$, $l = 0.06$, $\sigma = 0.1$ and $N_\varepsilon = 10^6$ samples of $\varepsilon$. \label{fig:Hist_n64l06N6}}
\end{figure}

\begin{figure}[tb]
   \centering
\subfigure[][Discrepancy principle vs PSURE]{\includegraphics[width= \textwidth]{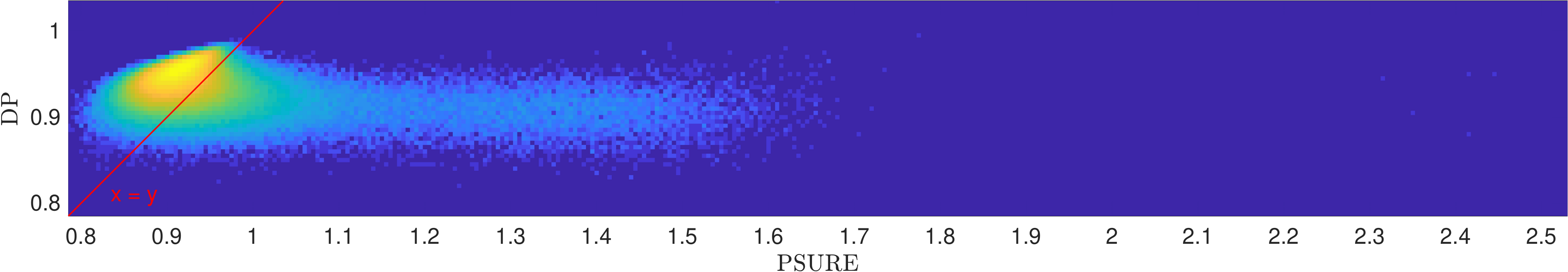}\label{subfig:DISvsPSURE2dHist}}
\subfigure[][Discrepancy principle vs SURE]{\includegraphics[width= \textwidth]{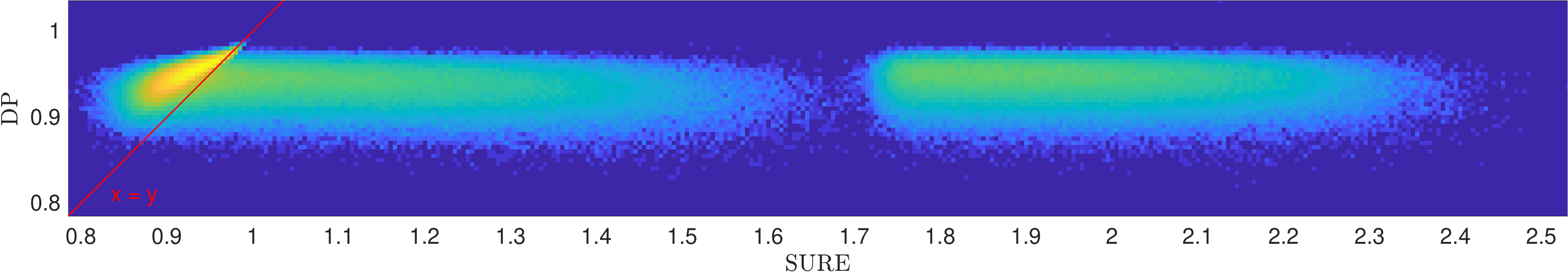}\label{subfig:DISvsGSURE2dHist}}
\caption{Joint empirical probabilities of $\log_{10}\|x^* - x_{\hat \alpha}\|_2$ using $m=n=64$, $l = 0.06$, $\sigma = 0.1$ and $N_\varepsilon = 10^6$ samples of $\varepsilon$ (the histograms in Figure \ref{subfig:Err_n64l06N6} are the marginal distributions thereof). As in Figure \ref{subfig:Err_n64l06N6}, the logarithms of the probabilities are displayed (here in form of a color-coding) to facilitate the identification of smaller modes and tails. The red line at $x=y$ divides the areas where one method performs better than the other: In \subref{subfig:DISvsPSURE2dHist}, all samples falling into the area on the right of the red line correspond to a noise realization where the discrepancy principle leads to a smaller error than PSURE. \reply{The percentage of samples for which that is true is 13\% for PSURE and 44\% for SURE.}
  \label{fig:2DHist_n64l06N6}}
\end{figure}

\begin{table}[tb]
\reply{
\caption{Statistics of the $\ell_2$-error $\|x^* - x_{\hat \alpha}\|_2$ for different parameter choice rules using $m=n=64$, $l = 0.06$, $\sigma = 0.1$ and $N_\varepsilon = 10^6$ samples of $\varepsilon$.}
\label{tbl:L2ErrStats}
  \centering
  \begin{tabular}{lrrrrr}
  \toprule
 & min & max & mean & median & std \\
  \cmidrule{1-6}
optimal  & 4.78  &  9.63  &  8.04  &  8.05  &  0.43 \\
DP  		  & 6.57  &  10.81  &  8.82  &  8.87  &  0.34 \\
PSURE  	  & 6.10  &  277.24  &  8.38  &  8.23  &  1.53 \\
SURE   & 6.08  &  339.80  &  27.71  &  8.95  &  37.26 \\
\bottomrule
  \end{tabular} 
  }
\end{table}

\clearpage

\section{Properties of the Parameter Choice Rules for Quadratic Regularization} \label{sec:Theory}

\begin{figure}[tb]
   \centering
\subfigure[][]{\includegraphics[width= 0.32\textwidth]{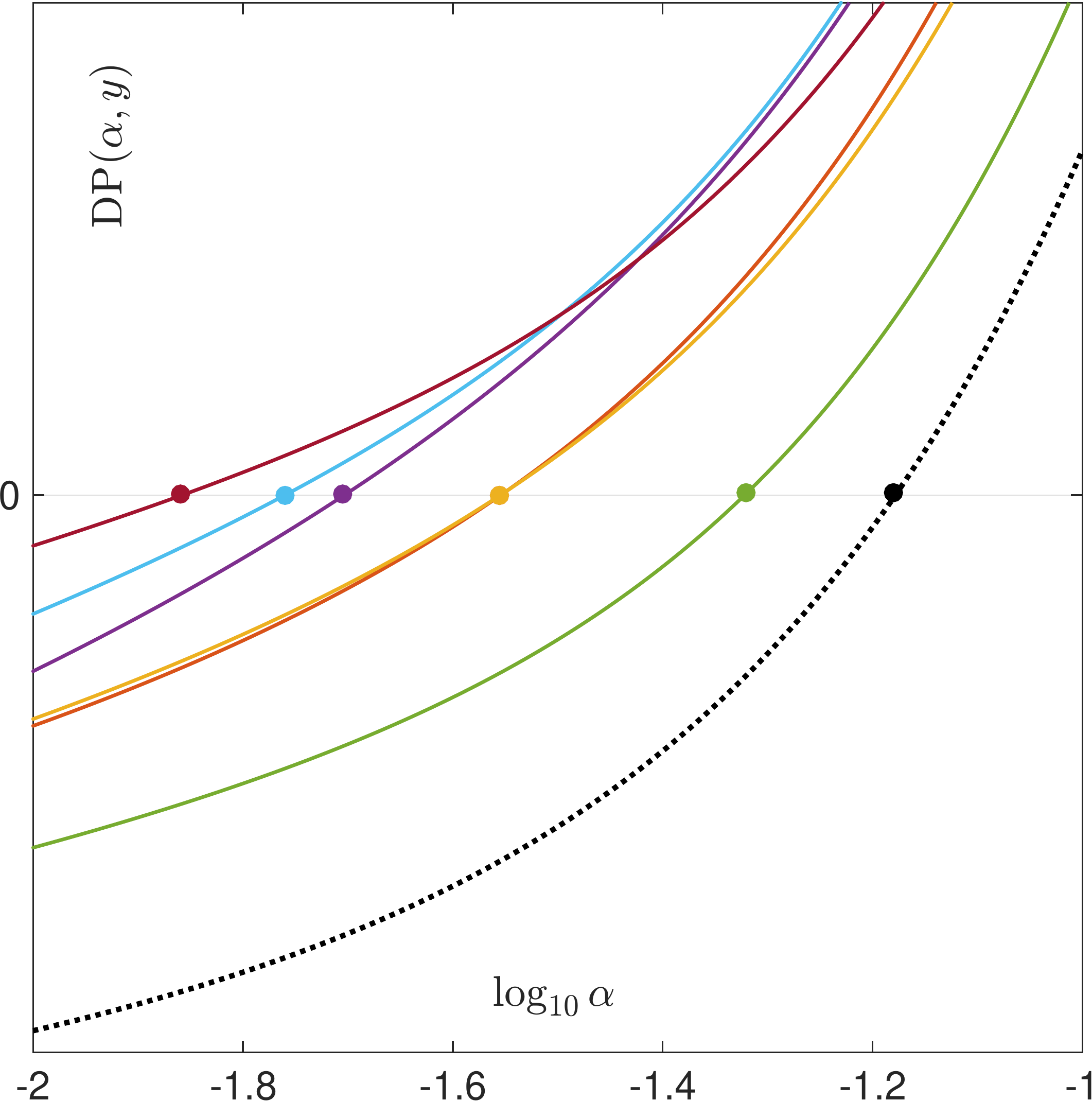}\label{subfig:L2DPrisk}}
\subfigure[][]{\includegraphics[width= 0.32\textwidth]{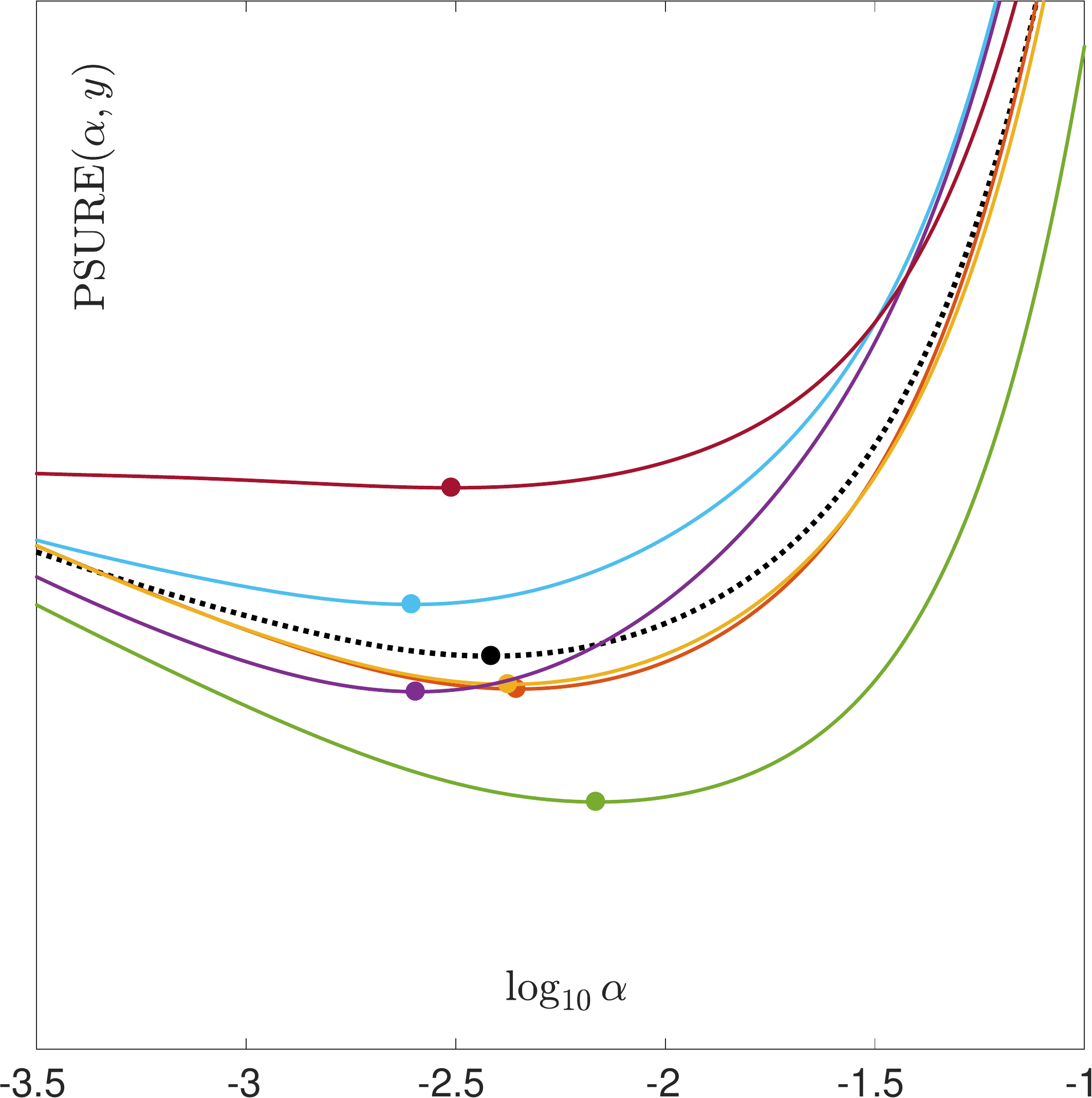}\label{subfig:L2PSURErisk}}
\subfigure[][]{\includegraphics[width= 0.32\textwidth]{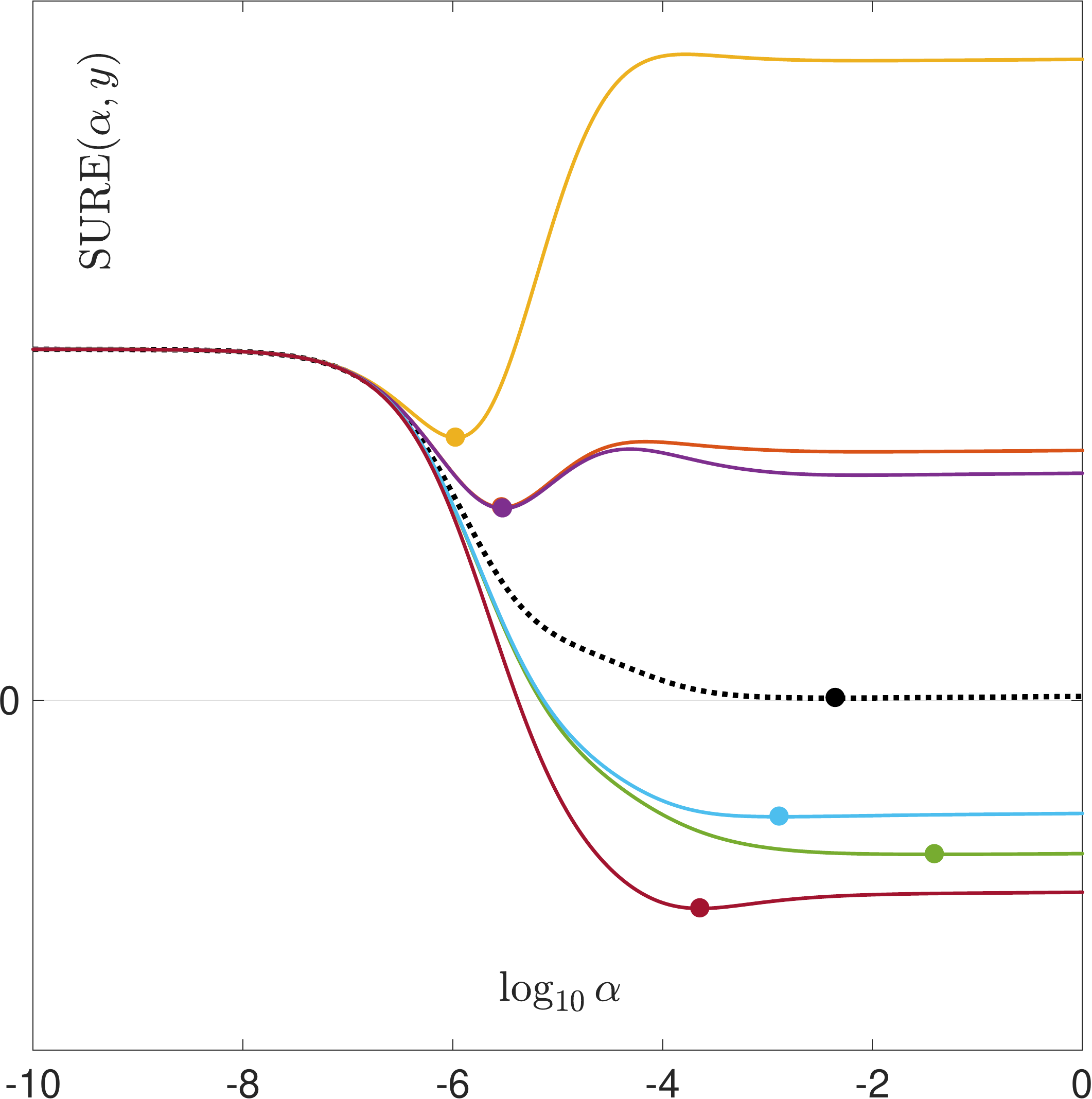}\label{subfig:L2GSURErisk}}
\caption{True risk functions (black dotted line), their estimates for six different realizations $y^k$, $k=1\ldots 6$ (solid lines), and their corresponding minima/roots (dots on the lines) in the setting described in Figure \ref{fig:Setting} using $\ell_2$-regularization: \subref{subfig:L2DPrisk} $\DP(\alpha,Ax^*)$ and $\DP(\alpha,y^k)$. \subref{subfig:L2PSURErisk} $\mspe(\alpha)$ and $\psure(\alpha,y^k)$. \subref{subfig:L2GSURErisk} $\msee(\alpha)$ and $\gsure(\alpha,y^k)$. \label{fig:L2RiskPlots}}
\end{figure}

In this section we consider the theoretical (risk) properties of PSURE, SURE and the discrepancy principle. 
\reply{To allow for a concise and accessible presentation of the main results, all proofs are shifted to Appendix \ref{sec:Proofs}. As we are investigating  random quantities, convergence rates are given in terms of the stochastic order symbols $O_{\mathbb{P}}$ and $o_{\mathbb{P}}$, which correspond to Landau's big $O$ and small $o$ notation, respectively, when convergence in probability is considered. Let us recall the definition of $O_{\mathbb{P}}$ and $o_{\mathbb{P}}$ using the formulation in \cite{vandeGeer2000}, Chapter 2.1.
\begin{Definition}[Stochastic Order Symbols]\label{def:OrderSymbols}
Let $(\Omega,\mathcal{F},\mathbb{P})$ a probability space. $Z_n:\Omega\to\mathbb{R}$, $n\in\mathbb{N},$ be a sequence of random variables, and $(r_n)_{n\in\mathbb{N}}$ be a sequence of positive numbers. We say that
\begin{align}\label{def:O}
Z_n=O_{\mathbb{P}}(r_n)\quad\text{if}\quad\lim_{T\to\infty}\limsup_{n\to\infty}\mathbb{P}\bigl(|Z_n|>Tr_n\bigr)=0.
\end{align}
We say that
\begin{align}\label{def:o}
Z_n=o_{\mathbb{P}}(r_n)\quad\text{if for all } \nu>0\quad\lim_{n\to\infty}\mathbb{P}\bigl(|Z_n|>\nu\, r_n\bigr)=0.
\end{align}
Instead of $Z_n=O_{\mathbb{P}}(r_n)$ or $Z_n=o_{\mathbb{P}}(r_n)$ we may also write $Z_n/r_n=O_{\mathbb{P}}(1)$ or $Z_n/r_n=o_{\mathbb{P}}(1)$, respectively.
\end{Definition}
}

\begin{Assumption}\label{Assumption}
For the sake of simplicity we only consider $m=n$ in this first analysis. Furthermore, we assume 
\begin{equation}\label{CondGamma}
  1=\gamma_{1,m}\geq\ldots\geq\gamma_{m,m}\geq0~.
\end{equation}
Note that all assumptions are fulfilled in the numerical example we described in the previous section.
\end{Assumption}
We mention that we consider here a rather moderate size of the noise, which remains bounded in variances as $m \rightarrow \infty$. A scaling corresponding to white noise in the infinite dimensional limit is rather  $\sigma^2 \sim m$ and an inspection of the estimates below shows that the risk estimate is potentially far from the expected values in such cases additionally.


\subsection{PSURE-Risk}

We start with an investigation of the \reply{PSURE} risk estimate. Based on \eqref{PSUREspectral} and Stein's result, the representation for the risk is given as
\begin{align}\label{e2}
\mspe(\alpha)&=\mathbb{E}[\psure(\alpha,y)]\notag\\
&=\sum_{i=1}^m \frac{\alpha^2}{(\gamma_i^2 + \alpha)^2} \mathbb{E}[y_i^2] - \sigma^2 m + 2 \sigma^2 \sum_{i=1}^m \frac{\gamma_i^2}{\gamma_i^2 + \alpha}\notag\\
&=\sum_{i=1}^m \frac{\alpha^2}{(\gamma_i^2 + \alpha)^2} (\gamma_i^2\cdot(x_i^{*})^2+\sigma^2)- \sigma^2 m + 2 \sigma^2 \sum_{i=1}^m \frac{\gamma_i^2}{\gamma_i^2 + \alpha}.
\end{align}

Figure \ref{subfig:L2PSURErisk} illustrates the typical shape of $\mspe(\alpha)$ and \reply{PSURE} estimates thereof. Following \reply{\cite{Li87,k1994} who considered the case $A=I_m$ and} \cite{Xie2012,Ghoreishi2014}, who investigated the performance of Stein's unbiased risk estimate in the different context of hierarchical modeling, we show that, with the definition of the loss $\mathcal{L}$ by

$$\mathcal{L}(\alpha):=\frac{1}{m}\|Ax^*-A\hat x_{\alpha}(y)\|_2^2,$$

$1/m \, \psure(\alpha,y)$ is close to $\mathcal{L}$ for large $m.$ Note that PSURE is an unbiased estimate of the expectation of $\mathcal{L}.$

\begin{Theorem}\label{PSUREl}
If Assumption \ref{Assumption} holds, then  \reply{ we have for any sequence of vectors  $(x^*_m)_{m\in\mathbb{N}}$, $x^*_m\in\R^m,$ such  that $\|x^*_m\|_2^2=O(m)$} as $m\rightarrow\infty$
\begin{equation*}
  \sup_{\alpha\in[0,\infty)}\Bigl|\frac{1}{m}\psurem(\alpha,y)-\mathcal{L}_m(\alpha)\Bigr|=O_\mathbb{P}\left(\frac{1}{\sqrt{m}}\right).
\end{equation*}
\end{Theorem}
\medskip
\reply{
\begin{Remark}
The result of Theorem \ref{PSUREl} guarantees stochastic boundedness of the sequence
\begin{equation*}
\biggl(\sqrt{m}\sup_{\alpha\in[0,\infty)}\Bigl|\frac{1}{m}\psurem(\alpha,y)-\mathcal{L}_m(\alpha)\Bigr|\biggr)_{m\in\mathbb{N}}.
\end{equation*}
It does not entail the existence of a proper weak limit of \,$\psure$, which would require stronger assumptions on the sequences $(x^*_m)_{m\in\mathbb{N}}$ and $\big((\gamma_{i,m})_{i=1}^m\big)_{m\in\mathbb{N}}.$
\end{Remark}}
The latter result can be used to show that, in an asymptotic sense, if the loss $\mathcal{L}$
 is considered, the estimator $\alphaHatPSURE$ does not have a larger risk than any other choice of regularization parameter. This statement is made precise in the following corollary.

\begin{Corollary}\label{CorPSUREl} Let $(\delta_m)_{m\in\mathbb{N}}$ be  a sequence of positive real numbers  such that $1/\delta_m=o(\sqrt{m})$.
 \reply{Under the assumptions of Theorem \ref{PSUREl} }the following holds true for any sequence of positive real numbers $(\alpha_m)_{m\in\mathbb{N}}$:
\begin{align*}
  \mathbb{P}\big(\mathcal{L}(\hat\alpha_{\psure,m})\geq \mathcal{L}_m(\alpha_m)+\delta_m\big)\rightarrow0.
\end{align*}
\end{Corollary}

We finally mention that our estimates are rather conservative, in particular with respect to the quantity $Sl_3(\alpha)$ in the proof of Theorem \ref{PSUREl}, since we do not assume particular smoothness of $x^*$. With an additional source condition, i.e., certain decay speed of the $x_i^*$, it is possible to derive improved rates, which are however beyond the scope of our paper. \reply{We refer to \cite{cg2014} and \cite{lw2017} for recent results in that direction, where optimality of $x_{\alphaHatPSURE}$ with respect to the risk $\msee$ under source conditions for spectral cut-off and more general, filter based methods are shown, respectively.}
 We  turn our attention to the convergence of the risk estimate as $m \rightarrow \infty$ as well as the convergence of the estimated regularization parameters.


	
\begin{Theorem}\label{PSURER}
	If Assumption \ref{Assumption} holds, then \reply{ we have for any sequence of vectors  $(x^*_m)_{m\in\mathbb{N}}$, $x^*_m\in\R^m,$ such  that $\|x^*_m\|_2^2=O(m)$} as $m\rightarrow\infty$
\begin{equation*}
    \sup_{\alpha\in[0,\infty)}\Bigl|\frac{1}{m}\bigl(\psurem(\alpha,y)-\mspem(\alpha,y)\bigr)\Bigr|=O_\mathbb{P}\Bigl(\frac{1}{\sqrt{m}}\Bigr)
 \end{equation*}
	and
\begin{equation}\label{ExpSupPSURE}
    \mathbb{E}\biggl(\sup_{\alpha\in[0,\infty)}\Bigl|\frac{1}{m}\bigl(\psurem(\alpha,y)-\mspem(\alpha,y)\bigr)\Bigr|\biggr)^2=O\Bigl(\frac{1}{m}\Bigr).
\end{equation}
\end{Theorem}
\reply{
\begin{Remark}
It follows from Theorem \ref{PSURER} and Definition \ref{def:OrderSymbols} that
\begin{equation*}
\mathbb{P}\bigg(\sup_{\alpha\in[0,\infty)}\Bigl|\frac{1}{m}\bigl(\psurem(\alpha,y)-\mspem(\alpha,y)\bigr)\Bigr|>\nu \frac{\log(m)}{\sqrt{m}}\bigg)=0\quad\text{for all}\quad \nu>0,
\end{equation*}
whereas $\mspem(\alpha,y)$ is bounded away from zero by the assumptions of Theorem \ref{Thm:alphaconv}. Therefore, asymptotically, minimizing $\mspem$ is the same as minimizing $\psure.$
\end{Remark}
}
\medskip

In order to understand the behaviour of the estimated regularization parameters we start with some bounds on $\alphaPSURE$, which recover a standard property of deterministic Tikhonov-type regularization methods, namely that $\frac{\sigma^2}\alpha$ does not diverge for suitable parameter choices (cf. \cite{engl1996regularization}).

\begin{Lemma} \label{alphasurelemma}
A regularization parameter $\hat\alpha^*_{\psure,m}$ obtained from $\mspem$ satisfies
\begin{equation*}
\frac{\sigma^2}{\max_{1\leq i\leq m} \vert x_{i,m}^*  \vert^2} \leq \hat\alpha^*_{\psure,m}  \leq \max\bigg\{1,8 \sigma^2 \frac{\sum \gamma_{i,m}^4}{\sum \gamma_{i,m}^4 (x_{i,m}^*)^2}  \bigg\}
\end{equation*}
\end{Lemma}

From a straight-forward estimate of the derivative of $\mspem$ on sets where $\alpha$ is bounded away from zero, together with the Arzela-Ascoli theorem  we obtain the following result:
%
\begin{Proposition} \label{fmproposition}
The sequence of functions $f_m: \alpha \mapsto \frac{1}m \mspem(\alpha)$ is equicontinuous on sets $[C_1,C_2]$ with $0 < C_1 < C_2$ and hence has a uniformly convergent subsequence $f_{m_k}$ with continuous limit function $f$.   
\end{Proposition}

In order to obtain convergence of minimizers it suffices to be able to choose uniform constants $C_1$ and $C_2$, which is possible if the bounds in Lemma \ref{alphasurelemma} are uniform:
\begin{Theorem}\label{Thm:alphaconv}
Let $\max_{i=1}^m \abs{x_{i,m}^*}$ be uniformly bounded in $m$ and $\frac{1}m \sum_{i=1}^m \gamma_{i,m}^4 (x_{i,m}^*)^2$ be uniformly bounded away from zero. Then there exists a subsequence $\hat \alpha_{\mspe,m_k}$ that converges to a minimizer of the asymptotic risk $f$. Moreover  $\hat \alpha_{\psure,m_k}$  converges to to a minimizer of the asymptotic risk $f$ in probability.
\end{Theorem}

\subsection{Discrepancy Principle}

We now turn our attention to the discrepancy principle, which we can formulate in a similar setting as the \reply{PSURE} approach above. With a slight abuse of notation, in analogy to the other methods, we denote the expectation of $ \DP(\alpha,y)$ by $\EDP(\alpha)$ and define $\alphaDP$ as the solution of the  equation
\begin{equation*}
	  \EDP(\alpha)=\sum_{i=1}^m\frac{\alpha^2}{(\gamma_i^2+\alpha)^2}\mathbb{E}[y_i^2]-m\sigma^2=0.
	\end{equation*}
Figure \ref{subfig:L2DPrisk} illustrates the typical shape of $\EDP(\alpha)$ and its DP estimates. Observing that
\begin{equation*}
\DP(\alpha,y)-\EDP(\alpha)= \psure(\alpha,y)-\mspe(\beta)
\end{equation*}
we immediately obtain the following result:
\begin{Theorem}\label{DR} 
If Assumption \ref{Assumption} holds, \reply{ we have for any sequence of vectors  $(x^*_m)_{m\in\mathbb{N}}$, $x^*_m\in\R^m,$ such  that $\|x^*_m\|_2^2=O(m)$} 
\begin{equation*}
\sup_{\alpha\in[0,\infty)}\Bigl|\frac{1}{m}\bigl(\DPm(\alpha,y)-\EDPm(\alpha)\bigr)\Bigr|=O_\mathbb{P}\Bigl(\frac{1}{\sqrt{m}}\Bigr)
\end{equation*}
and
\begin{equation*}
\mathbb{E}\biggl(\sup_{\alpha\in[0,\infty)}\Bigl|\frac{1}{m}\bigl(\DPm(\alpha,y)-\EDPm(\alpha)\bigr)\Bigr|\biggr)^2=O\Bigl(\frac{1}{m}\Bigr).
\end{equation*}
\end{Theorem}

\subsection{SURE-Risk}

 Now we consider the SURE-risk estimation procedure. Figure \ref{subfig:L2GSURErisk} illustrates the typical shape of $\msee(\alpha)$ and SURE estimates thereof. 
{Based on \eqref{GSUREspectral}, if $\gamma_m>0$ for all $m$, the risk can be written as
 \begin{align*}
	  \msee(\alpha,y) = \sum_{i=1}^m \left( \frac{1}{\gamma_i} - \frac{\gamma_i}{(\gamma_i^2 + \alpha)} \right)^2 \bigl(\gamma_i^2(x_i^*)^2+\sigma^2\bigr) - \sigma^2 \sum_{i=1}^m \frac{1}{\gamma_i^2}  + 2 \sigma^2 \sum_{i=1}^m \frac{1}{\gamma_i^2 + \alpha}.
	\end{align*}
}\\
For the PSURE criterion we showed in Theorem \ref{PSUREl}  that $\psure(\alpha,y)$ is close to the loss $\mathcal{L}$ in an asymptotic sense with the standard $\sqrt{m}$-rate of convergence. An analogous result can be shown for SURE and the associated loss  $\tilde{\mathcal{L}}(\alpha):=c_m\|\Pi (x^*-\hat x_{\alpha})\|_2^2$ but with different associated rates of convergence $c_m$, dependent on the singular values.

\begin{Theorem}\label{GSUREl}
Let Assumption \ref{Assumption}  be satisfied and in addition to \eqref{CondGamma}, let $\gamma_{m,m}>0$ for all $m$ and $m=n=r$. Then 
\reply{ we have for any sequence of vectors  $(x^*_m)_{m\in\mathbb{N}}$, $x^*_m\in\R^m,$ such  that $\max_{i=1}^m \abs{x_{i,m}^*}$ is uniformly bounded} 
as $m\rightarrow\infty,$ 
$$
  \sup_{\alpha\in[0,\infty)}\Bigl|c_m\gsurem(\alpha,y)-\tilde{\mathcal{L}}_m(\alpha)\Bigr|=
  O_\mathbb{P}\left(d_m\right),
$$
where 
$$
  c_m:=\left(\sum_{i=1}^m\frac{1}{\gamma_{i,m}^2}\right)^{-1}\quad\text{and}\qquad d_m:=c_m\cdot\sqrt{\sum_{i=1}^m\frac{1}{\gamma_{i,m}^4}}\,.
$$
\end{Theorem}
\reply{In the same manner as for PSURE, we may use the latter convergence result to show that, in an asymptotic sense, if the loss $\tilde{\mathcal{L}}$ is considered, the estimator $\hat{\alpha}_{\gsure}$ does not have a larger risk than any other choice of regularization parameter.  We stress again that this optimality property depends on the loss considered, as it is the case in Corollary \ref{CorPSUREl}.
\begin{Corollary}\label{CorGSUREl}
Let $(\delta_m)_{m\in\mathbb{N}}$ be  a sequence of positive reals   such that $d_m=o(\delta_m)$.
If the assumptions of Theorem \ref{GSUREl} hold, we have  for any sequence of positive real numbers $(\alpha_m)_{m\in\mathbb{N}}$:
 \begin{align*}
 \mathbb{P}\big(\tilde{\mathcal{L}}_m(\hat\alpha_{\gsure,m}\big) \geq \tilde{\mathcal{L}}_m(\alpha_m)+\delta_m)\rightarrow0.
 \end{align*}	
\end{Corollary}	}
\reply{Note that $1/\sqrt{m}\leq d_m\leq1$, depending on the behaviour of the singular values. If $\inf_m d_m>0$, $O_{\mathbb{P}}(d_m)=O_{\mathbb{P}}(1)$ in Theorem \ref{GSUREl} and only sequences $\delta_m$ such that  $\inf_m \delta_m>0$ are permissible in Corollary \ref{CorGSUREl}. }

We can now proceed to an estimate between $\gsure$ and $\msee$ similar to the ones for the \reply{PSURE} risk, however we observe a main difference due to the appearance of the condition number of the forward matrix $A$:
\begin{Theorem}\label{GSURER}
Let $A_{m}\in\mathbb{R}^{m\times m}$ be a full rank matrix. In addition to Assumption \ref{Assumption}, let $\gamma_{m,m}>0$ for all $m$ and $\gamma_{m,m}\rightarrow 0$. Then, 
\reply{ we have for any sequence of vectors  $(x^*_m)_{m\in\mathbb{N}}$, $x^*_m\in\R^m,$ such  that $\|x^*_m\|_2^2=O(m)$}  as $m\rightarrow\infty,$ 

\begin{equation*}
  \sup_{\alpha\in[0,\infty)}\Bigl|\frac{1}{m \, \cond(A_m)^2}\bigr(\gsurem(\alpha,y)-\mseem(\alpha)\bigr)\Bigr|=O_\mathbb{P}\Bigl(\frac{1}{\sqrt{m}}\Bigr)
\end{equation*}
{ and 
\begin{equation} \label{ExpSupGSURE}
  \mathbb{E}\biggl(\sup_{\alpha\in[0,\infty)}\Bigl|\frac{1}{m \, \cond(A_m)^2}\bigr(\gsurem(\alpha,y)-\mseem(\alpha)\bigr)\Bigr|\biggr)^2=O_\mathbb{P}\Bigl(\frac{1}{m}\Bigr).
\end{equation}
}
\end{Theorem}

We finally note that in the best case the convergence of SURE is slower than that of PSURE. However, since for ill-posed problems the condition number of $A$ will grow with $m$ the typical case is rather divergence of $\frac{\cond(A)^2}{\sqrt{m}}$, hence the empirical estimates of the regularization parameters might have a large variation, which will be confirmed by the numerical results below.

\section{Numerical Studies for Quadratic Regularization} \label{sec:NumStudiesL2}

\subsection{Setup}

As in the illustrative example in Section \ref{subsec:SinSysQuadReg}, we computed the empirical distributions of the different parameter choice rules for the same scenario (cf. Section \ref{subsec:Setting}) for each combination of $m=n=16, 32, 64, 128, 256, 512, 1024, 2048$, $l = 0.01, 0.02, 0.03, 0.04, 0.06, 0.08, 0,1$ and $\sigma = 0.1$. For $m = 16,\ldots,512$, $N_\varepsilon = 10^6$ and for $m=1024, 2048$, $N_\varepsilon = 10^5$ noise realizations were sampled. The computation was, again, based on a logarithmical $\alpha$-grid, i.e., $\log_{10} \alpha$ is increased linearly in between -40 and 40 with a step size of $0.01$. In addition to the distributions of $\alpha$, the expressions
\begin{equation} \label{supExpNum}
\sup_{\alpha}\, \Bigl|\psure(\alpha,y)-\mspe(\alpha,y)\Bigr|, \quad \text{and} \quad  \sup_{\alpha}\, \Bigl|\gsure(\alpha,y)-\msee(\alpha,y)\Bigr|
\end{equation}
were computed over the $\alpha$-grid. As in some cases, the supremum is obtained in the limit $\alpha \rightarrow \infty$, and hence, on the boundary of our computational grid, we also evaluated \eqref{supExpNum} for $\alpha = \infty$ in these cases. 
  
\subsection{Illustration of Theorems}

We first illustrate Theorems \ref{PSURER} and \ref{GSURER} by computing \eqref{ExpSupPSURE} and \eqref{ExpSupGSURE} based on our samples. The results are plotted in Figure \ref{fig:IlluTheoL2} and show that the asymptotic rates hold. For SURE, the comparison between Figures \ref{subfig:TheoIlluL2GSUREnoCond} and \ref{subfig:TheoIlluL2GSURE} also shows that the dependence on $\cond(A)$ is crucial. 

\begin{figure}[tb]
   \centering
\subfigure[][PSURE]{\includegraphics[width= 0.32\textwidth]{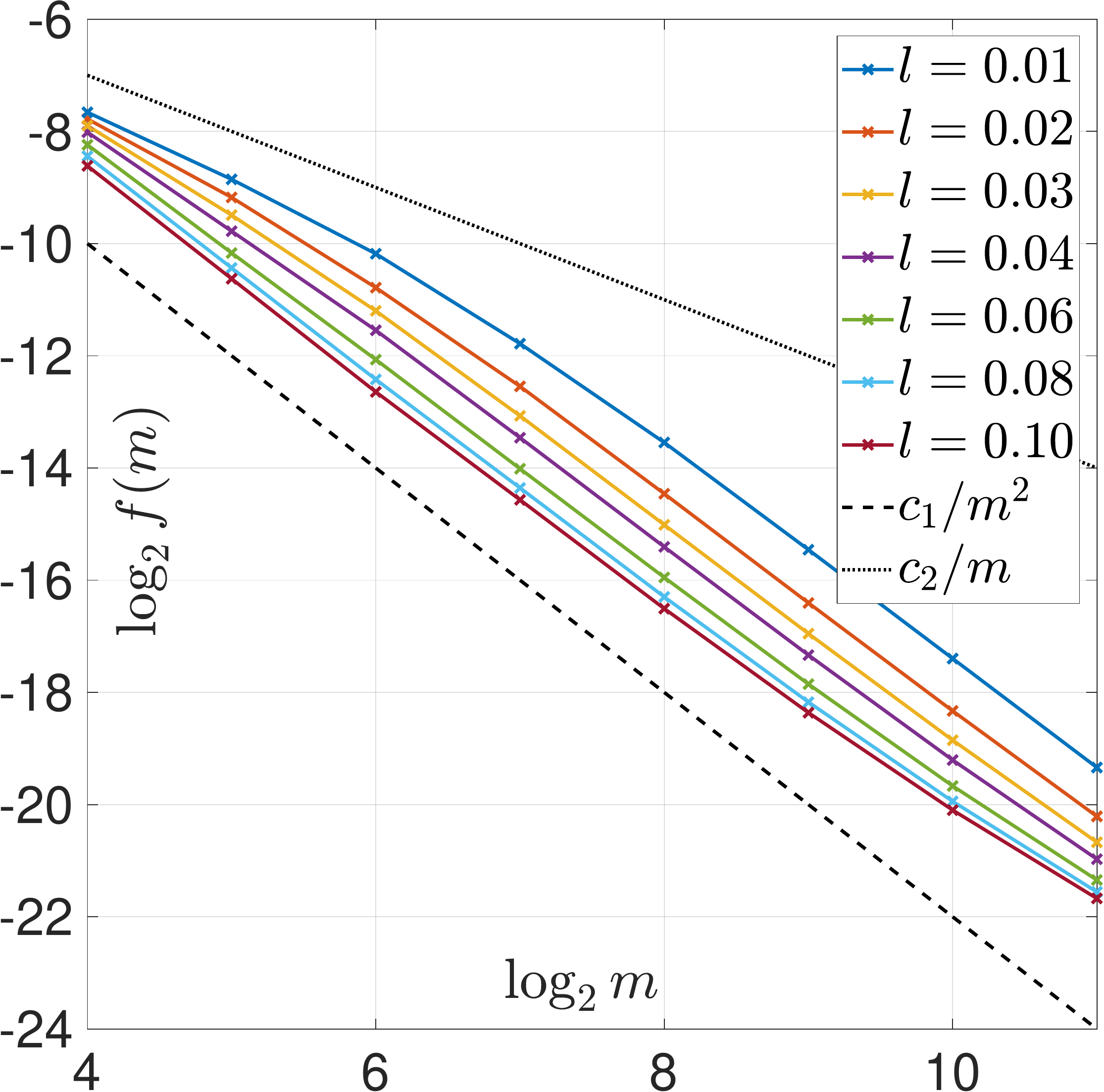}\label{subfig:TheoIlluL2PSURE}}
\subfigure[][SURE, without $\cond(A)$]{\includegraphics[width= 0.32\textwidth]{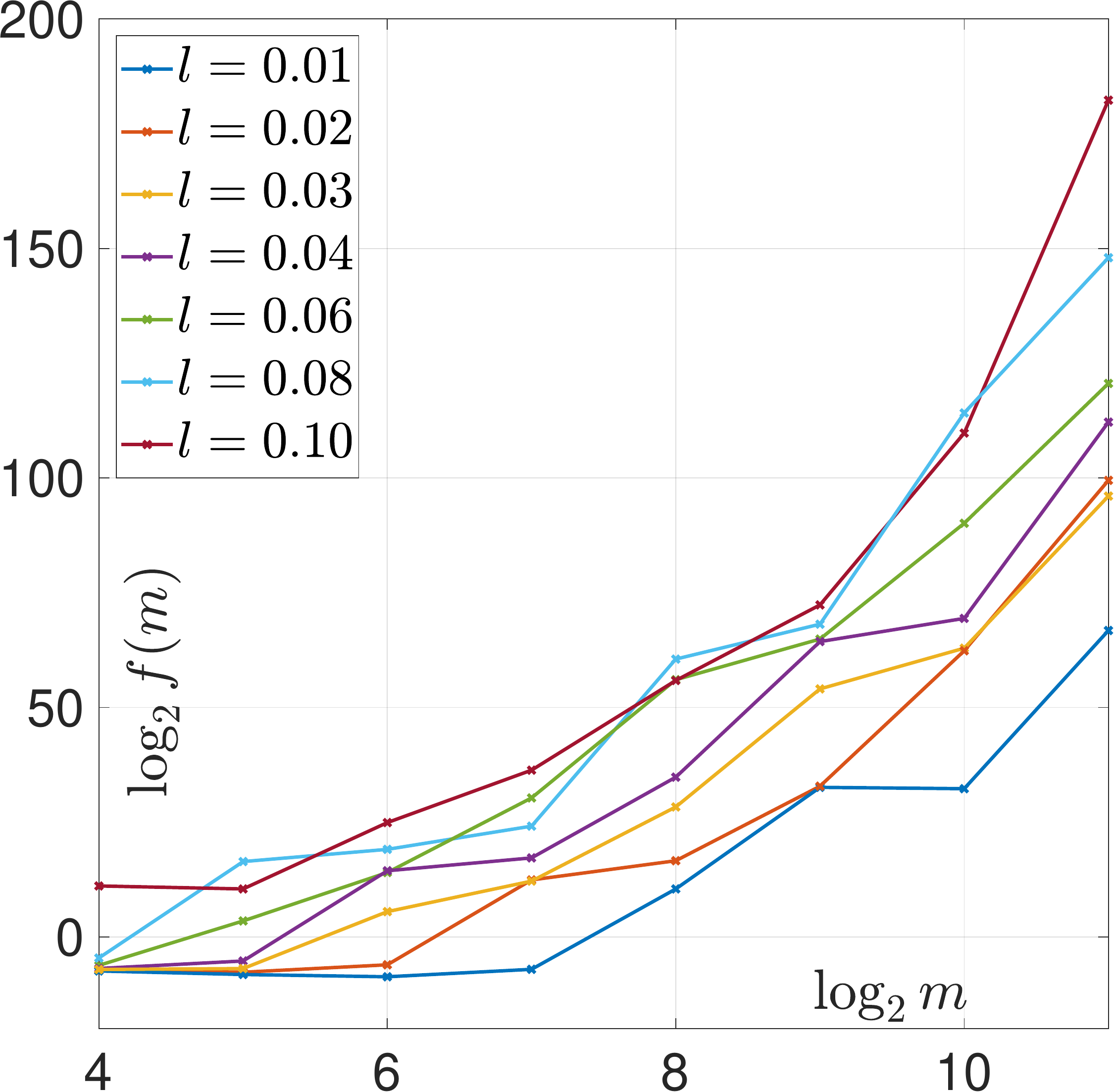}\label{subfig:TheoIlluL2GSUREnoCond}}
\subfigure[][SURE, with $\cond(A)$]{\includegraphics[width= 0.32\textwidth]{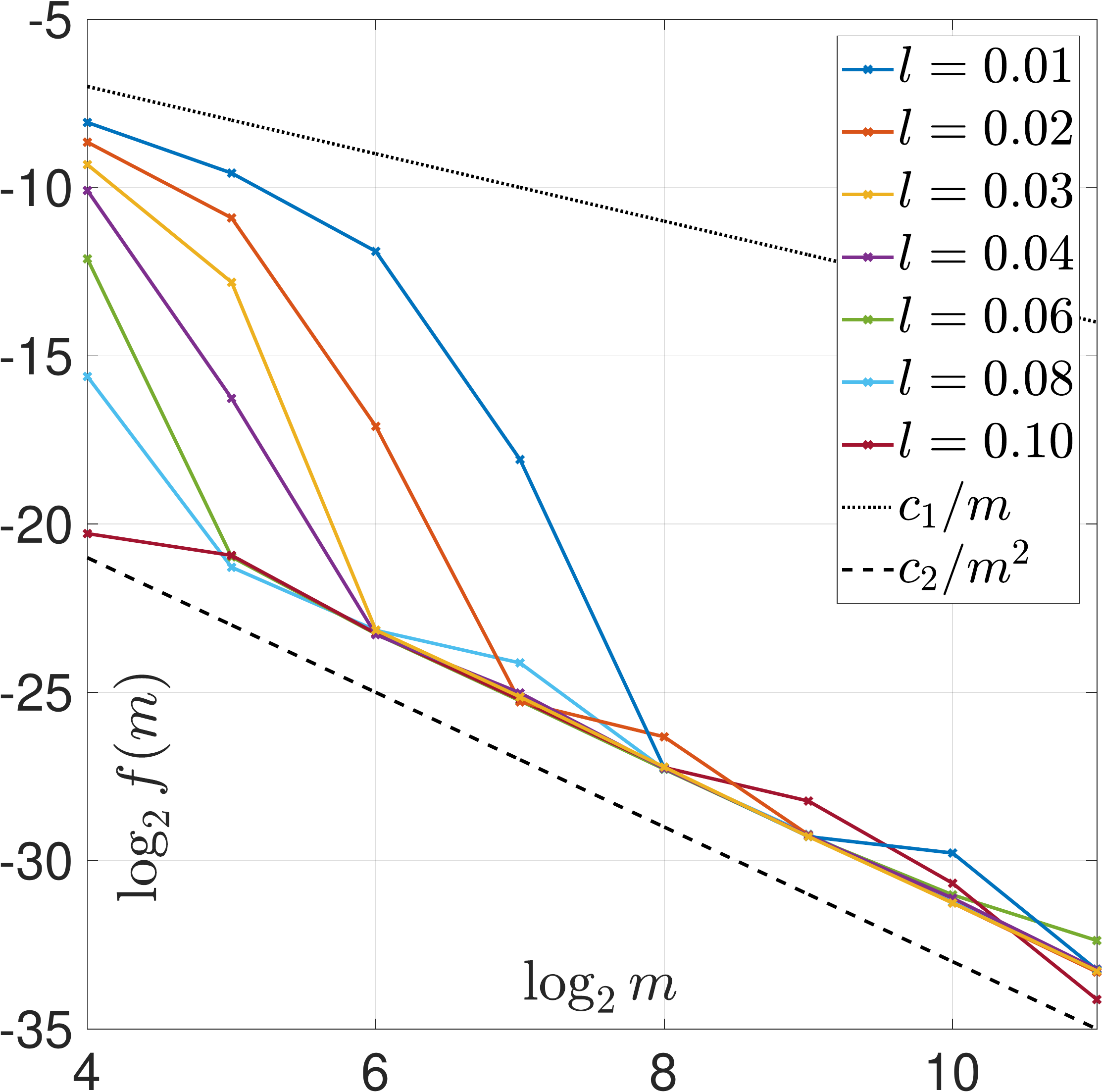}\label{subfig:TheoIlluL2GSURE}}
\caption{Illustration of Theorems \ref{PSURER} and \ref{GSURER} for $\ell_2$-regularization: The left hand side of \eqref{ExpSupPSURE}/\eqref{ExpSupGSURE} was estimated by the sample mean and plotted vs. $m$. For \eqref{ExpSupGSURE}, the normalization with $\cond(A)$ was omitted in \subref{subfig:TheoIlluL2GSUREnoCond} and included in \subref{subfig:TheoIlluL2GSURE}. The black dotted lines were added to compare the order of convergence.} \label{fig:IlluTheoL2}
\end{figure}

\subsection{Dependence on the Ill-Posedness}

We then demonstrate how the empirical distributions of $\hat \alpha$ and the corresponding $\ell_2$-error, $\|x^*-x_{\hat \alpha}\|_2^2$, such as those plotted in Figure \ref{fig:Hist_n64l06N6}, depend on the ill-posedness of the inverse problem. 

\paragraph{Dependence on $m$} 
In Figures \ref{fig:L2AlphaHist_nVarl06N6} and \ref{fig:L2ErrHist_nVarl06N6}, $m$ is increased while the width of the convolution kernel is kept fix. The impact of this on the singular value spectrum is illustrated in Figure \ref{fig:SingularValues}. Most notably, smaller singular values are added and the condition of $A$ increases (cf. Table \ref{tbl:CondA}). Figures \ref{subfig:L2Alpha_nVarOpt} and \ref{subfig:L2Err_nVarOpt} suggest that the distribution of the optimal $\alpha^*$ is Gaussian and converges to a limit for increasing $m$. The distribution of the corresponding $\ell_2$-error looks Gaussian as well and seems to concentrate while shifting to larger mean values. For the discrepancy principle, Figures \ref{subfig:L2Alpha_nVarDP} and \ref{subfig:L2Err_nVarDP} show that the distribution of $\alphaHatDP$ widens for increasing $m$, and the distribution of the corresponding $\ell_2$-error develops a tail while shifting to larger mean values. Figures \ref{subfig:L2Alpha_nVarPSURE} and \ref{subfig:L2Err_nVarPSURE} show that the distribution of $\alphaHatPSURE$ seems to converge to a limit for increasing $m$. The distribution of the corresponding $\ell_2$-error also develops a tail while shifting to larger mean values. For SURE, Figures \ref{subfig:L2Alpha_nVarGSURE} and \ref{subfig:L2Err_nVarGSURE} reveal that increasing $m$ leads to erratic, multimodal distributions: Compared to the other $\alpha$-distributions, the distribution of $\alphaHatGSURE$ includes a significant amount of very small values, and the corresponding $\ell_2$-error distributions range over very large values.

\begin{figure}[tb]
   \centering
\subfigure[][$\alpha^*$]{\includegraphics[width= 0.49\textwidth]{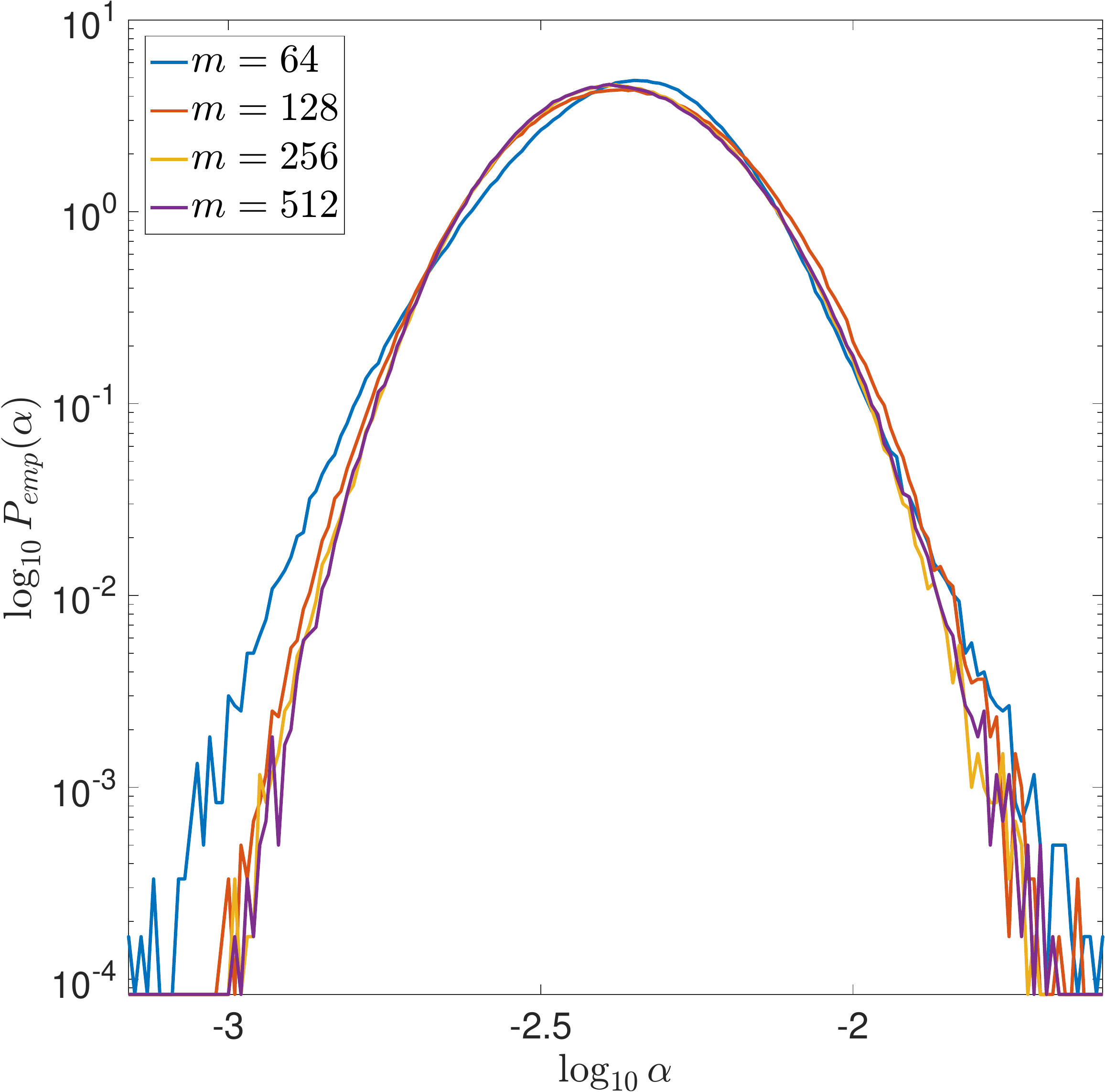}\label{subfig:L2Alpha_nVarOpt}}
\subfigure[][$\alphaHatDP$]{\includegraphics[width= 0.49\textwidth]{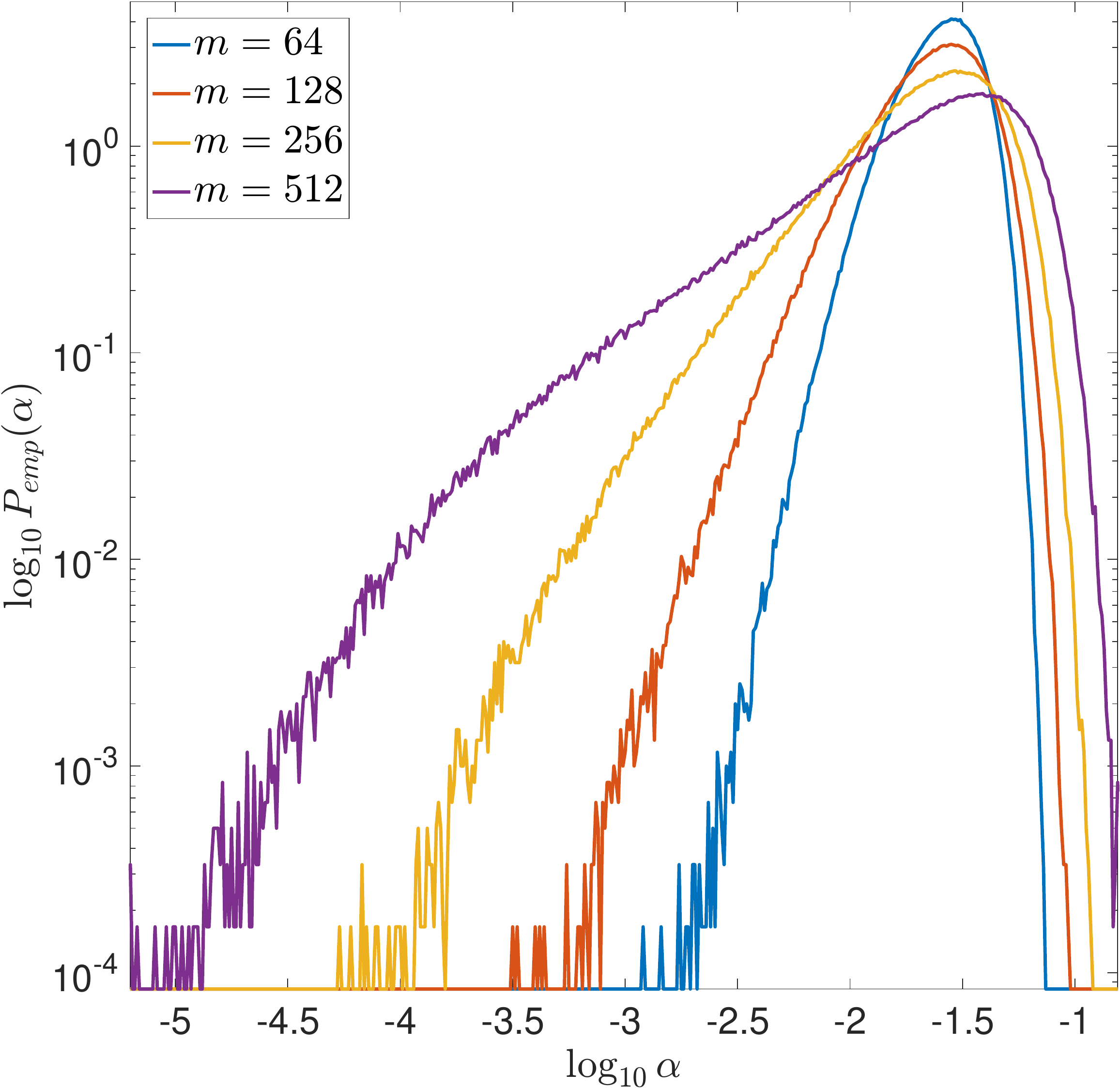}\label{subfig:L2Alpha_nVarDP}}\\
\subfigure[][$\alphaHatPSURE$]{\includegraphics[width= 0.49\textwidth]{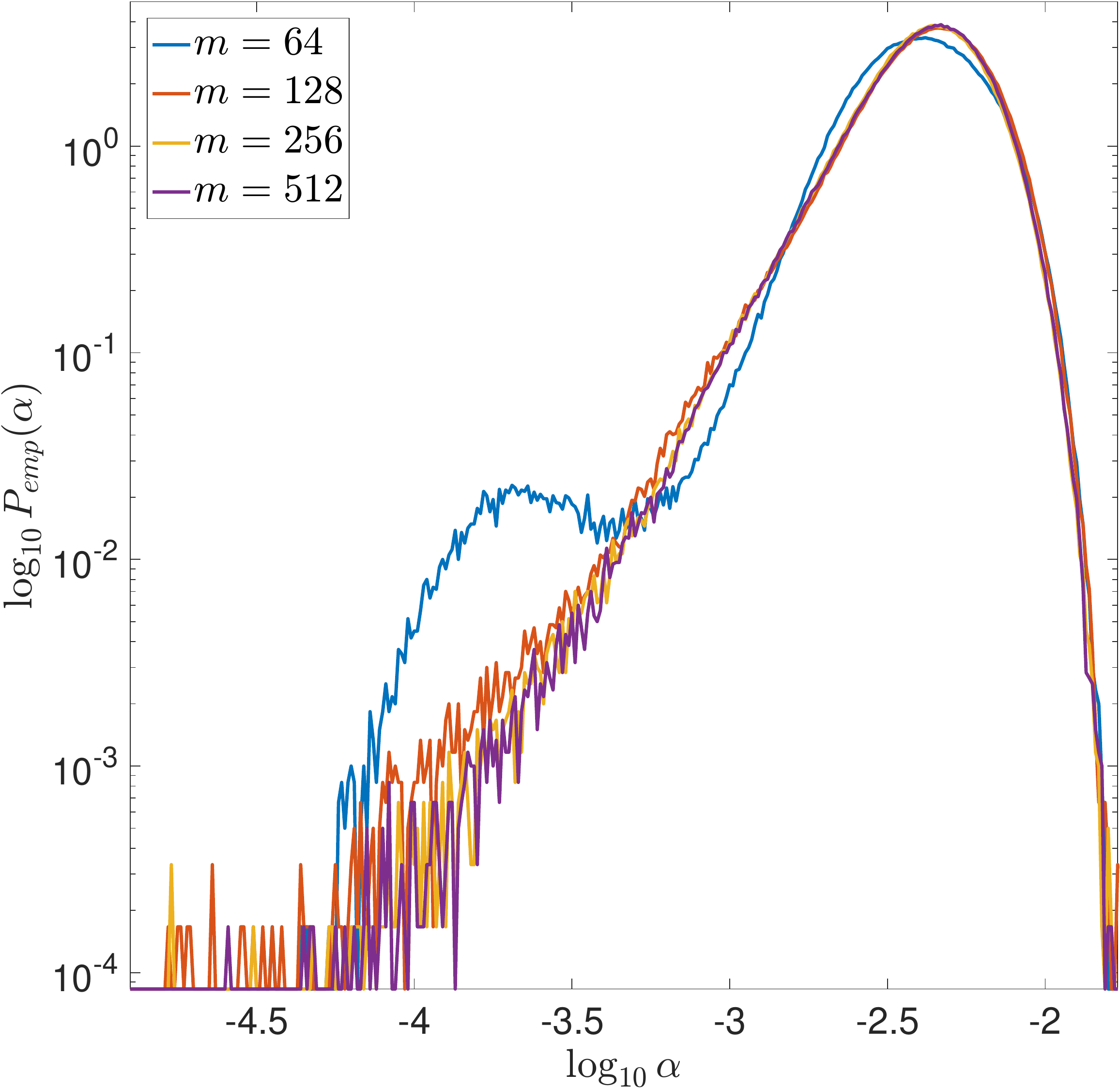}\label{subfig:L2Alpha_nVarPSURE}}
\subfigure[][$\alphaHatGSURE$]{\includegraphics[width= 0.49\textwidth]{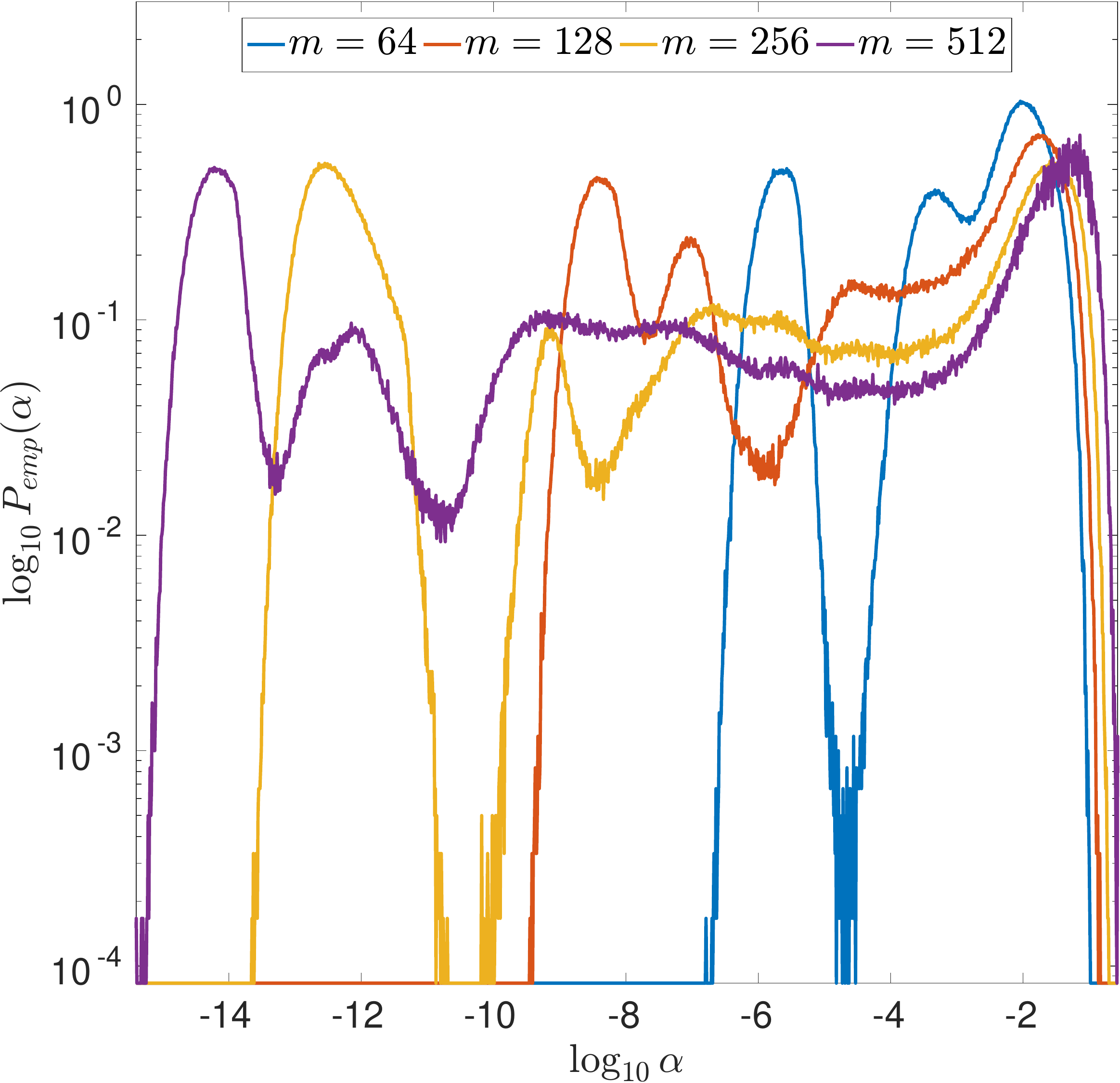}\label{subfig:L2Alpha_nVarGSURE}}
\caption{Empirical probabilities of $\alpha$ for $\ell_2$-regularization and different parameter choice rules for $l = 0.06$ and varying $m$. \label{fig:L2AlphaHist_nVarl06N6}}
\end{figure}

\begin{figure}[tb]
   \centering
\subfigure[][optimal]{\includegraphics[width= 0.49\textwidth]{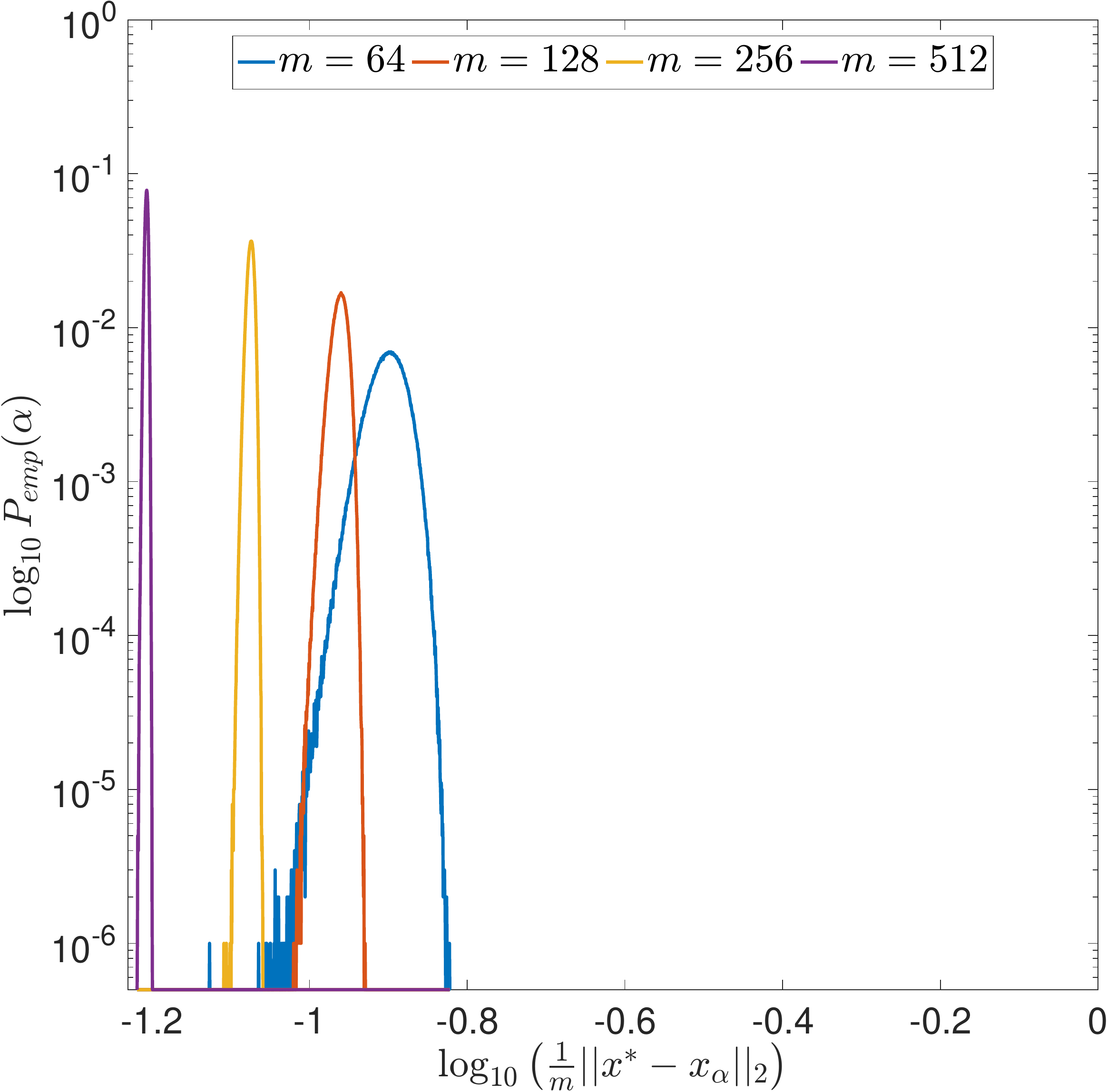}\label{subfig:L2Err_nVarOpt}}
\subfigure[][discrepancy principle]{\includegraphics[width= 0.49\textwidth]{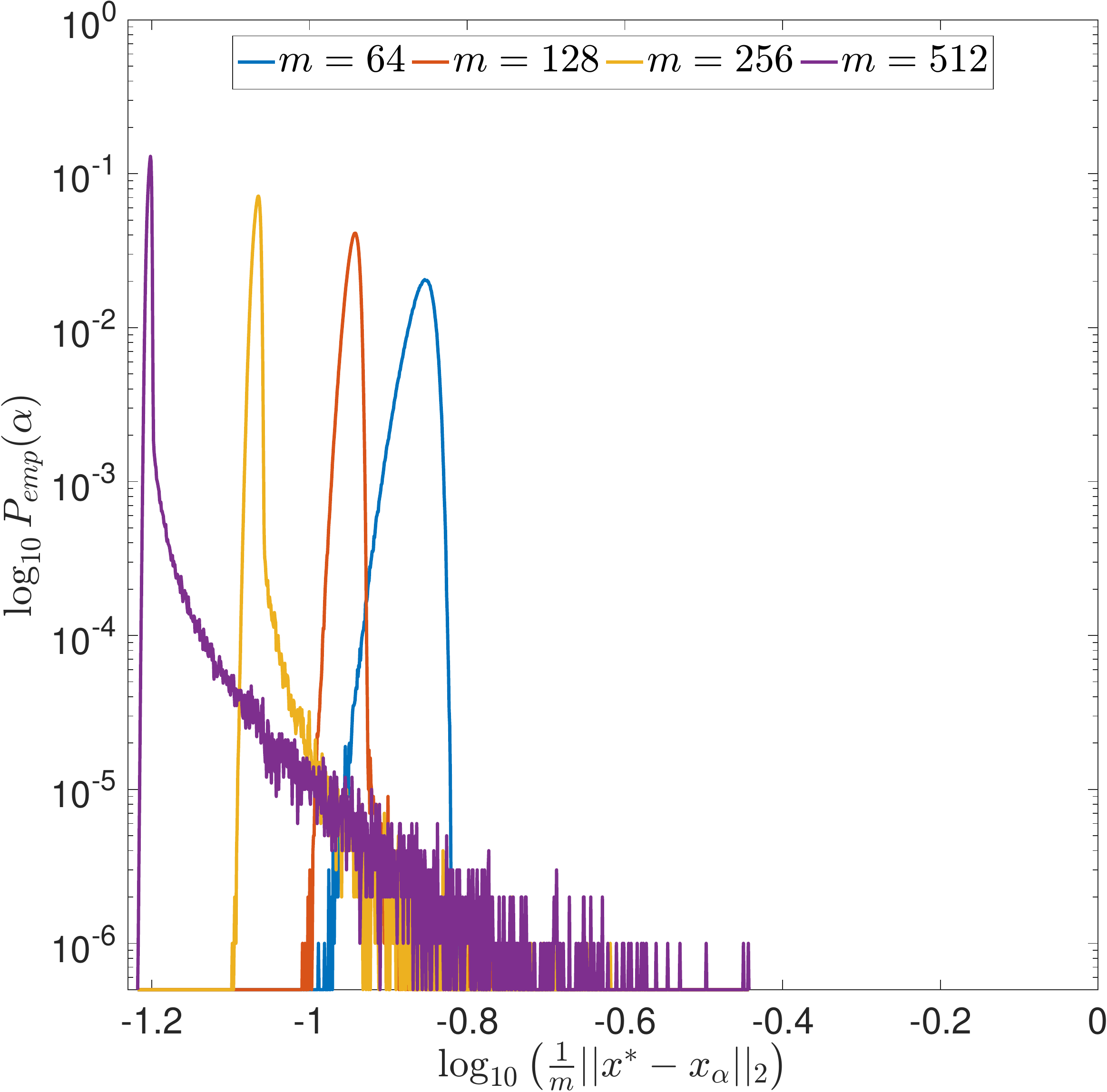}\label{subfig:L2Err_nVarDP}}\\
\subfigure[][PSURE]{\includegraphics[width= 0.49\textwidth]{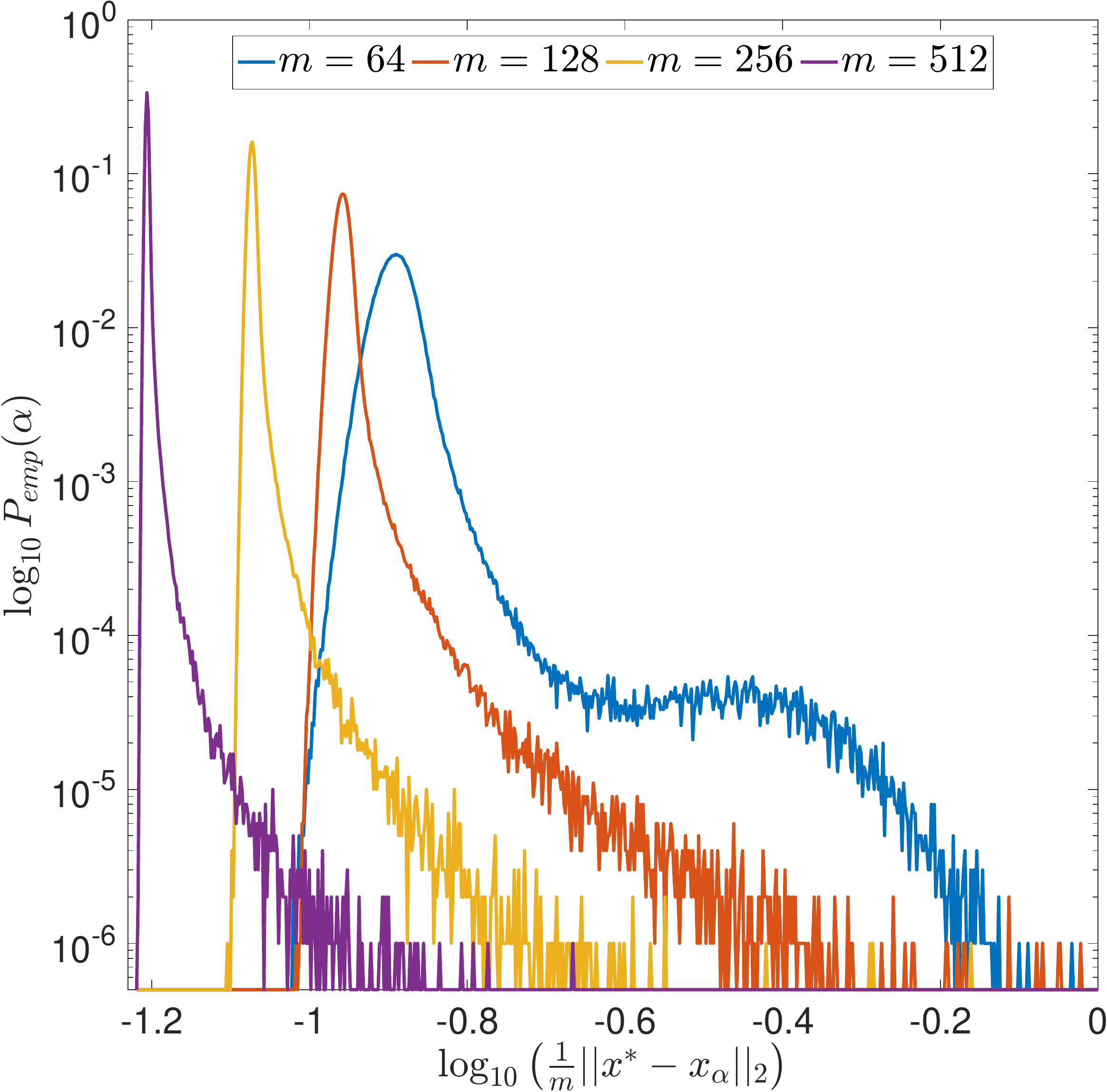}\label{subfig:L2Err_nVarPSURE}}
\subfigure[][SURE]{\includegraphics[width= 0.49\textwidth]{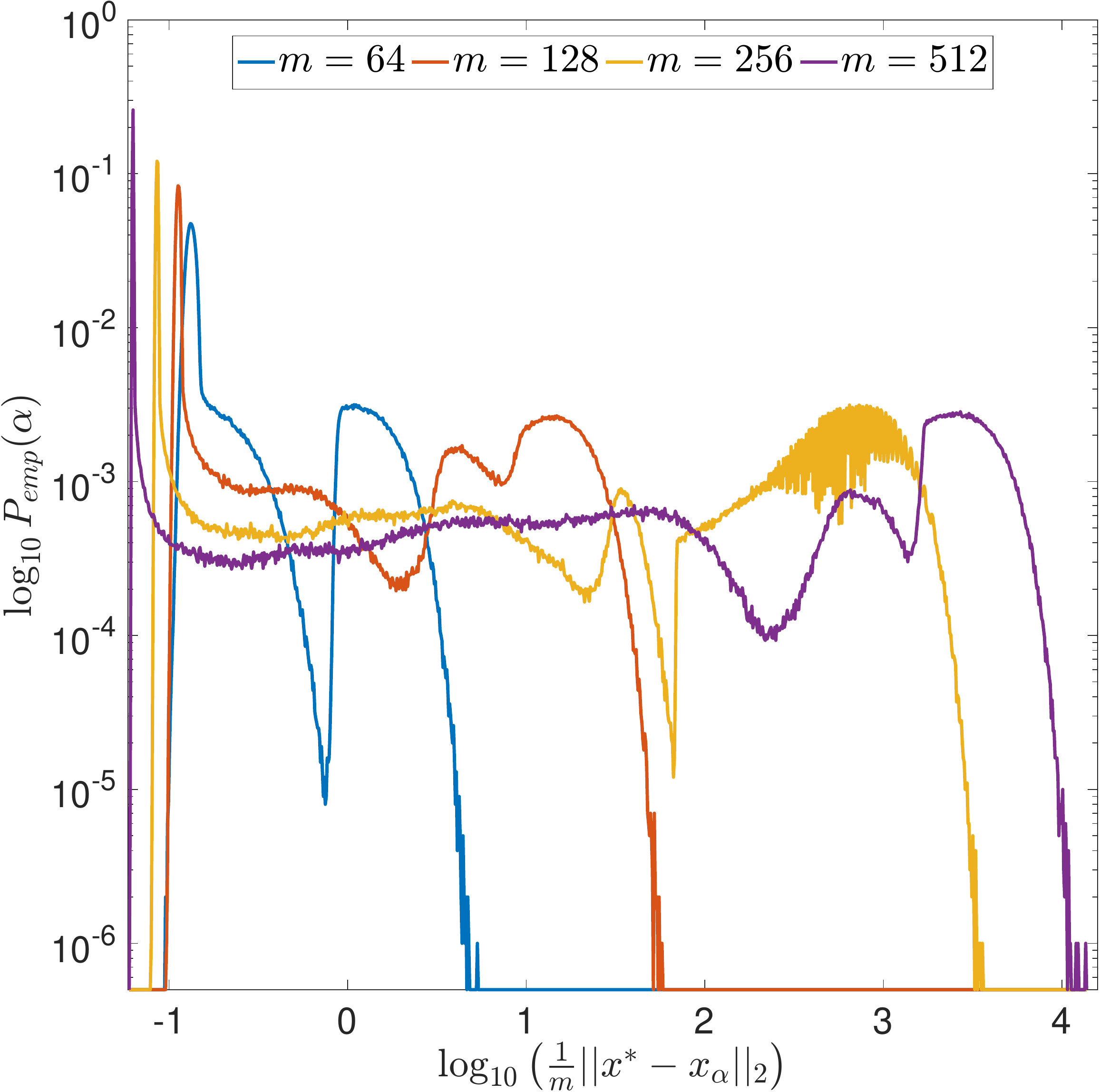}\label{subfig:L2Err_nVarGSURE}}
\caption{Empirical probabilities of $\log_{10}\left(\tfrac{1}{m} \| x^* - x_{\alpha} \|_2^2\right)$ for $\ell_2$-regularization and different parameter choice rules for $l = 0.06$ and varying $m$. \label{fig:L2ErrHist_nVarl06N6}}
\end{figure}

\paragraph{Dependence on $l$} 
In Figures \ref{fig:L2AlphaHist_n64Varl06N6} and \ref{fig:L2ErrHist_n64Varl06N6}, the width of the convolution kernel, $l$, is increased while $m=64$ is kept fix (cf. Figure \ref{fig:SingularValues} and Table \ref{tbl:CondA}). It is worth noticing that as $l=0.02$ corresponds to a very well-posed problem, the optimal $\alpha^*$ is often extremely small or even $0$, as can be seen from Figure \ref{subfig:L2Alpha_lVarOpt}. The general tendencies are similar to those observed when increasing $m$. For SURE, Figures \ref{subfig:L2Alpha_lVarGSURE} and \ref{subfig:L2Err_lVarGSURE} illustrate how the multiple modes of the distributions slowly evolve and shift to smaller vales of $\alpha$ (and larger corresponding $\ell_2$-errors).

\begin{figure}[tb]
   \centering
\subfigure[][$\alpha^*$]{\includegraphics[width= 0.49\textwidth]{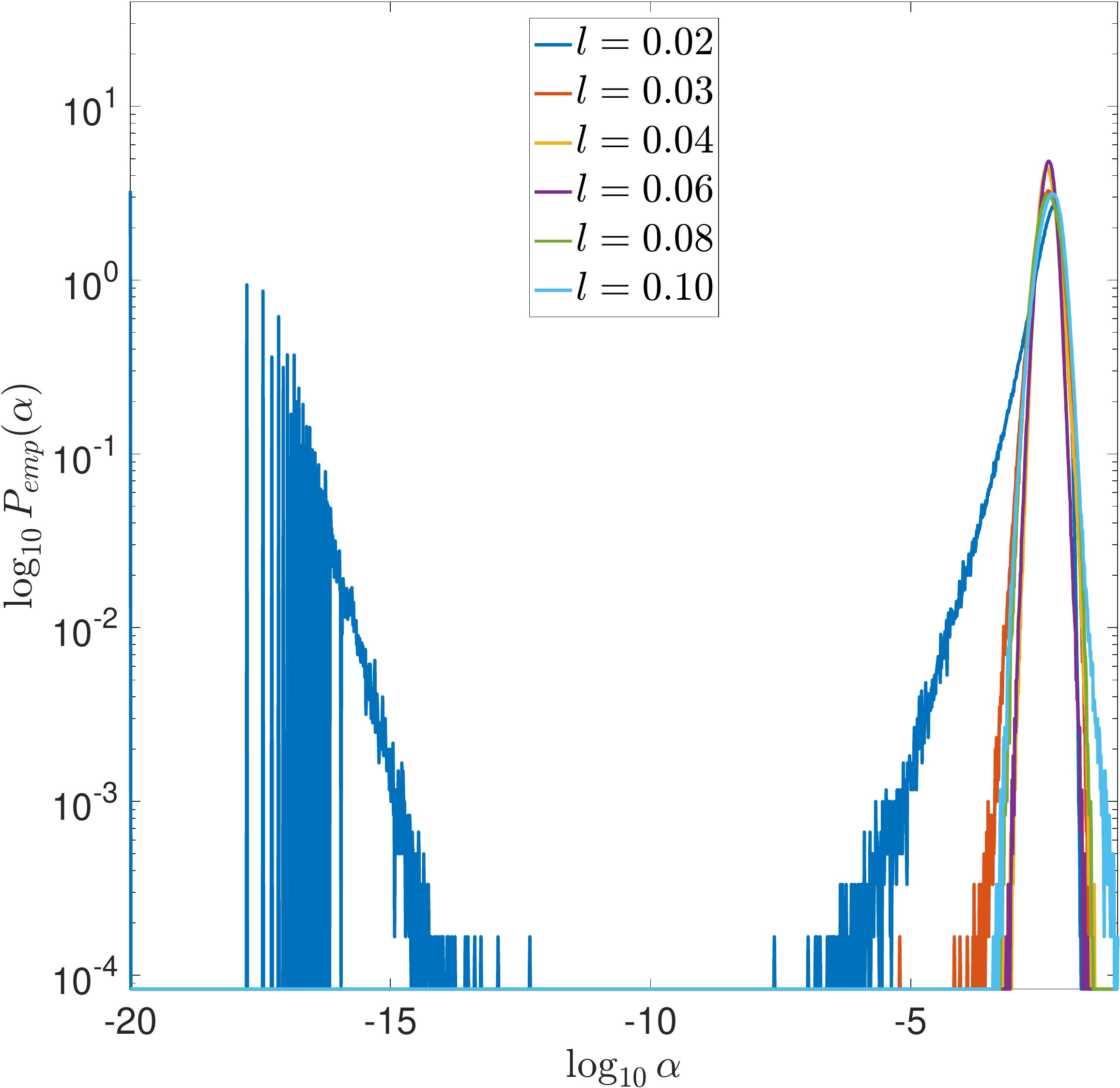}\label{subfig:L2Alpha_lVarOpt}}
\subfigure[][$\alphaHatDP$]{\includegraphics[width= 0.49\textwidth]{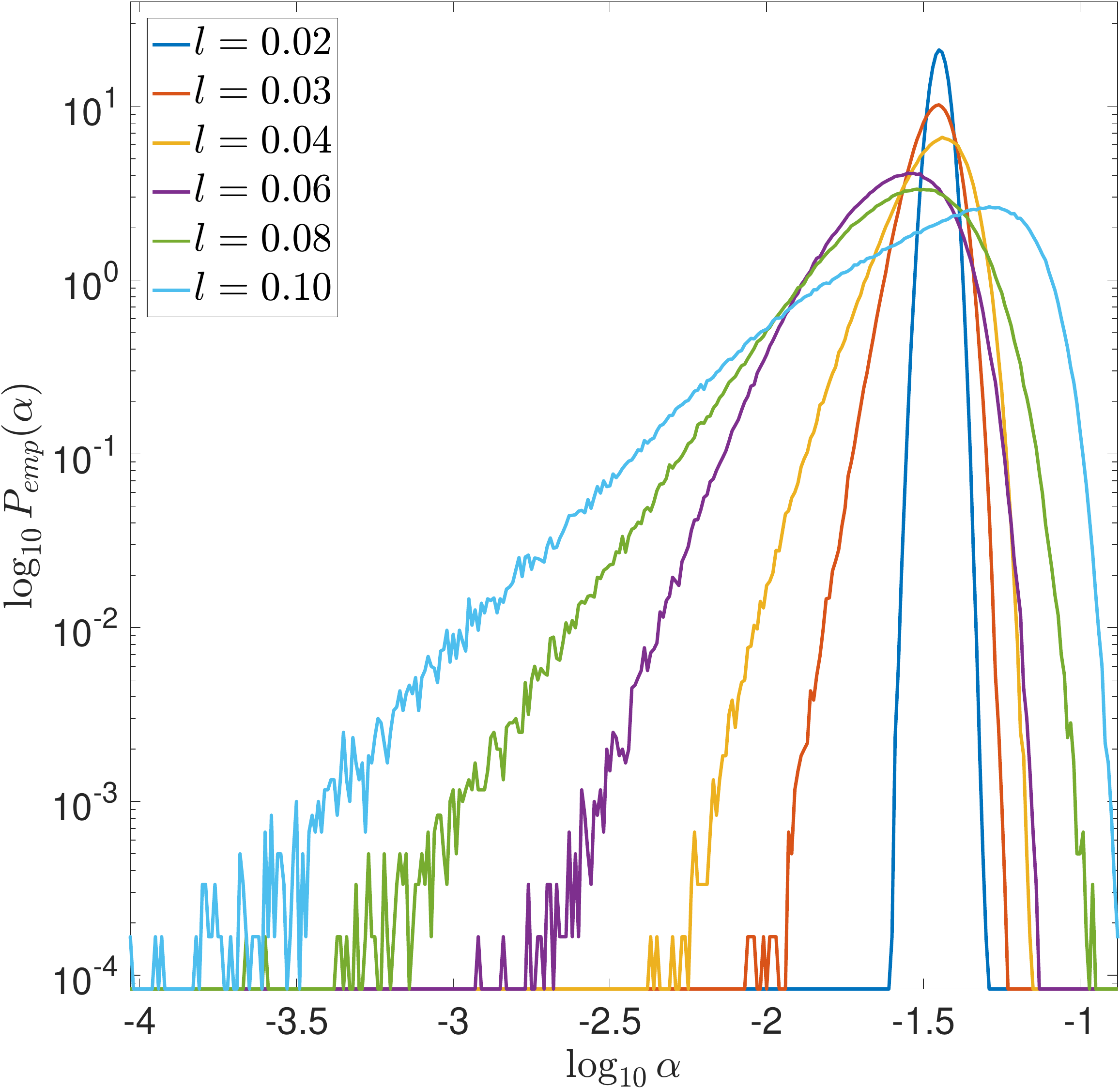}\label{subfig:L2Alpha_lVarDP}}\\
\subfigure[][$\alphaHatPSURE$]{\includegraphics[width= 0.49\textwidth]{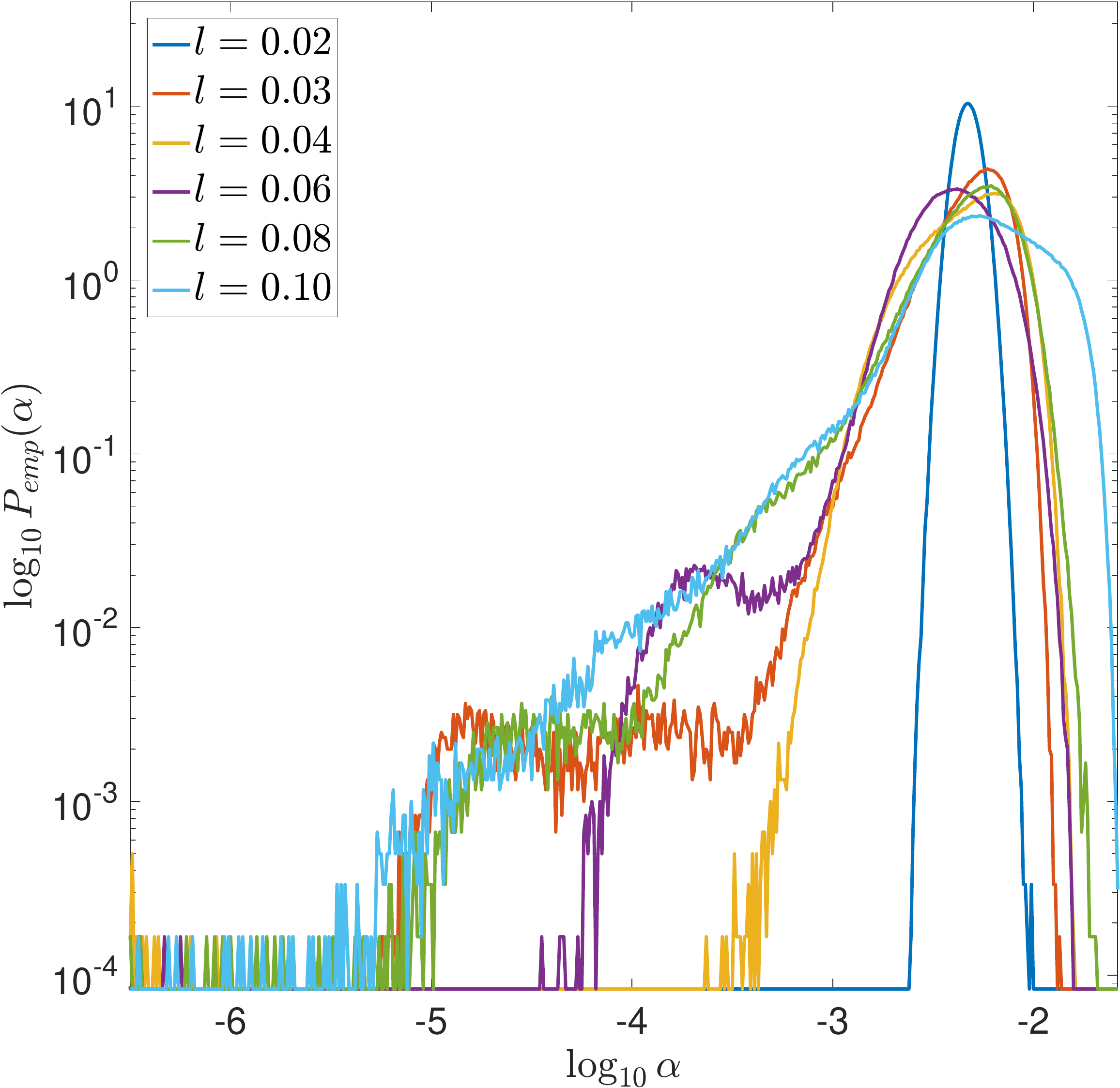}\label{subfig:L2Alpha_lVarPSURE}}
\subfigure[][$\alphaHatGSURE$]{\includegraphics[width= 0.49\textwidth]{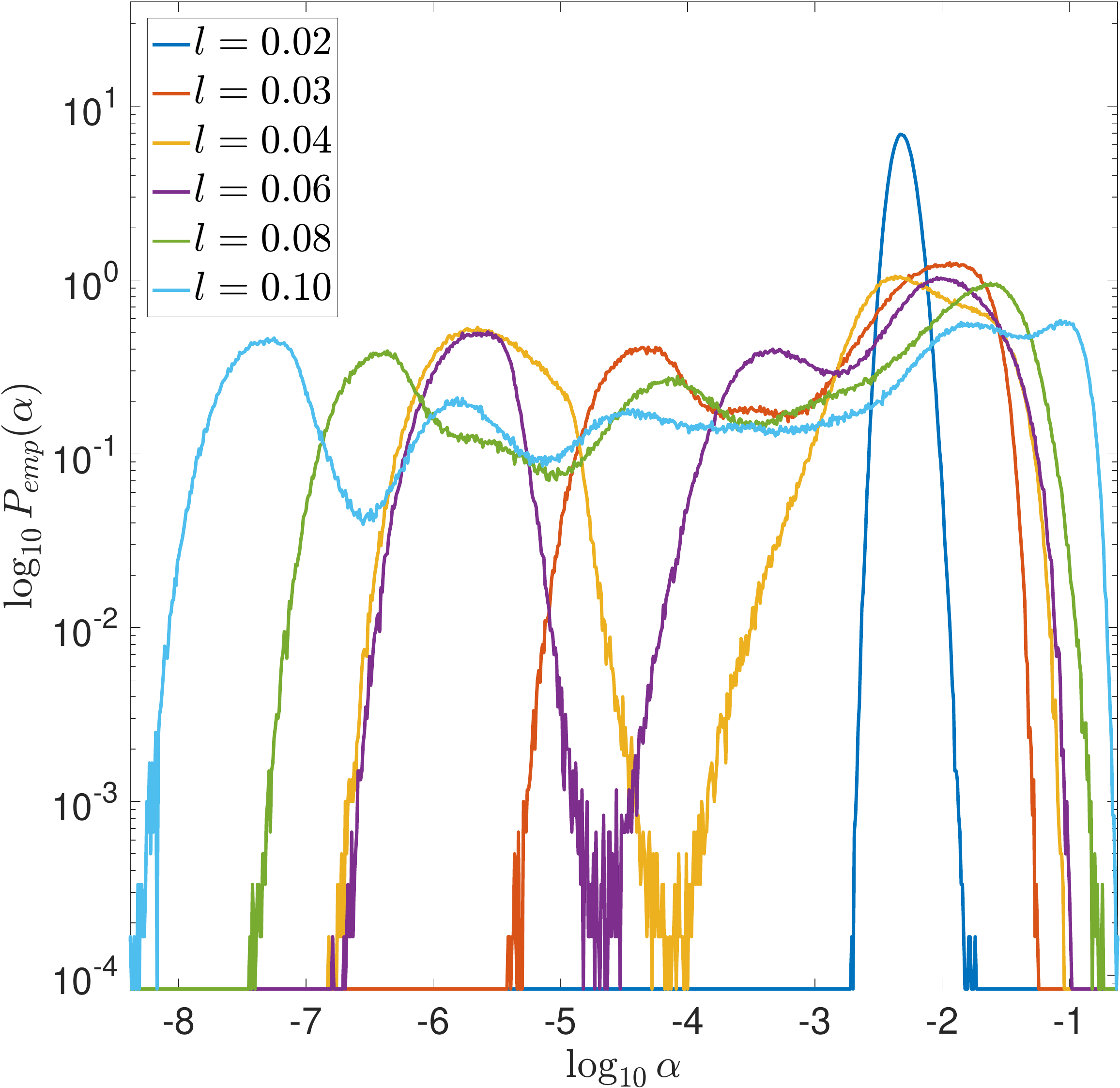}\label{subfig:L2Alpha_lVarGSURE}}
\caption{Empirical probabilities of $\alpha$ for $\ell_2$-regularization and different parameter choice rules for $m=64$ and varying $l$. \label{fig:L2AlphaHist_n64Varl06N6}}
\end{figure}

\begin{figure}[tb]
   \centering
\subfigure[][optimal]{\includegraphics[width= 0.49\textwidth]{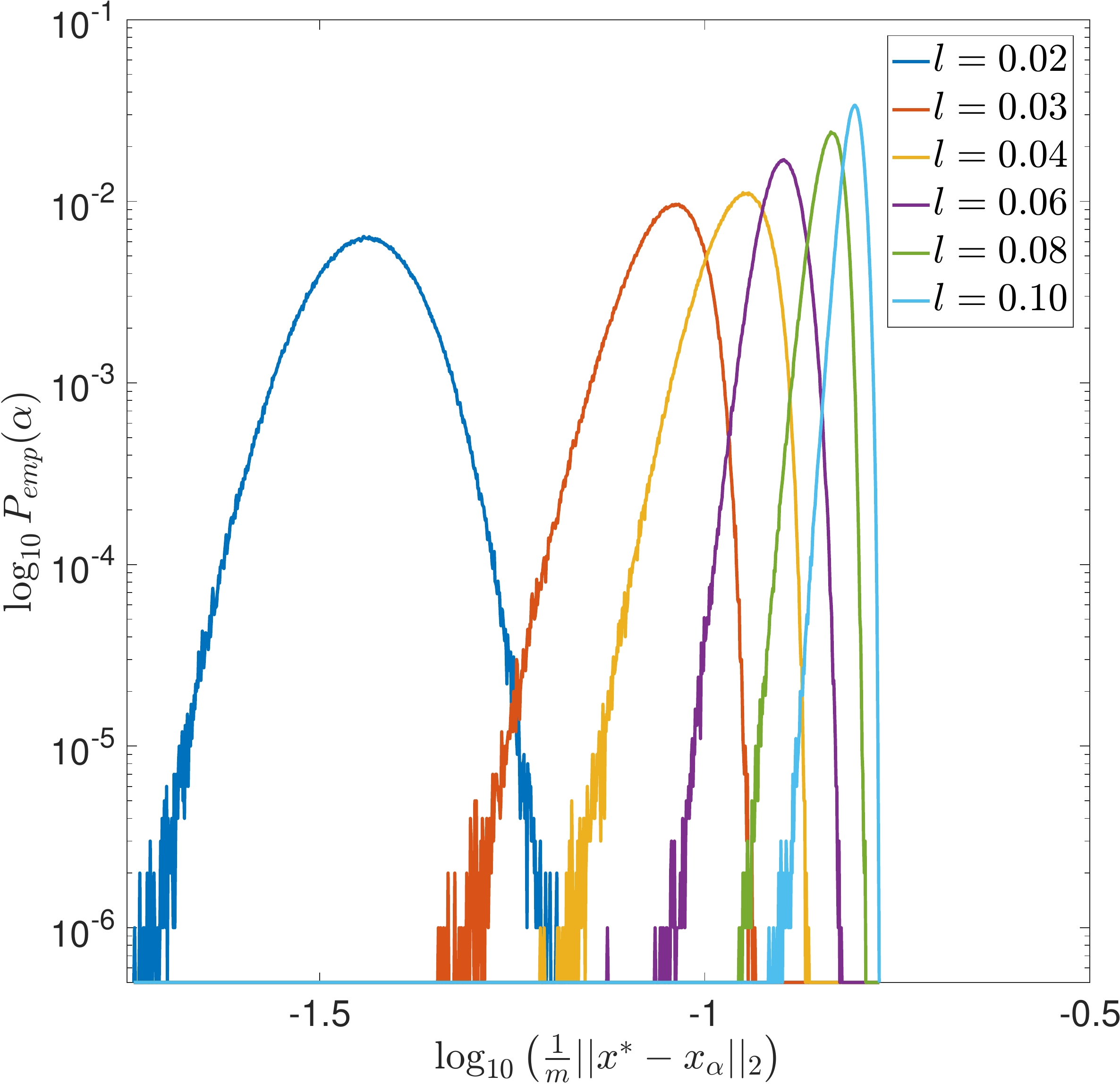}\label{subfig:L2Err_lVarOpt}}
\subfigure[][discrepancy principle]{\includegraphics[width= 0.49\textwidth]{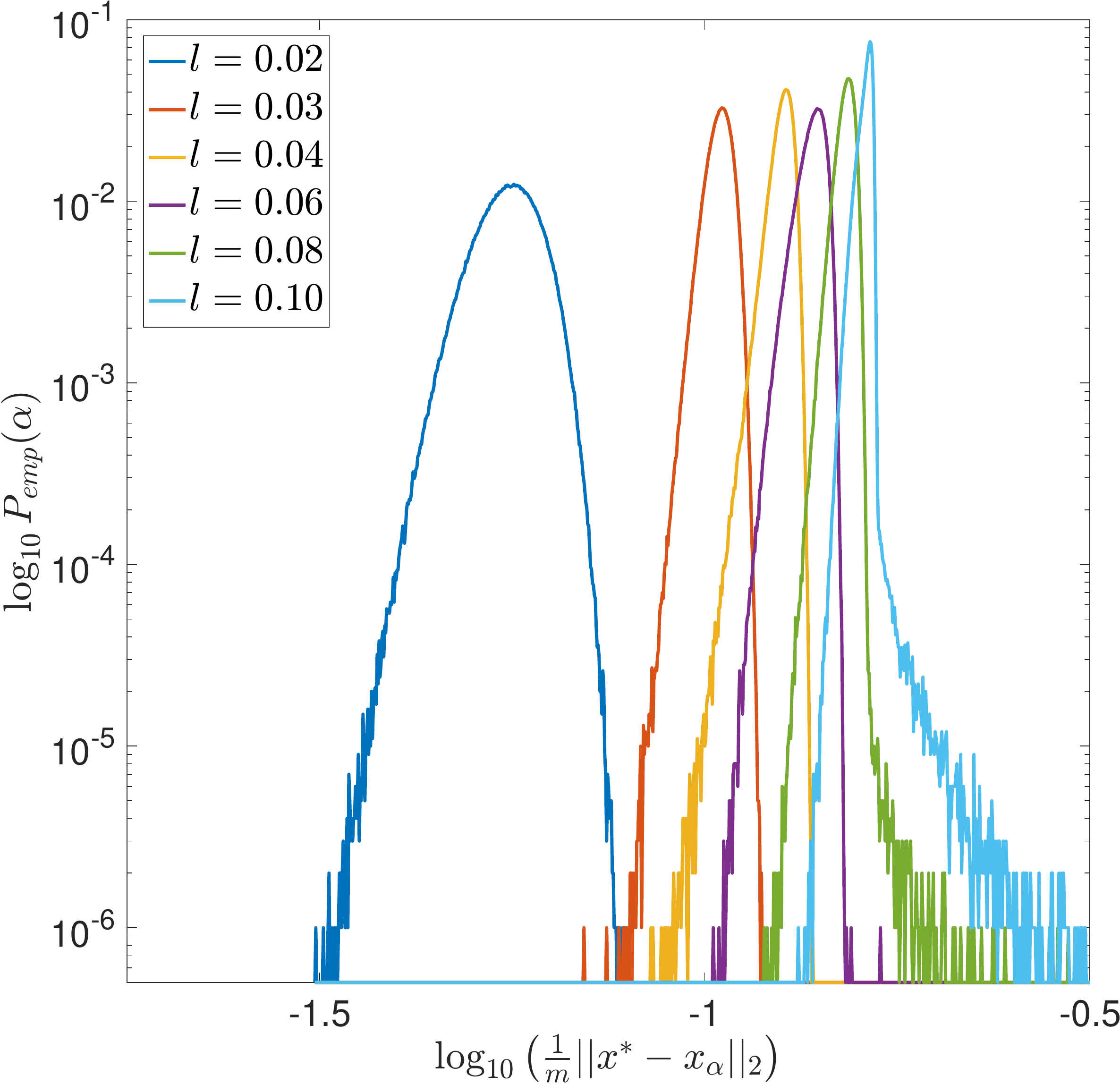}\label{subfig:L2_lVarDP}}\\
\subfigure[][PSURE]{\includegraphics[width= 0.49\textwidth]{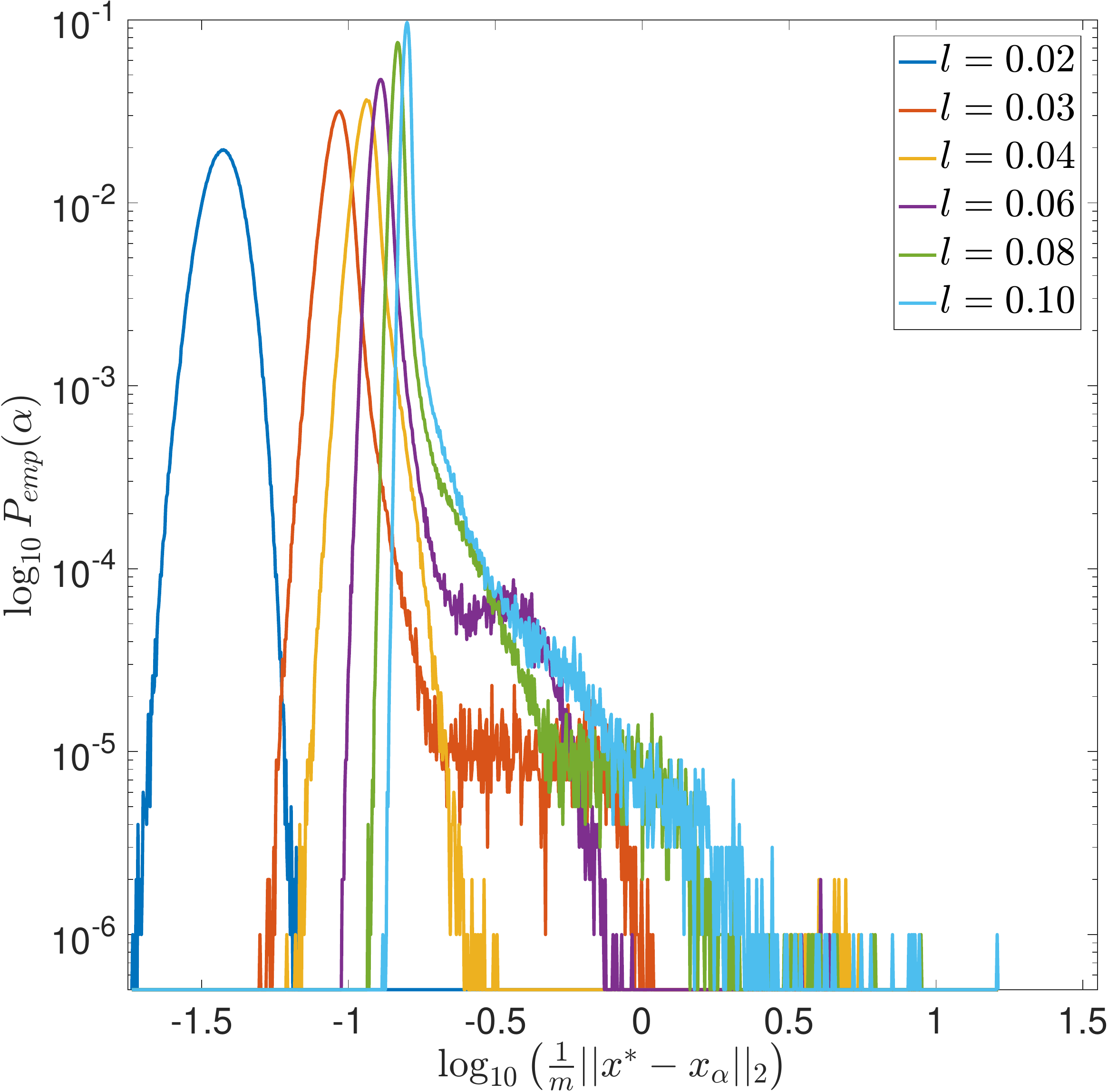}\label{subfig:L2Err_lVarPSURE}}
\subfigure[][SURE]{\includegraphics[width= 0.49\textwidth]{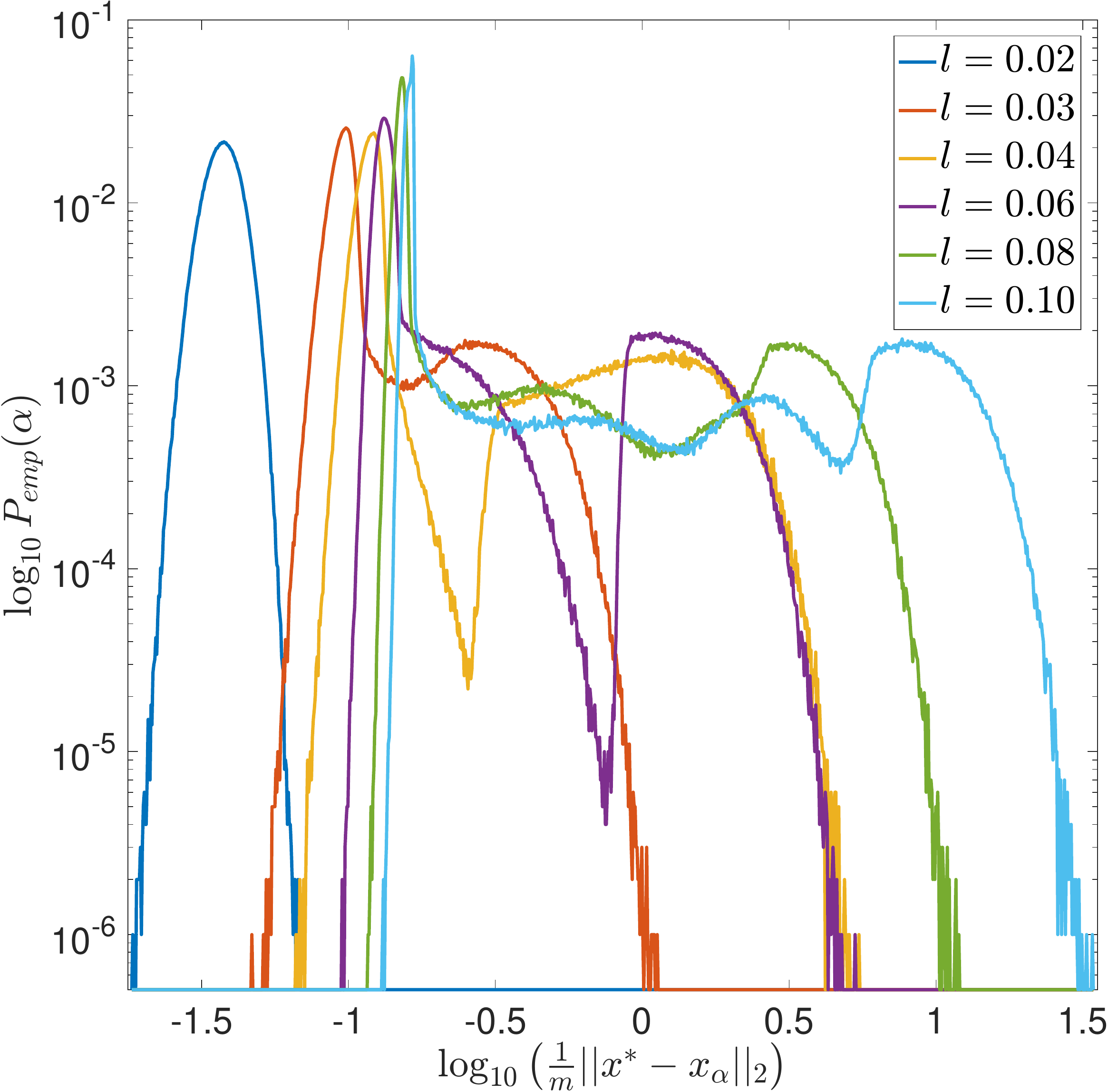}\label{subfig:L2Err_lVarGSURE}}
\caption{Empirical probabilities of $\log_{10}\left(\tfrac{1}{m} \| x^* - x_{\alpha} \|_2^2\right)$ for $\ell_2$-regularization and different parameter choice rules for $m=64$ and varying $l$. \label{fig:L2ErrHist_n64Varl06N6}}
\end{figure}

\subsection{Linear vs Logarithmical Grids} \label{subsec:LinVsLog}

One reason why the properties of SURE exposed in this work have not been noticed so far is that they only become apparent in very ill-conditioned problems (cf. Section \ref{sec:Intro}). Another reason is the way the risk estimators are typically computed: Firstly, for high dimensional problems, \eqref{eq:ExplTikhonov} often needs to be solved by an iterative method. For very small $\alpha$, the condition of $(A^* A + \alpha  I)$ is very large and the solver will need a lot of iterations to reach a given tolerance. If, instead, a fixed number of iterations is used, an additional regularization of the solution to \eqref{model} is introduced which alters the risk function. Secondly, again due to the computational effort, a coarse, linear $\alpha$-grid excluding $\alpha = 0$ instead of a fine, logarithmic one is often used for evaluating the risk estimators. For two of the risk estimations plotted in Figure \ref{subfig:L2GSURErisk}, Figure \ref{fig:L2LinVsLog} demonstrates that this insufficient coverage of small $\alpha$ values by the grid can lead to missing the global minimum and other misinterpretations.

\begin{figure}[tb]
   \centering
\subfigure[][]{\includegraphics[height= 0.47\textwidth]{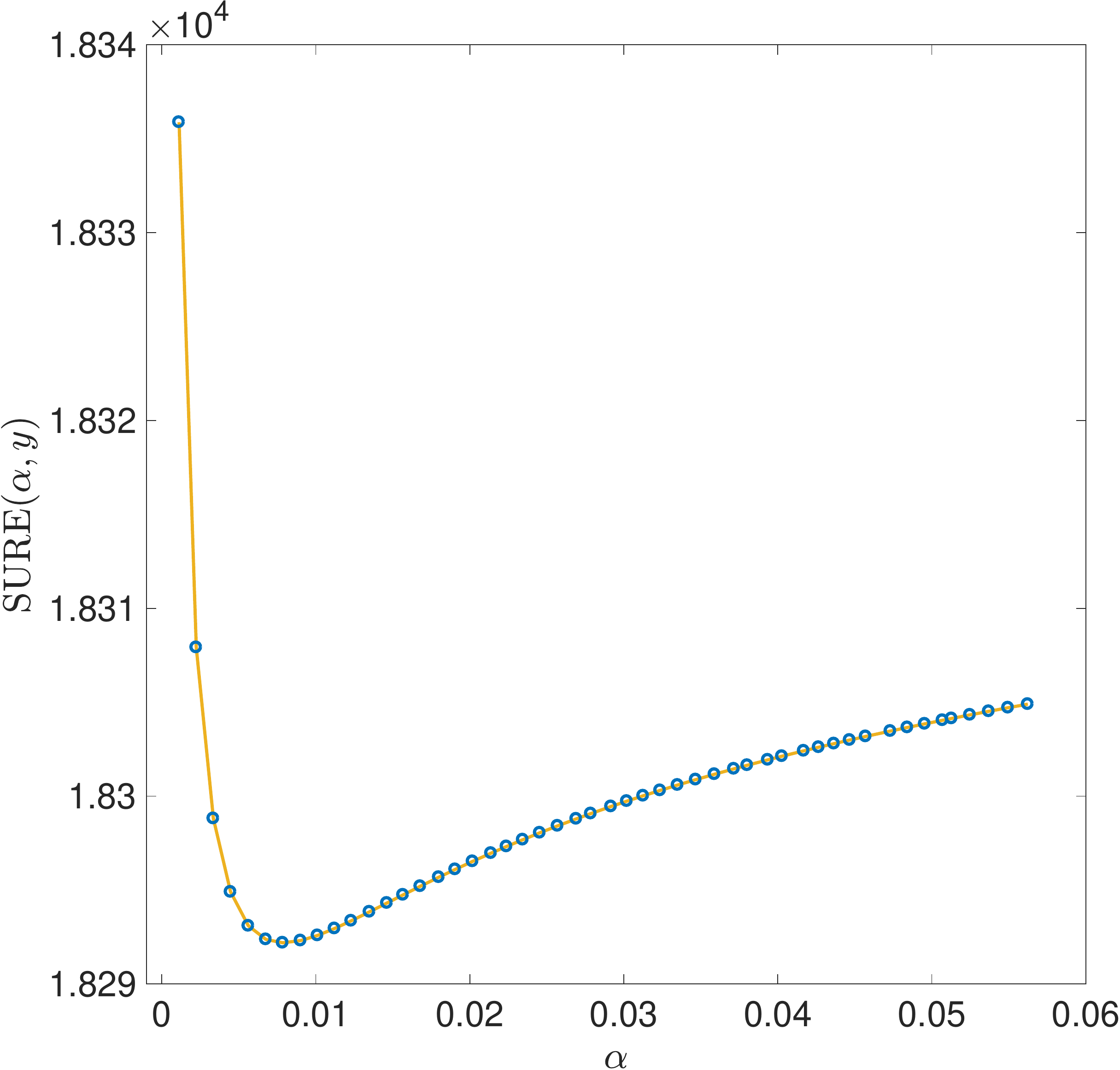}\label{subfig:L2LinVsLog2lin}}
\subfigure[][]{\includegraphics[height= 0.475\textwidth]{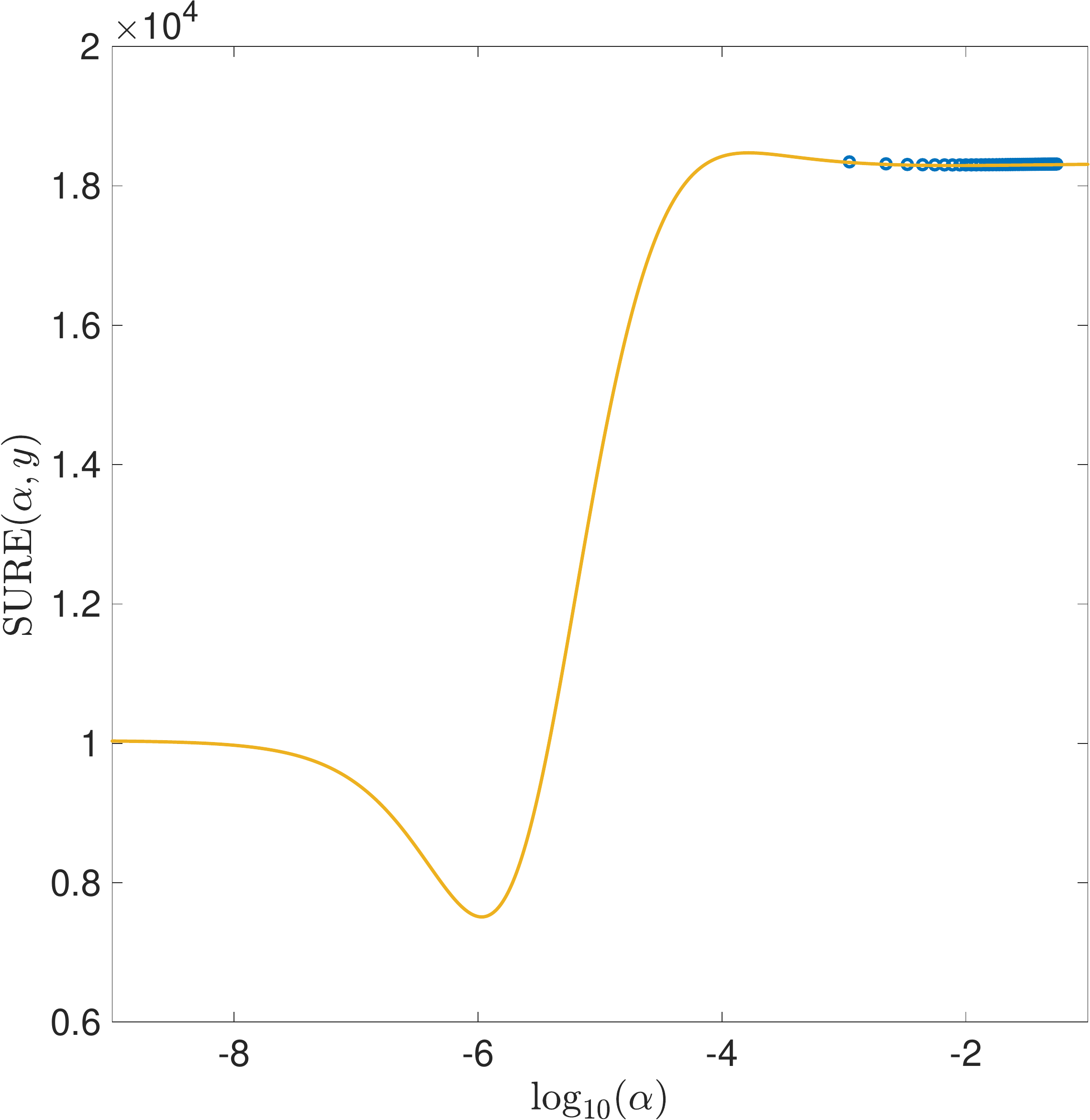}\label{subfig:L2LinVsLog2log}}\\
\subfigure[][]{\includegraphics[height= 0.455\textwidth]{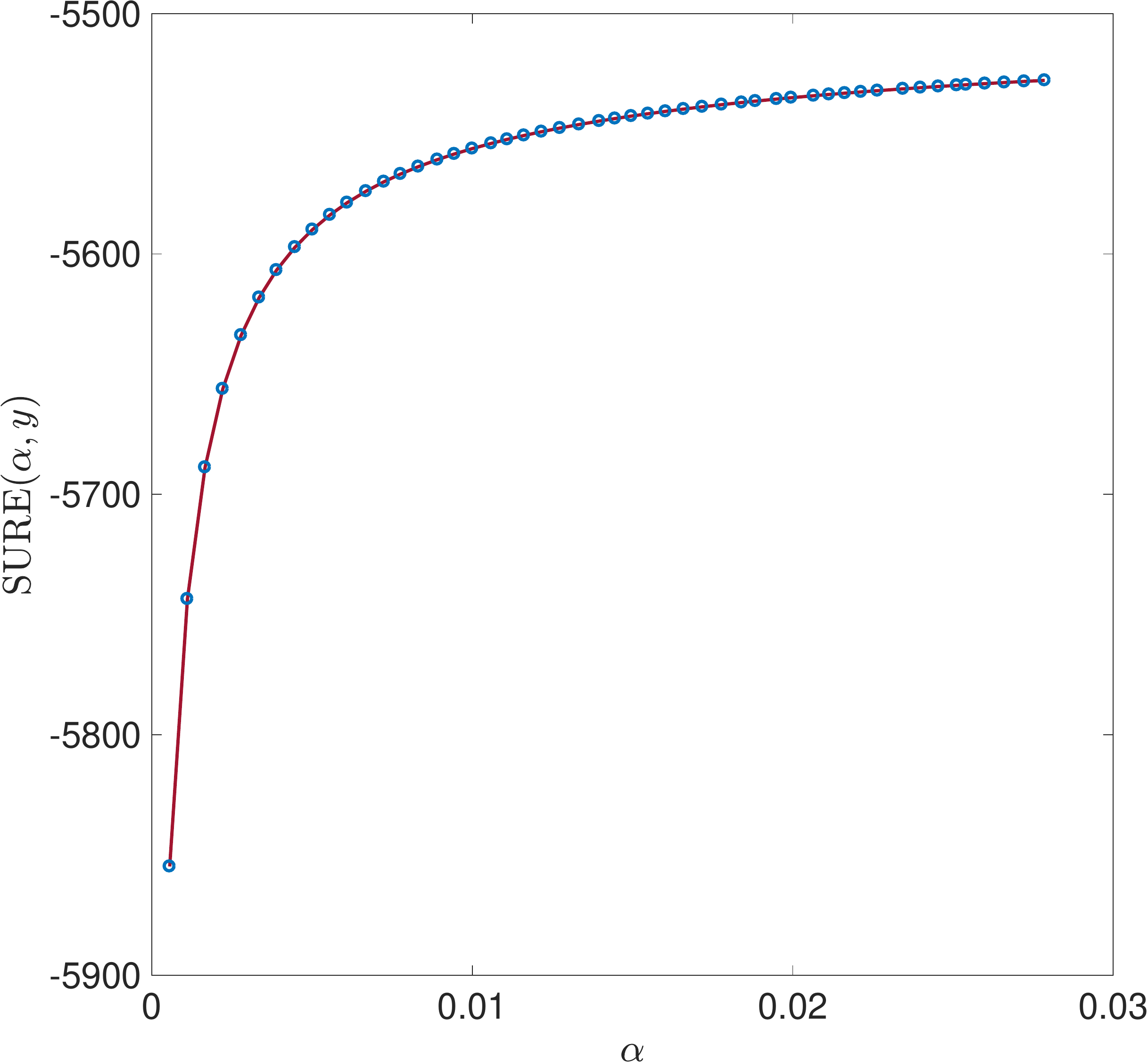}\label{subfig:L2LinVsLog7lin}}
\subfigure[][]{\includegraphics[height= 0.45\textwidth]{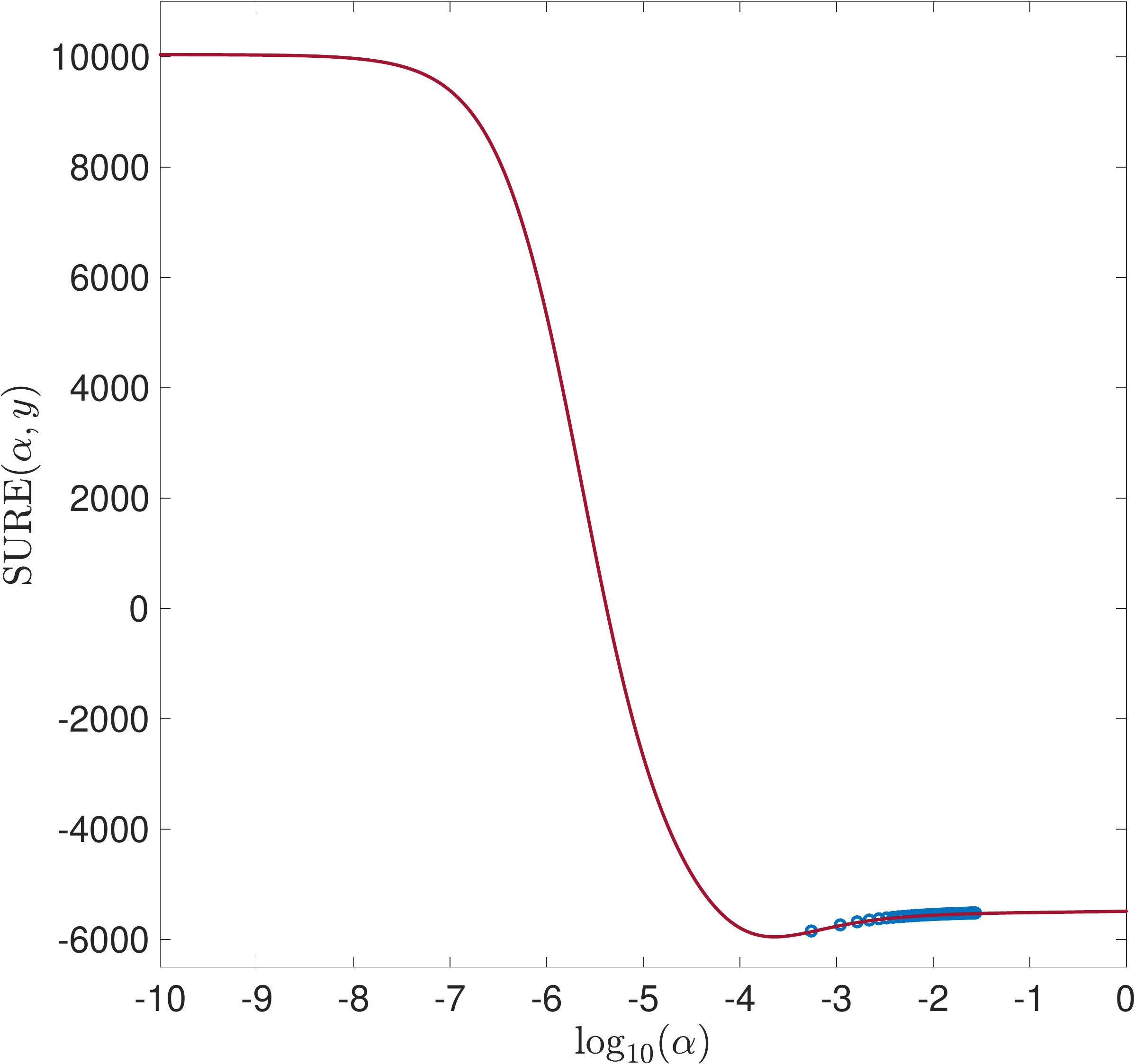}\label{subfig:L2LinVsLog7log}}
\caption{Illustration of the difference between evaluating the SURE risk on a coarse, linear grid for $\alpha$ as opposed to a fine, logarithmic one: In \subref{subfig:L2LinVsLog2lin}, a linear grid is constructed around $\alphaHatDP$ as $\alpha = \Delta_\alpha,2 \Delta_\alpha,\ldots,50\Delta_\alpha$ with $\Delta_\alpha = 2 \alphaHatDP/50$. While the plot suggests a clear minimum, \subref{subfig:L2LinVsLog2log} reveals that it is only a sub-optimal local minimum and that the linear grid did not cover the essential parts of $\gsure(\alpha,y)$. \subref{subfig:L2LinVsLog7lin} and \subref{subfig:L2LinVsLog7log} show the same plots for a different noise realization. Here, a linear grid will not even find a clear minimum. Both risk estimators are the same as those plotted in Figure \ref{subfig:L2GSURErisk} with the same colors. \label{fig:L2LinVsLog}}
\end{figure}

\section{Numerical Studies for Non-Quadratic Regularization} \label{sec:NumStudiesL1}

In this section, we consider the popular sparsity-inducing $R(x) = \norm{ x }_1$ as a regularization functional (LASSO penalty) to examine whether our results also apply to non-quadratic regularization functionals. For this, let $I$ be the support of $\hat{x}_\alpha(y)$ and $J$ its complement. Let further $|I| = k$ and $P_I \in \R^{k \times n}$ be a projector onto $I$ and $A_I$ the restriction of $A$ to $I$. We have that   
\begin{equation*}
\text{df}_\alpha = \|\hat{x}_\alpha(y)\|_0 = k \qquad \quad \text{and} \qquad \quad \text{gdf}_\alpha = \trace(\Pi B^{[J]}), \quad B^{[J]} := P_I (A_I^* A_I)^{-1} P_I^*,
\end{equation*}
as shown, e.g., in \cite{Vaiter2014,Dossal2013,Deledalle2012}, which allows us to compute PSURE \eqref{PSURE} and SURE \eqref{defGSURE}. Notice that while $\hat x_\alpha(y)$ is a continuous function of $\alpha$ \cite{bringmann2016homotopy}, PSURE and SURE are discontinuous at all $\alpha$ where the support $I$ changes.\\
To carry out similar numerical studies as those presented the last section, we have to overcome several non-trivial difficulties: While there exist various iterative optimization techniques to solve \eqref{eq:VarReg} nowadays (see, e.g., \cite{BuSaSt14}), each method typically only works well for certain ranges of $\alpha$, $\cond(A)$ and tolerance levels to which the problem should be solved. In addition, each method comes with internal parameters that have to be tuned for each problem separately to obtain fast convergence. As a result, it is difficult to compute a consistent series of $\hat x_\alpha(y)$ for a given logarithmical $\alpha$-grid, i.e., that accurately reproduces all the change-points in the support and has a uniform accuracy over the grid. Our solution to this problem is to use an all-at-once implementation of ADMM \cite{BoPaChPeEc11} that solves \eqref{eq:VarReg} for the whole $\alpha$-grid simultaneously, i.e., using exactly the same initialization, number of iterations and step sizes. See Appendix \ref{sec:ADMM} for details. In addition, an extremely small tolerance level ($tol = 10^{-14}$) and $10^4$ maximal iterations were used to ensure a high accuracy of the solutions. \\
Another problem for computing quantities like \eqref{supExpNum} is that we cannot compute the expectations defining the real risks $\mspe$ \eqref{PSURE} and $\msee$ \eqref{defGSURE} anymore: We have to estimate them as the sample mean over PSURE and SURE in a first run of the studies, before we can compute \eqref{supExpNum} in a second run (wherein $\mspe$ and $\msee$ are replaced by the estimates from the first run).\\
We considered scenarios with each combination of $m=n=16, 32, 64, 128, 256, 512$, $l = 0.02, 0.04, 0.06$ and $\sigma = 0.1$. Depending on $m$, $N_\varepsilon = 10^5,10^4,10^4,10^4,10^3,10^3$ noise realizations were examined. The computation was based on a logarithmical $\alpha$-grid where $\log_{10} \alpha$ is increased linearly in between -10 and 10 with a step size of $0.01$.

\paragraph{Risk Plots:}

Figure \ref{fig:L1RiskPlots} shows the different risk functions and estimates thereof. The jagged form of the PSURE and SURE plots evaluated on this fine $\alpha$-grid indicates that the underlying functions are discontinuous. Also note that while PSURE and SURE for each individual noise realization are discontinuous, $\mspe$ and $\msee$ are smooth and continuous, as can be seen already from the empirical means over $N_\varepsilon = 10^4$.

\begin{figure}[tb]
   \centering
\subfigure[][DP]{\includegraphics[width= 0.32 \textwidth]{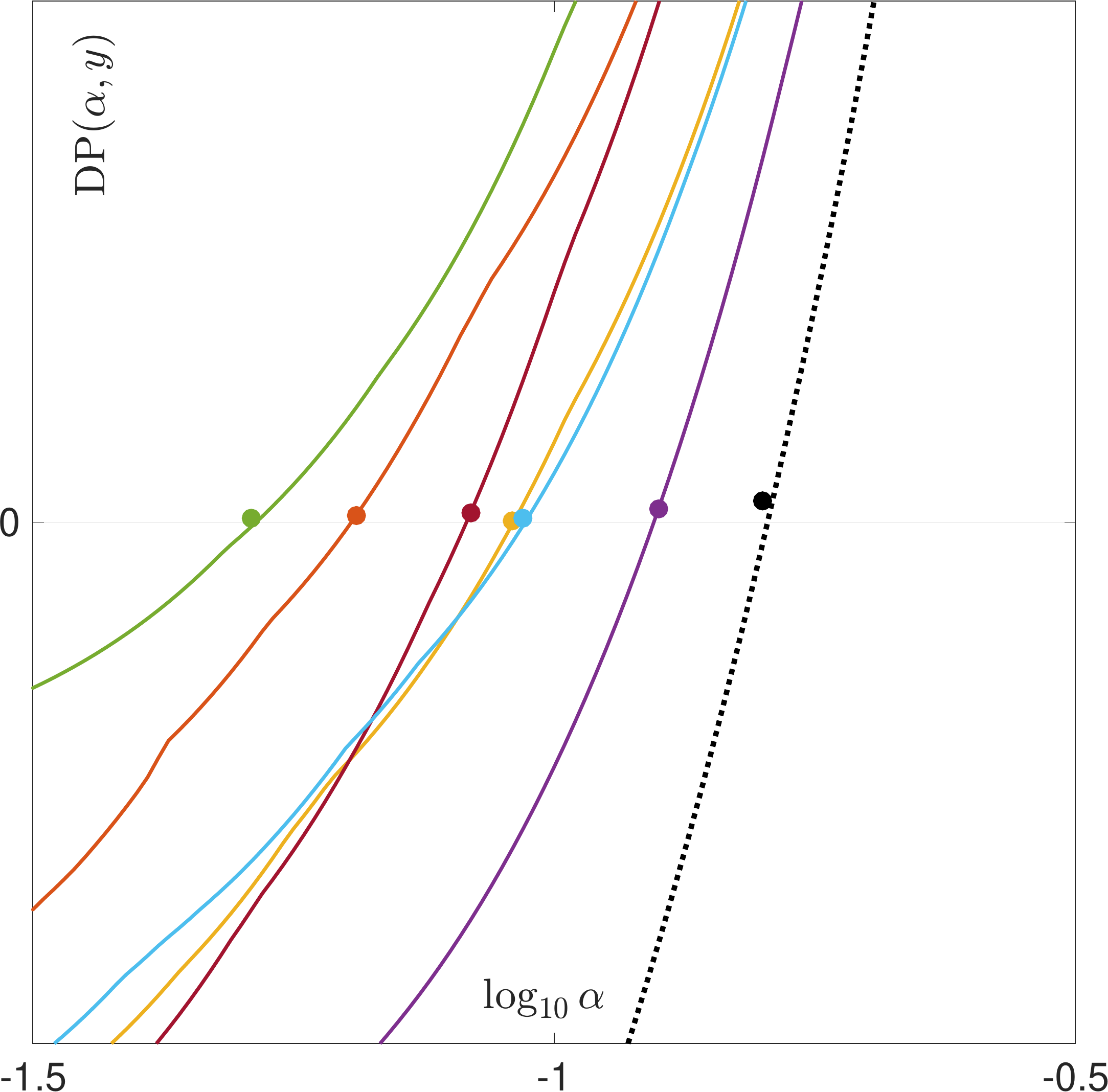}\label{subfig:L1DPrisk}}
\subfigure[][PSURE]{\includegraphics[width= 0.32\textwidth]{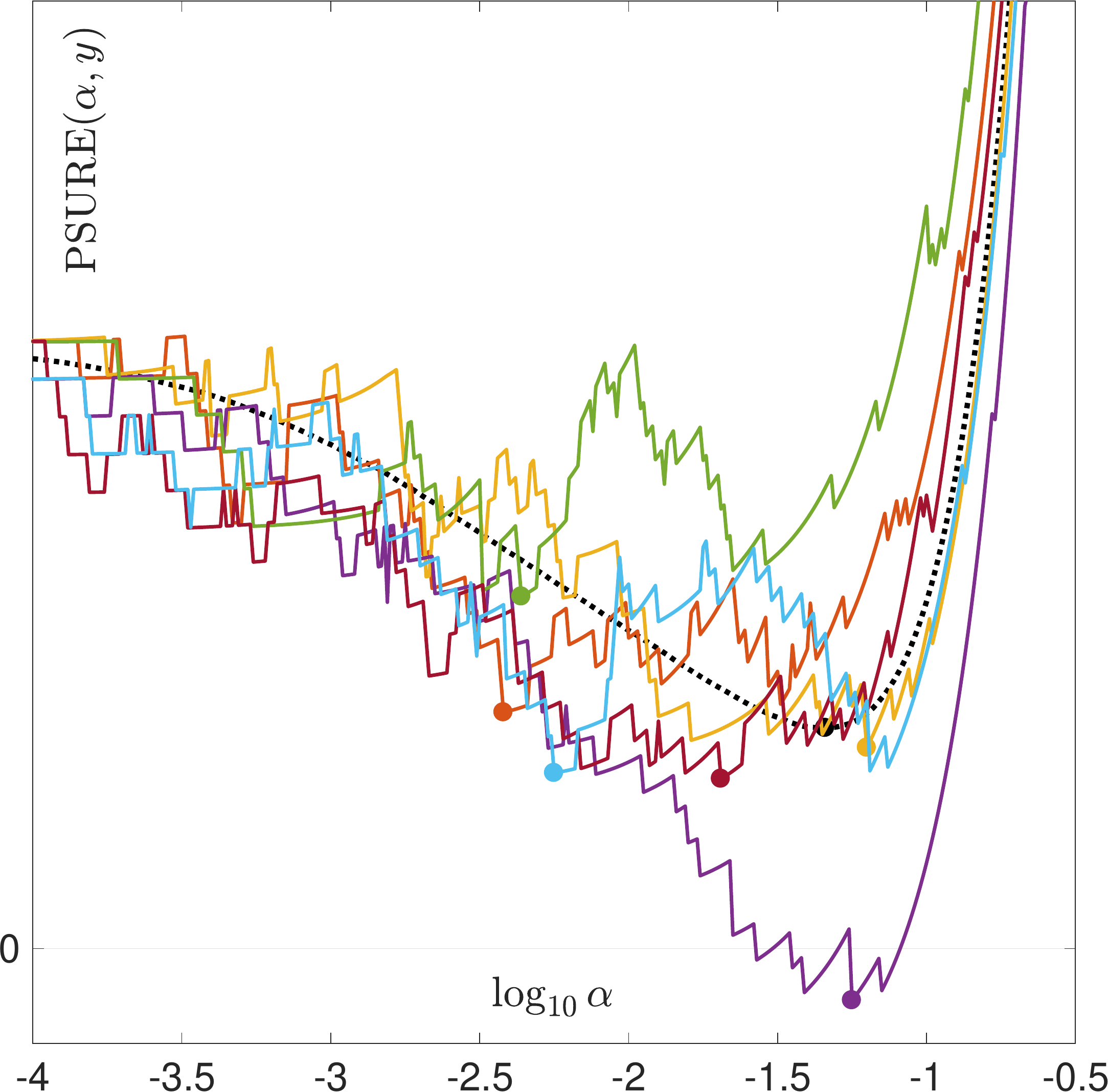}\label{subfig:L1PSURErisk}}
\subfigure[][SURE]{\includegraphics[width= 0.32\textwidth]{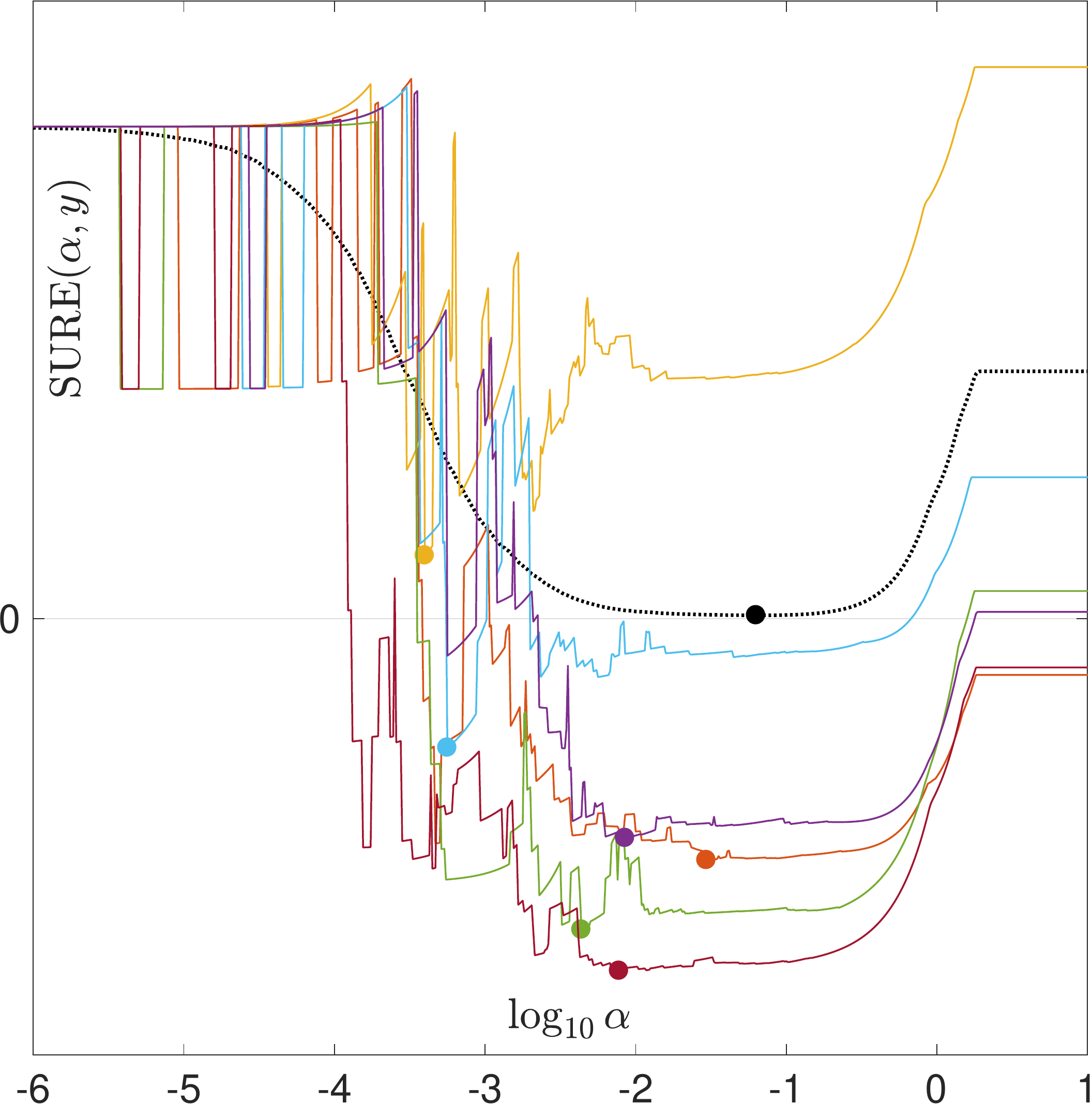}\label{subfig:L1GSURErisk}}
\caption{Risk functions (black dotted line), $k = 1,\ldots,6$ estimates thereof (solid lines) and their corresponding minima/roots (dots on the lines) in the setting described in Figure \ref{fig:Setting} using $\ell_1$-regularization: \subref{subfig:L1DPrisk} $\DP(\alpha,Ax^*)$ and $\DP(\alpha,y^k)$. \subref{subfig:L1PSURErisk} $\mspe(\alpha)$ (empirical mean over $N_\varepsilon = 10^4$) and $\psure(\alpha,y^k)$. \subref{subfig:L1GSURErisk} $\msee(\alpha)$ (empirical mean over $N_\varepsilon = 10^4$) and $\gsure(\alpha,y^k)$. \label{fig:L1RiskPlots}}
\end{figure}

\paragraph{Empirical Distributions:}

Figure \ref{fig:L1Hist_n64l06N6} shows the empirical distributions of the different parameter choice rules for $\alpha$.  Here, the optimal ${\alpha}^*$ is chosen as the one minimizing the $\ell_1$-error $\|x^* - x_{\hat \alpha} \|_1$ to the true solution $x^*$. We can observe similar phenomena as for $\ell_2$-regularization. In particular, the distributions for SURE, also have multiple modes at small values of $\alpha$ and at large values of $\ell_1$-error.

\begin{figure}[tb]
   \centering
\subfigure[][]{\includegraphics[height= 0.48\textwidth]{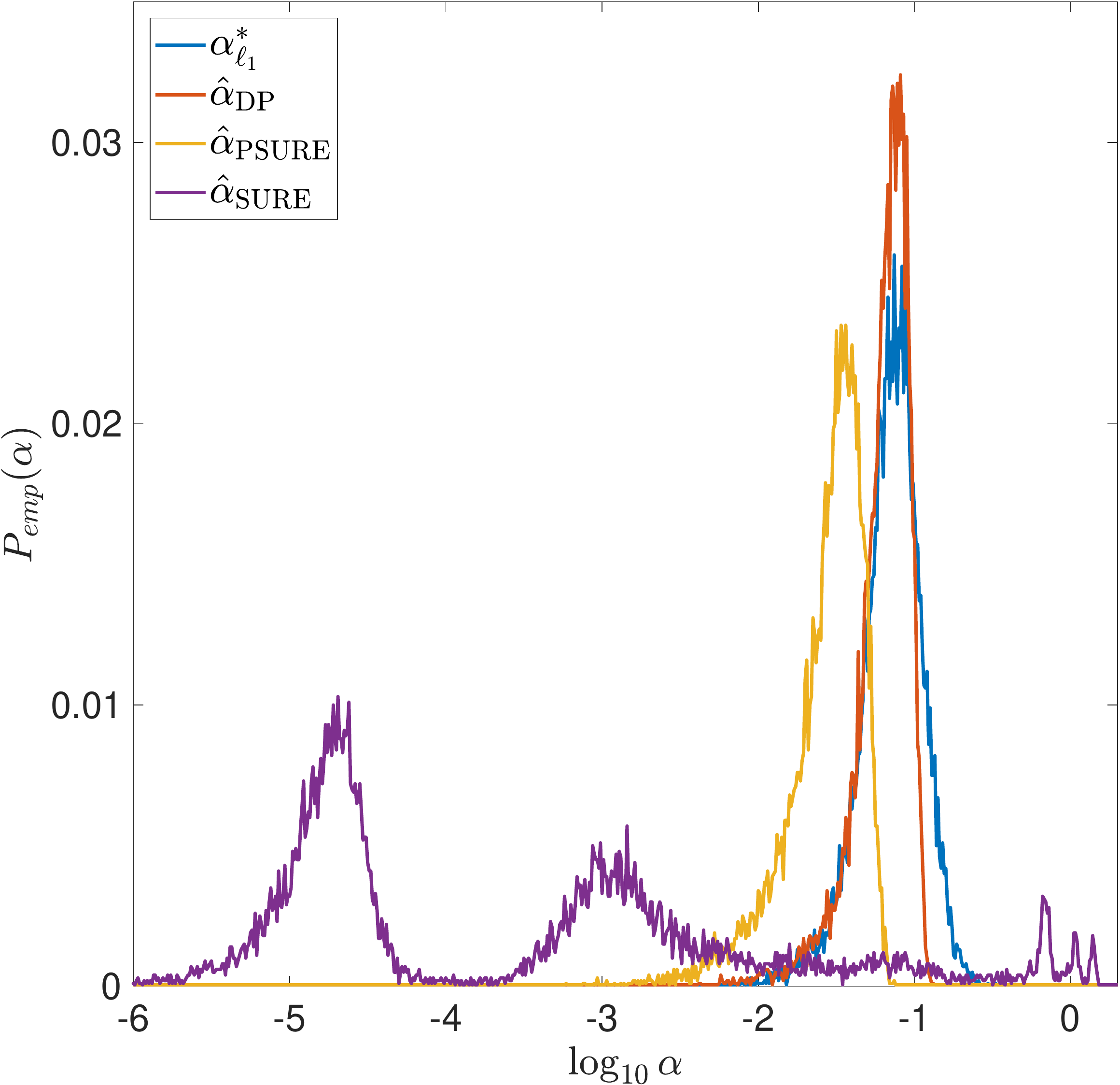}\label{subfig:L1n64AlphaHist}}
\subfigure[][]{\includegraphics[height= 0.48\textwidth]{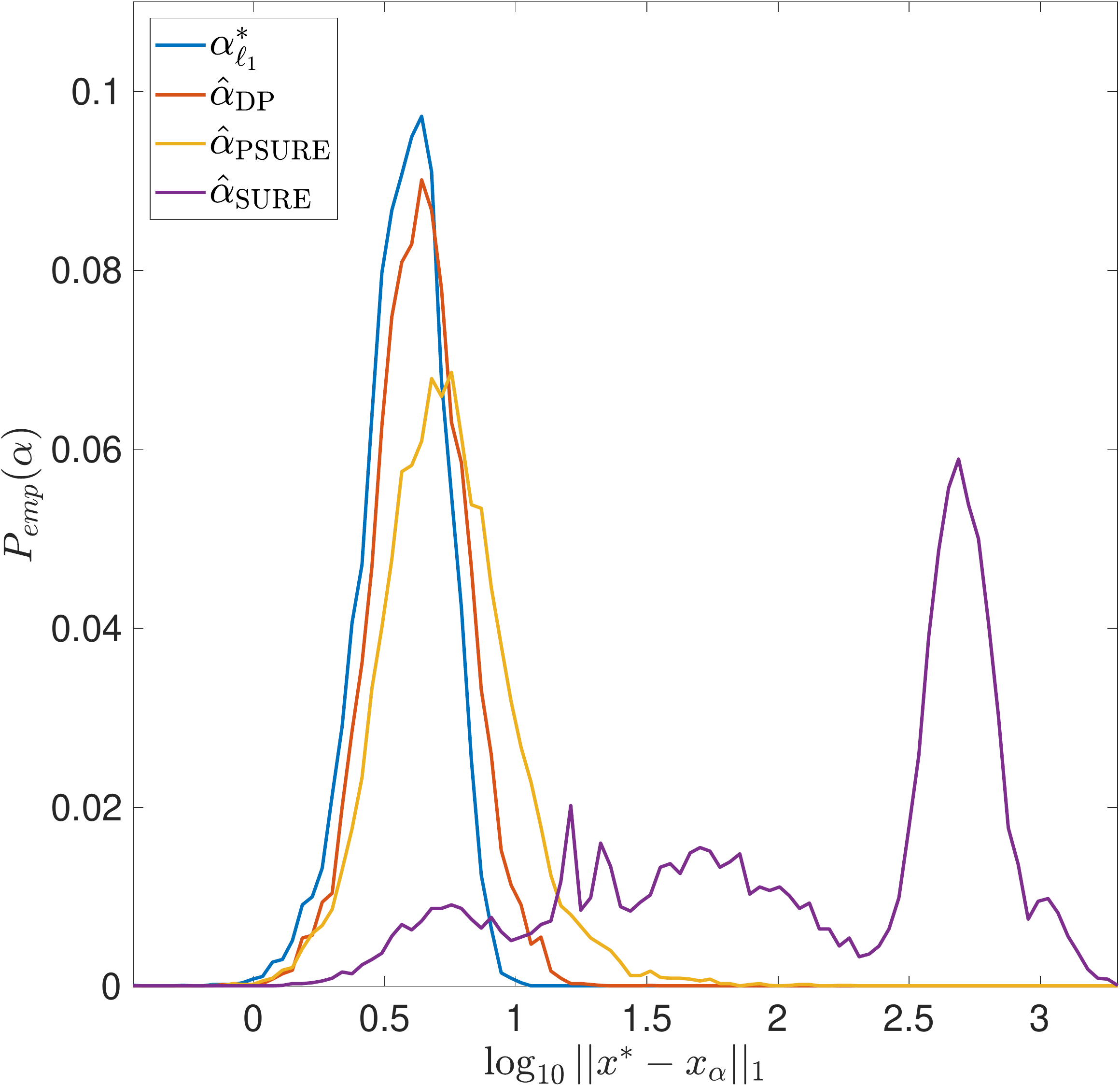}\label{subfig:L1n64ErrHistL1}}
\caption{Empirical probabilities of \subref{subfig:L1n64AlphaHist} $\alpha$ and \subref{subfig:L1n64ErrHistL1} the corresponding $\ell_1$-error for different parameter choice rules using $\ell_1$-regularization, $m=n=64$, $l = 0.06$, $\sigma = 0.1$ and $N = 10^4$ samples of $\varepsilon$. \label{fig:L1Hist_n64l06N6}}
\end{figure}

\paragraph{Sup-Theorems:}

Due to the lack of explicit formulas for the $\ell_1$-regularized solution $x_\alpha(y)$, carrying out similar analysis as in Section \ref{sec:Theory} to derive theorems such as  Theorems \ref{PSURER} and \ref{GSURER} is very challenging. In this work, we only illustrate that similar results may hold for the case of $\ell_1$-regularization by computing the left hand side of \eqref{ExpSupPSURE} and \eqref{ExpSupGSURE} based on our samples. The results are shown in Figure \ref{fig:IlluTheoL1} and are remarkably similar to those shown in Figure \ref{fig:IlluTheoL2}.

\begin{figure}[tb]
   \centering
\subfigure[][PSURE]{\includegraphics[width= 0.48\textwidth]{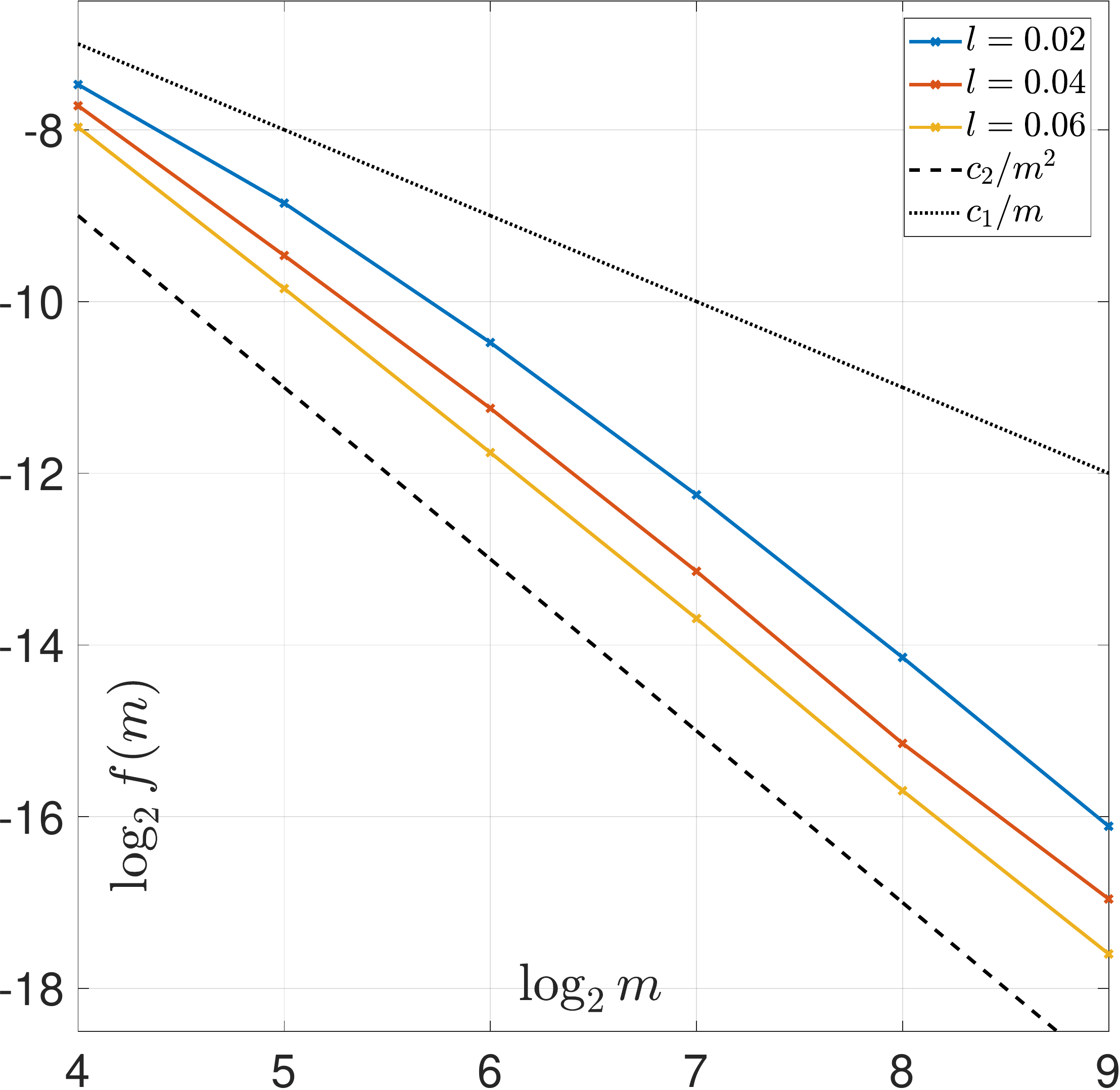}\label{subfig:TheoIlluL1PSURE}}
\subfigure[][SURE]{\includegraphics[width= 0.48\textwidth]{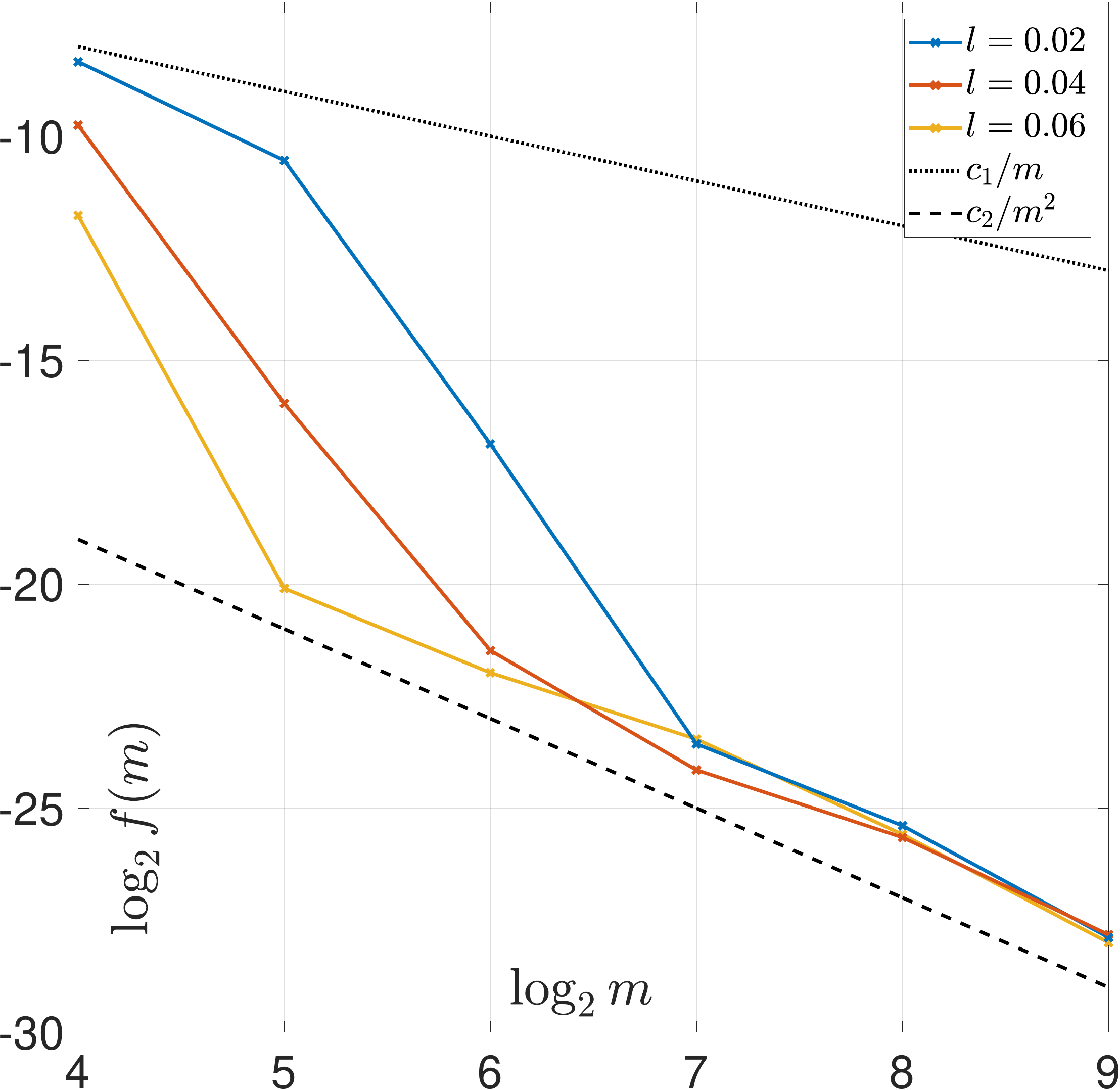}\label{subfig:TheoIlluL1GSURE}}
\caption{Illustration that Theorems \ref{PSURER} and \ref{GSURER} might also hold for $\ell_1$-regularization: The left hand side of \eqref{ExpSupPSURE}/\eqref{ExpSupGSURE} is estimated by the sample mean and plotted vs. $m$. The black dotted lines were added to compare the order of convergence.} \label{fig:IlluTheoL1}
\end{figure}

\paragraph{Linear Grids and Accurate Optimization}

All the issues raised in Section \ref{subsec:LinVsLog} about why the properties of SURE revealed in this work are likely to be overlooked when working on high dimensional problems are even more crucial for the case of $\ell_1$-regularization: For computational reasons, the risk estimators are often evaluated on a coarse, linear $\alpha$-grid using a small, fixed number of iterations of an iterative method such as ADMM. Figure \ref{fig:L1LinVsLog} illustrates that this may obscure important  features of the real SURE function, such as the strong discontinuities for small $\alpha$, or even change it significantly.  

\begin{figure}[tb]
   \centering
\subfigure[][]{\includegraphics[height= 0.47\textwidth]{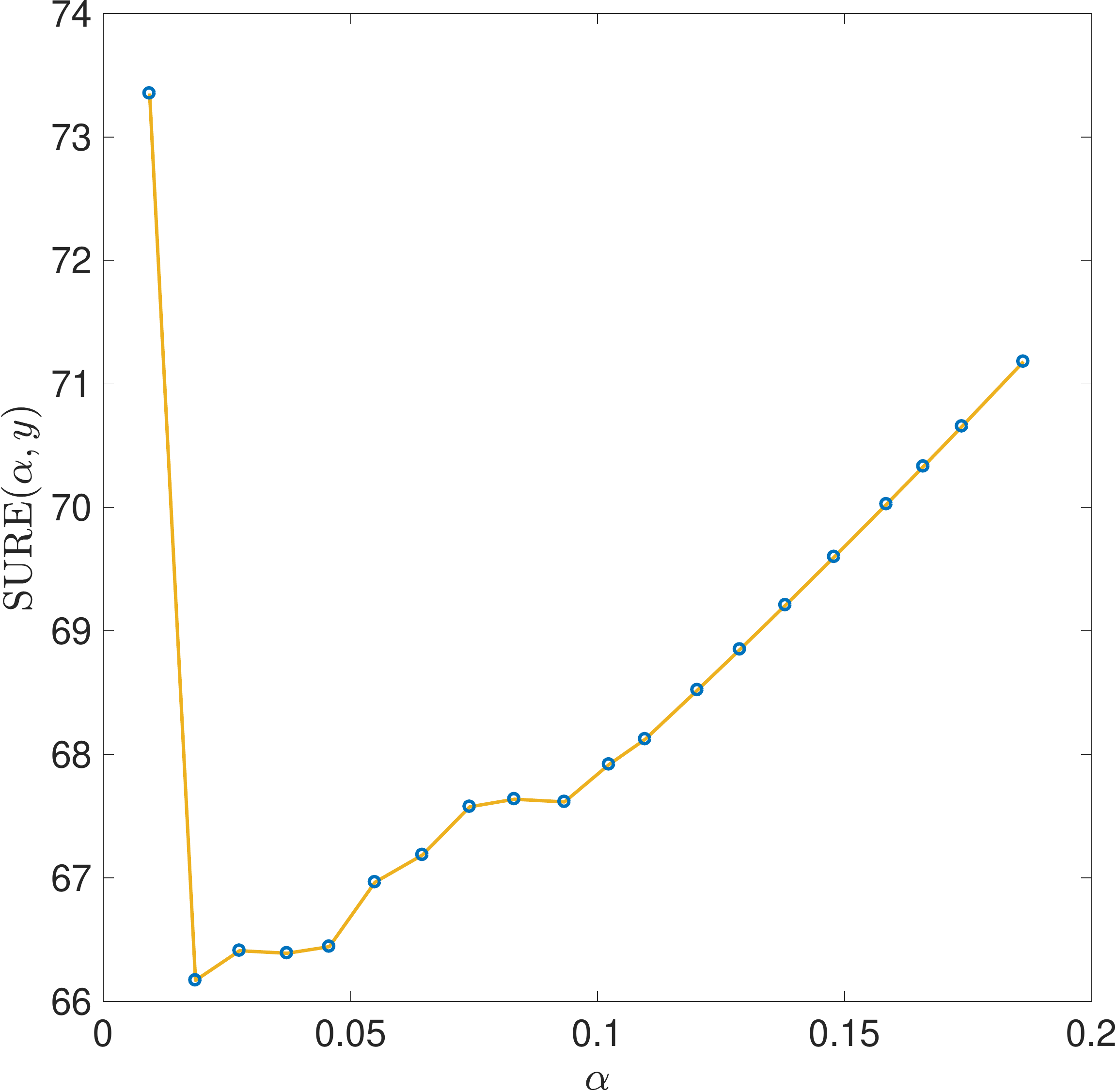}\label{subfig:L1LinVsLog2lin}}
\subfigure[][]{\includegraphics[height= 0.466\textwidth]{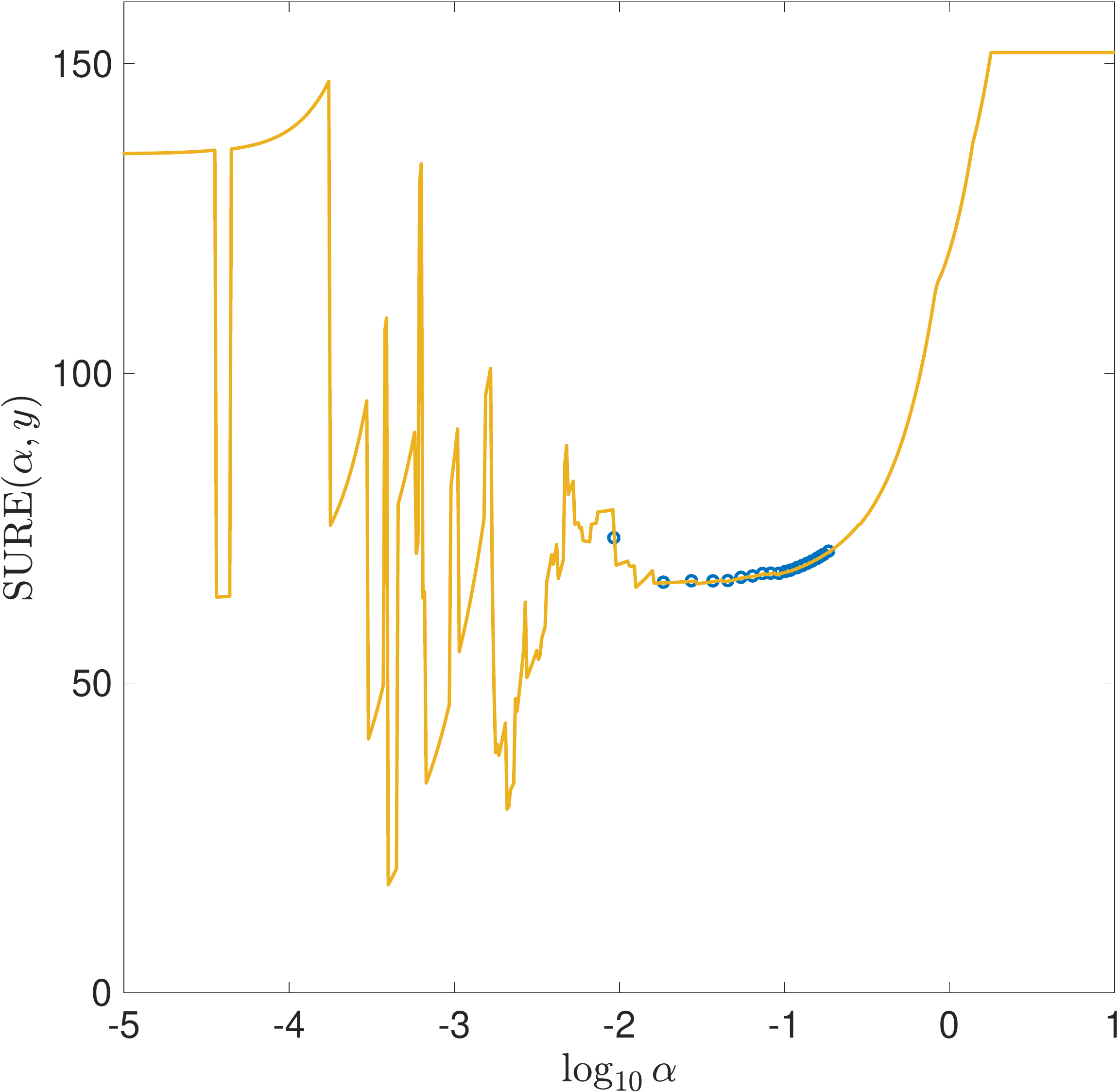}\label{subfig:L1LinVsLog2log}}\\
\subfigure[][]{\includegraphics[height= 0.47\textwidth]{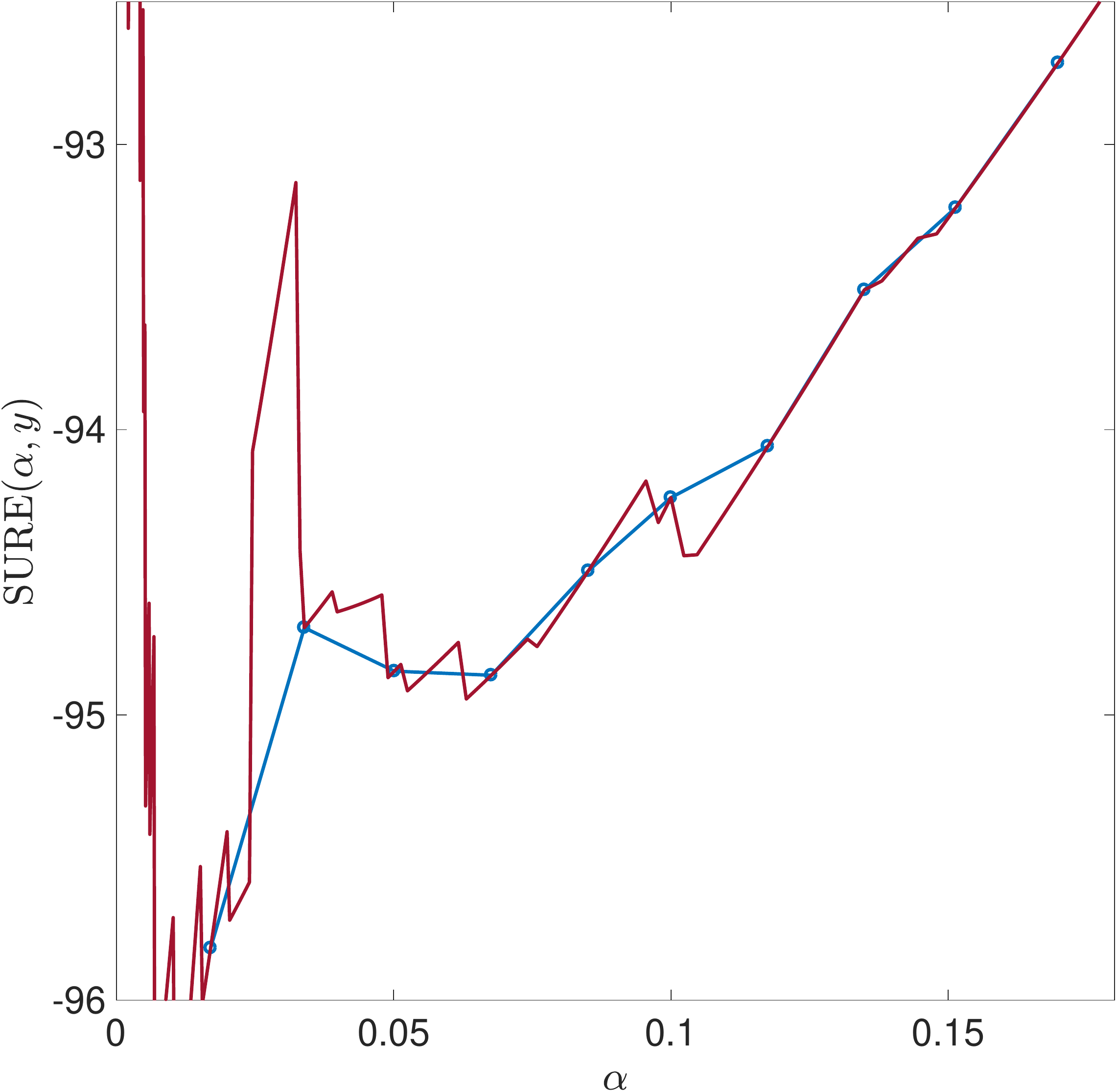}\label{subfig:L1CoarseGrid}}
\subfigure[][]{\includegraphics[height= 0.47\textwidth]{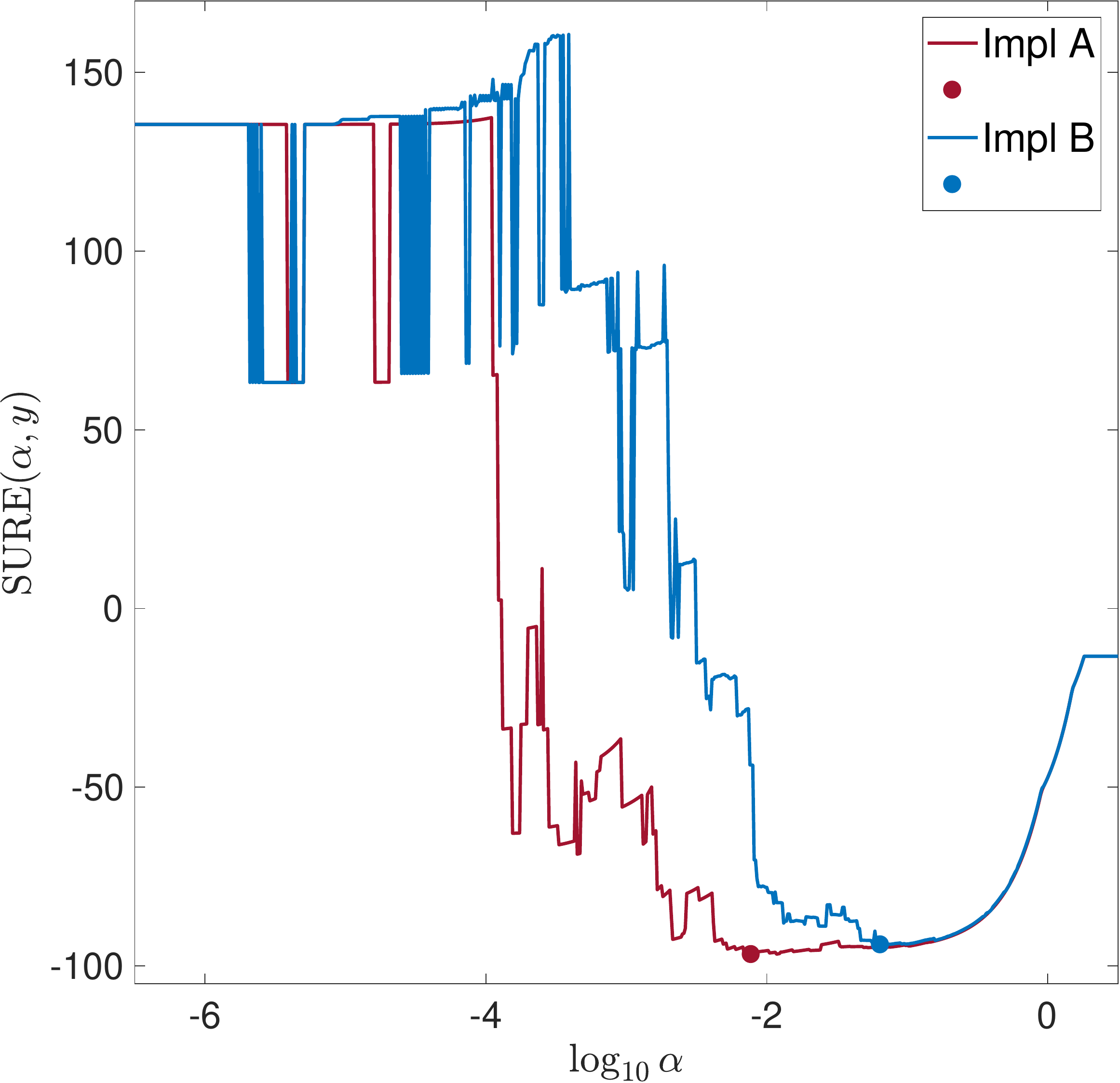}\label{subfig:L1iterOpt}}
\caption{Illustration of the difficulties of evaluating the SURE risk in the case of $\ell_1$-regularization: In \subref{subfig:L1LinVsLog2lin}, a coarse linear grid is constructed around $\alphaHatDP$ as $\alpha = \Delta_\alpha,2 \Delta_\alpha,\ldots,20\Delta_\alpha$ with $\Delta_\alpha =  \alphaHatDP/10$. Similar to Figure \ref{subfig:L2LinVsLog2lin} the plot suggests a clear minimum. However, using a fine, logarithmic grid, \subref{subfig:L1LinVsLog2log} reveals that it is only a sub-optimal local minimum before a very erratic part of $\gsure(\alpha,y)$ starts. \subref{subfig:L1CoarseGrid} shows how a coarse $\alpha$-grid  can lead to an arbitrary projection of $\gsure(\alpha,y)$ that is likely to miss important features. Both risk estimators are the same as those plotted in Figure \ref{subfig:L1GSURErisk} with the same colors. In \subref{subfig:L1iterOpt}, the difference between computing $\gsure(\alpha,y)$ with the consistent and highly accurate version of ADMM (Impl A) and with a standard ADMM version using only 20 iterations (Impl B) is illustrated.\label{fig:L1LinVsLog}}
\end{figure}

\section{Conclusion} \label{sec:Conclusion}

\reply{We examined variational regularization methods for ill-posed inverse problems and conducted extensive numerical studies that assessed the statistical properties different parameter choice rules. In particular, we were interested in the influence of the degree of ill-posedness of the problem (measured in terms of the condition of the forward operator) on the probability distributions of the selected regularization parameters and of the corresponding induced errors. This perspective revealed important features that were not discussed or noticed before but are essential to know for practical applications, namely that unbiased risk estimators encounter enormous difficulties: While the discrepancy principle yields a rather unimodal distribution of regularization parameters resembling the optimal one with slightly increased mean value, the PSURE estimates start to develop multimodality, and the additional modes consist of underestimated regularization parameters, which may lead to significant errors in the reconstruction.} For the case of SURE, which is based on a presumably more reliable risk, the estimates produce quite wide distributions (at least in logarithmic scaling) for increasing ill-posedness, in particular there are many highly underestimated parameters, which clearly yield bad reconstructions. We expect that this behaviour is rather due to the bad quality of the risk estimators than the quality of the risk. These findings may be explained by Theorem \ref{GSURER}, which indicates that the estimated SURE risk might deviate strongly from the true risk function $\msee$ when the condition number of $A$ is large, i.e. the problem is asymptotically ill-posed as $m \rightarrow 0$. Consequently one might expect a strong variation in the minimizers of $\gsure$ with varying $y$ compared to the ones of  $\msee$. A potential way to cure those issues is to develop novel risk estimates for $\msee$ that are not based on Stein's method, possibly it might even be useful not to insist on the unbiasedness of the estimators.

We finally mention that for problems like sparsity-promoting regularization, the SURE risk leads to additional issues, since it is based on a Euclidean norm. While the discrepancy principle and the PSURE risk only use the norms appearing naturally in the output space of the inverse problem (or in a more general setting the log-likelihood of the noise), the Euclidean norm in the space of the unknown is rather arbitrary. In particular, it may deviate strongly from the Banach space geometry in $\ell^1$ or similar spaces in high dimensions. Thus, different constructions of SURE risks are to be considered in such a setting, e.g. based on Bregman distances.

\appendix

\section{Proofs} \label{sec:Proofs}

\begin{proof}[Proof of Theorem \ref{PSUREl}]
We find
		\begin{align*}
		  \mathcal{L}&=\frac{1}{m}\|A\hat x-y+\varepsilon\|_2^2=\frac{1}{m}\|A\hat x-y\|_2^2+\frac{1}{m}\|\varepsilon\|_2^2+\frac{2}{m}\langle\varepsilon,A
			\hat x-y\rangle\\
			&=\frac{1}{m}\sum_{i=1}^{m}\frac{\alpha^2}{(\gamma_i^2+\alpha)^2}y_i^2-\frac{1}{m}\|U^*\varepsilon\|_2^2+\frac{2}{m}\langle\varepsilon,A
			\hat x-Ax^*\rangle\\
			&=\frac{1}{m}\sum_{i=1}^{m}\frac{\alpha^2}{(\gamma_i^2+\alpha)^2}y_i^2-\frac{1}{m}\|\tilde\varepsilon\|_2^2+\frac{2}{m}\langle\varepsilon,A
			\hat x-Ax^*\rangle ~,
		\end{align*}
 where $\tilde \varepsilon=U^*\varepsilon$, that is, $\tilde\varepsilon_i=\langle u_i,\varepsilon\rangle$.
Note that
  \begin{align*}
    A\hat x-A x^*&=U\Sigma\Sigma_{\alpha}^{-1}U^*(Ax^*+\varepsilon)-U\Sigma V^*x^*\\
				&=U\{\Sigma\Sigma_{\alpha}^{-1}-I\}\Sigma V^*x^*+U\Sigma\Sigma_{\alpha}^{-1}U^*\varepsilon~,
		\end{align*}
and recall from \eqref{spectral} that $x_i^*=\langle v_i,x^*\rangle.$ Since $U^*U=UU^*=I$, Var$[\tilde\varepsilon_i]=\sigma^2$. This yields
  \begin{align*}
	  \frac{2}{m}\langle\varepsilon\,,\, A\hat x-A x^*\rangle=\frac{2}{m}\sum_{i=1}^m \frac{\tilde\varepsilon_i^2\gamma_i^2}{\gamma_i^2+\alpha}-
		\frac{2}{m}\sum_{i=1}^m\frac{\alpha\tilde\varepsilon_i\gamma_ix_i^*}{\gamma_i^2+\alpha}.
	\end{align*}
We obtain the representation
  \begin{align*}
	  \frac{1}{m}\psure(\alpha,y)-\mathcal{L}&=-\sigma^2+\frac{2\sigma^2}{m}\sum_{i=1}^m\frac{\gamma_i^2}{\gamma_i^2+\alpha}+\frac{1}{m}\sum_{i=1}^m\tilde\varepsilon_i^2-\frac{2}{m}\sum_{i=1}^m\frac{\tilde\varepsilon_i^2\gamma_i^2}{\gamma_i^2+\alpha}+
		\frac{2}{m}\sum_{i=1}^m\frac{\alpha\tilde\varepsilon_i\gamma_ix_i^*}{\gamma_i^2+\alpha}\\
		&=\frac{1}{m}\sum_{i=1}^m(\tilde\varepsilon_i^2-\sigma^2)-\frac{2}{m}\sum_{i=1}^m\frac{\gamma_i^2}{\gamma_i^2+\alpha}(\tilde\varepsilon_i^2-\sigma^2)+\frac{2}{m}\sum_{i=1}^m\frac{\alpha\gamma_i}{\alpha+\gamma_i^2}x_i^*\tilde\varepsilon_i\\
		&=:Sl_1(\alpha)+Sl_2(\alpha)+Sl_3(\alpha),
	\end{align*}	
	where the terms $Sl_j(\alpha),\,j\in\{1,2,3\}$ are defined in an obvious manner.
Since $\tilde\varepsilon_1^2,\ldots,\tilde\varepsilon_n^2$ are independent and identically distributed with expectation $\sigma^2$ we immediately obtain that
\begin{align*}
  \sqrt{m}Sl_1(\alpha)=O_\mathbb{P}(\sigma^2).
\end{align*}
Note that $Sl_1(\alpha)$ is independent of $\alpha.$ 
\medskip
Next, we consider the term $Sl_2(\alpha).$

Due to \eqref{CondGamma} the  values  $\gamma_i^2/(\gamma_i^2+\alpha ) \in (0,1]$ for $\alpha\in[0,\infty)$, are monotonically decreasing
(with respect to $i$). Thus,  we find
\begin{align*}
	 \sup_{\alpha\in[0,\infty)}|Sl_2(\alpha)|&= \sup_{\alpha\in[0,\infty)}\frac{1}{m}\biggl|\sum_{i=1}^m\frac{\gamma_i^2}{\gamma_i^2+\alpha}(\tilde\varepsilon_i^2-\sigma^2)\biggr|
		\leq\sup_{1\geq c_1\geq\ldots\geq c_m\geq0}\frac{1}{m}\biggl|\sum_{i=1}^m c_i(\tilde\varepsilon_i^2-\sigma^2)\biggr|.
\end{align*}
It follows from  \cite{Li1985}, Lemma 7.2: 		
\begin{align*}
		\sup_{1\geq c_1\geq\ldots\geq c_m\geq0}\frac{1}{m}\biggl|\sum_{i=1}^m c_i(\tilde\varepsilon_i^2-\sigma^2)\biggr|
		=\sup_{1\leq j\leq m}\frac{1}{m}\biggl| \sum_{i=1}^j(\tilde\varepsilon_i^2-\sigma^2)\biggr|,
\end{align*}
and an application of  Kolmogorov's maximal inequality yields
\reply{
\begin{align}\label{eq:Kolmogorov}
\mathbb{P}\bigg(\sup_{\alpha\in[0,\infty)}|Sl_2(\alpha)|>\frac{\sigma^2T}{\sqrt{m}}\bigg)\leq\frac{m}{\sigma^4T^2}\mathrm{Var}\biggl( \frac{1}{m} \sum_{i=1}^m(\tilde\varepsilon_i^2-\sigma^2)\biggr)=\frac{2}{T^2}.
\end{align}
Hence
\begin{align*}
\lim_{T\to\infty}\limsup_{m\to\infty}\mathbb{P}\bigg(\sup_{\alpha\in[0,\infty)}|Sl_2(\alpha)|>\frac{\sigma^2T}{\sqrt{m}}\bigg)=0
\end{align*}
and therefore, by Definition \eqref{def:O},
}
\begin{align*}
	\sup_{\alpha\in[0,\infty)}|Sl_2(\alpha)|=O_\mathbb{P}\bigl(\sigma^2/\sqrt{m}\bigr),
\end{align*}
where	we also used that Var$(\tilde\varepsilon_i^2-\sigma^2)=2\sigma^4,$ which follows from $\tilde\varepsilon_i\sim\mathcal{N}(0,\sigma^2)$.\\
Finally, we estimate $Sl_3(\alpha).$
\reply{
Now, if $\alpha \geq 1$, then it follows from  condition \eqref{CondGamma} that  $0\leq\alpha\gamma_i/(\gamma_i^2+\alpha)\leq \alpha\gamma_i/\alpha=\gamma_i\leq 1$ and
$$
\frac{\alpha\gamma_i}{\gamma_i^2+\alpha} - \frac{\alpha\gamma_{i+1}}{\gamma_{i+1}^2+\alpha} = \frac{\alpha(\gamma_i-\gamma_{i+1})(\alpha -\gamma_i\gamma_{i+1})}{(\gamma_i^2+\alpha)(\gamma_{i+1}^2+\alpha)} ~\geq ~ 0, 
$$
and a  further application of  Kolmogorov's maximal inequality as in \eqref{eq:Kolmogorov}  yields
\begin{align*}
\sup_{\alpha\in[0,\infty)}|Sl_3(\alpha)|& 
\leq \sup_{1\geq c_1\geq\ldots\geq c_m\geq0}\frac{1}{m}\biggl|\sum_{i=1}^m c_ix_i^*\tilde\varepsilon_i\biggr|+\sup_{\alpha\in[0,1]}|Sl_3(\alpha)|\\
&=O_\mathbb{P}\Bigl(\sigma\|x^*\|_2/m\Bigr)+\sup_{\alpha\in[0,1]}|Sl_3(\alpha)|=O_\mathbb{P}\Bigl(\sigma/\sqrt{m}\Bigr)+\sup_{\alpha\in[0,1]}|Sl_3(\alpha)|.
\end{align*}
To determine its asymptotic order, we consider the term
$\bigl(Sl_3(\alpha),\alpha\in[0,1]\bigr)$  as a (Gaussian) stochastic process in $\alpha\in [0,1]$ for fixed $m$. Clearly, by the Cauchy-Schwarz inequality
\begin{align*}
Sl_3(\alpha)^2\leq\frac{1}{m}\sum_{i=1}^m \Bigl(\frac{\alpha\gamma_i}{\gamma_i+\alpha}x_i^{*}\Bigr)^2\cdot\frac{1}{m}\sum_{i=1}^m\tilde{\varepsilon}_i^2\leq \frac{1}{m}\sum_{i=1}^m(x_i^{*})^2\cdot\frac{1}{m}\sum_{i=1}^m\tilde{\varepsilon}_i^2.
\end{align*}
The first factor is bounded since,  by Assumption, $\|x^*\|_2^2=O(m)$ and
for any $m\in\mathbb{N}$, $\frac{1}{m}\sum_{i=1}^m\tilde{\varepsilon}_i^2$ is a random variable (independent of $\alpha$) and therefore almost surely bounded (w.r.t. $\alpha$). Hence, the process $\bigl(Sl_3(\alpha),\alpha\in[0,1]\bigr)$ is almost surely bounded (w.r.t. $\alpha\in[0,1]$).  
Recall that we need to show that $\sup_{\alpha\in[0,1]}|Sl_3(\alpha)|=O_{\mathbb{P}}(1/\sqrt{m})$, where the stochastic order symbol $O_{\mathbb{P}}(1/\sqrt{m})$ is defined in \eqref{def:O}. 
 Let $T>0.$ An application of the Markov inequality yields
\begin{align*}
\mathbb{P}\biggl(\sup_{\alpha\in[0,1]}|Sl_3(\alpha)|>\frac{\sigma T}{\sqrt{m}}\biggr)\leq \frac{2\sqrt{m}}{\sigma T}\mathbb{E}\biggl[\sup_{\alpha\in[0,1]}|Sl_3(\alpha)|\biggr].
\end{align*}
Since $\tilde\varepsilon$ and $-\tilde\varepsilon$ have the same distribution due to symmetry of the standard normal distribution,
\begin{align*}
\mathbb{E}\biggl[\sup_{\alpha\in[0,1]}|Sl_3(\alpha)|\biggr]\leq \mathbb{E}\biggl[\sup_{\alpha\in[0,1]}Sl_3(\alpha)\biggr]+\mathbb{E}\biggl[\sup_{\alpha\in[0,1]}(-Sl_3(\alpha))\biggr]=2\mathbb{E}\biggl[\sup_{\alpha\in[0,1]}Sl_3(\alpha)\biggr].
\end{align*}
Hence, the desired result follows if  we  show that $$
\mathbb{E}\biggl[\sup_{\alpha\in[0,1]}Sl_3(\alpha)\biggr]=O(\sigma/\sqrt{m}).
$$
 To do so, we apply the following Gaussian comparison inequality.
\begin{Theorem} [Sudakov-Fernique inequality (Theorem 2.2.3 in \cite{Adler2007})]\label{Thm:Sudakov}
Let $f$ and $g$ be a.s. bounded Gaussian processes on $T$. If 
\begin{align*}\mathbb{E}[f_t]=\mathbb{E}[g_t]\qquad\text{and}\qquad\mathbb{E}[(f_s-f_t)^2]\leq \mathbb{E}[(g_s-g_t)^2] 
\end{align*}
for all $s,t\in T,$ then
\begin{align*}
\mathbb{E}\bigg[\sup_{t\in T}f_t\bigg]\leq \mathbb{E}\bigg[\sup_{t\in T}g_t\bigg].
\end{align*}
\end{Theorem}
Let $\alpha_1,\alpha_2\in[0,1].$ 
\begin{align*}
&\mathbb{E}\bigl[\bigl(Sl_3(\alpha_1)-Sl_3(\alpha_2)\bigr)^2\bigr]=\frac{4\sigma^2}{m^2}\sum_{i=1}^m\biggl(\frac{\alpha_1^2\gamma_i^2}{(\gamma_i^2+\alpha_1)^2}-\frac{\alpha_2^2\gamma_i^2}{(\gamma_i^2+\alpha_2)^2}\biggr)^2(x_i^*)^2\\
&~=4(\alpha_1-\alpha_2)^2\frac{\sigma^2}{m^2}\sum_{i=1}^m\frac{\gamma_i^6}{(\gamma_i^2+\alpha_1)^2(\gamma_i^2+\alpha_2)^2}(x_i^*)^2\leq (\sqrt{\alpha_1}-\sqrt{\alpha_2})^2\frac{16\sigma^2}{m^2}\sum_{i=1}^m (x_i^*)^2.
\end{align*}
Consider the process 
\begin{align*}
\widetilde{Sl}_3:=\biggl(\widetilde{Sl}_3(\alpha)=\frac{4\sqrt{\alpha}}{m}\sum_{i=1}^m x_i^{*}\tilde{\varepsilon_i},\,\alpha\in[0,1]\biggr).
\end{align*}
Obviously, $\widetilde{Sl}_3$ is almost surely bounded and 
\begin{align*}
&\mathbb{E}\bigl[\bigl(\widetilde{Sl}_3(\alpha_1)-\widetilde{Sl}_3(\alpha_2)\bigr)^2\bigr]= (\sqrt{\alpha_1}-\sqrt{\alpha_2})^2\frac{16\sigma^2}{m^2}\sum_{i=1}^m(x_i^*)^2,
\end{align*}
which yields
\begin{align*}
\mathbb{E}\bigl[\bigl(Sl_3(\alpha_1)-Sl_3(\alpha_2)\bigr)^2\bigr]\leq\mathbb{E}\bigl[\bigl(\widetilde{Sl}_3(\alpha_1)-\widetilde{Sl}_3(\alpha_2)\bigr)^2\bigr]\quad\text{for all }\alpha_1,\alpha_2\in[0,1].
\end{align*}
Since $\mathbb{E}[Sl_3(\alpha)]=\mathbb{E}[\widetilde Sl_3(\alpha)]=0$ for all $\alpha\in[0,1],$ the assumptions of Theorem \ref{Thm:Sudakov} are satisfied, which allows us to conclude
\begin{align*}
\mathbb{E}\bigg[\sup_{\alpha\in[0,1]}Sl_3(\alpha)\bigg]\leq\mathbb{E}\bigg[\sup_{\alpha\in[0,1]}\widetilde{Sl}_3(\alpha)\bigg].
\end{align*}
Furthermore,
\begin{align*}
\mathbb{E}\bigg[\sup_{\alpha\in[0,1]}\widetilde{Sl}_3(\alpha)\bigg]\leq\mathbb{E}\bigg[\sup_{\alpha\in[0,1]}|\widetilde{Sl}_3(\alpha)|\bigg]\leq\mathbb{E}\bigg[\frac{4}{m}\bigg|\sum_{i=1}^m x_i^*\tilde{\varepsilon}_i\bigg|\bigg]=\sqrt{\frac{2}{\pi}}\sqrt{\frac{16\sigma^2}{m^2}\sum_{i=1}^m (x_i^*)^2},
\end{align*}
where we used that $\frac{4}{m}\sum_{i=1}^m x_i^*\tilde{\varepsilon}_j\sim\mathcal{N}\big(0,16\sigma^2/m^2\sum_{i=1}^m(x_i^*)^2\big)$ and that for a random variable $Z\sim\mathcal{N}(0,s^2)$ the first absolute moment is given by $\mathbb{E}|Z|=s\sqrt{2/\pi}.$ This yields
\begin{align*}
\lim_{T\to \infty}\limsup_{m\to\infty}\mathbb{P}\biggl(\sup_{\alpha\in[0,1]}|Sl_3(\alpha)|>\frac{\sigma T}{\sqrt{m}}\biggr)\leq\lim_{T\to \infty}\limsup_{m\to\infty}4\sigma\frac{\sqrt{2}\|x^*\|_2}{\sqrt{m\pi}T}=0,
\end{align*}
since, by Assumption, $\|x^*\|_2^2=O(m)$. By Definition \eqref{def:O} we  conclude $\sup_{\alpha\in[0,1]}|Sl_3(\alpha)|=O(\sigma/\sqrt{m}).$
}
\end{proof}

\begin{proof}[Proof of Corollary \ref{CorPSUREl}]
By definition $\psure(\alphaHatPSURE,y)\leq \psure(\alpha_m,y)$. This yields
\begin{align*}
  \mathbb{P}(\mathcal{L}(\alphaHatPSURE)\geq \mathcal{L}(\alpha_m)+\delta_m)&\leq
  \mathbb{P}\Bigl(\mathcal{L}(\alphaHatPSURE)-\frac{1}{m}\psure(\alphaHatPSURE,y)\geq \mathcal{L}(\alpha_m)-\frac{1}{m}\psure(\alpha_m,y)+\delta_m\Bigr)\\
  &\leq \mathbb{P}\Bigl(2\sup_{\alpha\in[0,\infty)}|\mathcal{L}(\alpha)-\frac{1}{m}\psure(\alpha,y)|\geq \delta_m\Bigr).
\end{align*}
It follows from Theorem \ref{PSUREl} that $\sup_{\alpha\in[0,\infty)}|\mathcal{L}(\alpha)-\frac{1}{m}\psure(\alpha,y)|=o_{\mathbb{P}}(\delta_m)$ for any sequence $(\delta_m)_{m\in\mathbb{N}}$ such that
$1/\delta_m=o(\sqrt{m}).$ By definition (see \eqref{def:o}),
\begin{align*}
\mathbb{P}\Bigl(\sup_{\alpha\in[0,\infty)}|\mathcal{L}(\alpha)-\frac{1}{m}\psure(\alpha,y)|\geq \delta_m\,\nu\Bigr)\to0\quad\text{for all}\quad \nu>0.
\end{align*}
Setting  $\nu=1/2$ above,
 the claim now follows.

\end{proof}

\begin{proof}[Proof of Theorem \ref{PSURER}]
Observing \eqref{PSUREspectral} and \eqref{e2} we find 
  \begin{align*}
	  \frac{1}{m}\bigl(\psure(\alpha,y)-\mspe(\alpha)\bigr)=\frac{1}{m}\sum_{i=1}^m\frac{\alpha^2}{(\gamma_i^2+\alpha)^2}\check{\varepsilon_i},
	\end{align*}	
where $\check{\varepsilon}_i:=y_i^2-\mathbb{E}[y_i^2]$. The random variables $\check{\varepsilon}_1,\ldots,\check{\varepsilon}_n$	are independent and centered. Notice that
 \begin{align*}
	 {\rm Var}[\check{\varepsilon}_i]&={\rm Var}[y_i^2]=\mathbb{E}[y_i^4]-(\mathbb{E}[y_i^2])^2=4\gamma_i^2{x_i^*}^2\sigma^2+2\sigma^4,
	\end{align*}
since $y_i\sim\mathcal{N}(\gamma_i{x_i^*},\sigma^2)$. 
\reply{
 Consider the monotonically increasing function $\alpha\mapsto\frac{\alpha^2}{(\gamma_i^2+\alpha)^2} \in [0,1]$  (where $\alpha\in[0,\infty)$)
 and note that the sequence $\big(\frac{1}{(\gamma_i^2+\alpha)^2}\big)_{i=1}^m$ is increasing. With the same arguments as in the proof of Theorem \ref{PSUREl} (see \eqref{eq:Kolmogorov}), using Kolmogorov's maximal inequality, we estimate
  \begin{align*}
	   &\sup_{\alpha\in[0,\infty)}\Bigl|\psure(\alpha,y)-\mspe(\alpha)\Bigr|=\sup_{\alpha\in[0,\infty)}
		  \biggl|\sum_{i=1}^m\frac{\alpha^2}{(\gamma_i^2+\alpha)^2}\check{\varepsilon}_i\biggr|
			=\sup_{1\leq j\leq m}\biggl|\sum_{i=j}^m\check{\varepsilon}_i\biggr|\\
            &~=O_\mathbb{P}\biggl(\biggl(\sum_{i=1}^m(4\gamma_i^2(x_i^*)^2+2\sigma^4)\biggr)^{\frac{1}{2}}\biggr).
	\end{align*}
}	
It remains to show the $L^2$-convergence \eqref{ExpSupPSURE}. To this end define the $j$-th partial sum
$$
  S_j:=\sum_{i=1}^j\check{\varepsilon}_i
$$
and observe that $\{S_j\,|\,j\in\mathbb{N}\}$ forms a martingale. The $L^p$-maximal  inequality  for martingales yields
  \begin{align*}
	 &\mathbb{E}\biggl(\sup_{\alpha\in[0,\infty)}\Bigl|\frac{1}{m}\bigl(\psure(\alpha,y)-\mspe(\alpha)\bigr)\Bigr|\biggr)^2
	  \leq  \mathbb{E}\biggl(\sup_{\alpha\in[0,\infty)}\Bigl|\frac{1}{m}\bigl(\psure(\alpha,y)-\mspe(\alpha)\bigr)\Bigr|^2\biggr)\\
		&  \leq \frac{1}{m^2}\mathbb{E}\Bigl(\sup_{1\leq j\leq m}|S_j|^2\Bigr)\leq\frac{4}{m^2}\mathbb{E}\biggl(\sum_{i=1}^m\check{\varepsilon}_i\biggr)^2=O\biggl(\frac{1}{m^2}\sum_{i=1}^m(4\gamma_i^2(x_i^*)^2+2\sigma^4)\biggr)
	\end{align*}
as above.	

\end{proof}

\begin{proof}[Proof of Lemma \ref{alphasurelemma}]

It is straightforward to see the differentiability of $\mspe$ and to compute
\begin{equation*}
{\mspe}'(\alpha) = \sum_{i=1}^m \frac{2 \gamma_i^4}{(\gamma_i^2 +\alpha)^3} (\alpha (x_i^*)^2 - \sigma^2). 
\end{equation*} 
Hence, for $\alpha < \frac{\sigma^2}{\max_i \vert x_i^*  \vert^2}$, the risk $\mspe$ is strictly decreasing, which implies the first inequality. Moreover, for $\alpha \geq 1$ we obtain
\begin{eqnarray*}
\alpha^3 {\mspe}'(\alpha) &=&  2   \sum_{i=1}^m \frac{\gamma_i^4}{(\gamma_i^2/\alpha + 1)^3} (\alpha (x_i^*)^2 - \sigma^2) \\ &>& 
 \frac{\alpha}{4}  \sum_{i=1}^m {\gamma_i^4}  (x_i^*)^2 - 2 \sigma^2
\sum_{i=1}^m \gamma_i^4
\end{eqnarray*}
and we finally see that $\mspe'$ is nonnegative if in addition
$\alpha \geq 8 \sigma^2 \frac{\sum \gamma_i^4}{\sum \gamma_i^4 (x_i^*)^2} $.

\end{proof}

\begin{proof}[Proof of Theorem \ref{Thm:alphaconv}]
From the uniform convergence of the sequence $f_{m_k}$ in Proposition \ref{fmproposition} we obtain the convergence of the minimizers $\hat \alpha_{\mspe,m_k}$. Combined with Theorem \ref{PSURER} we obtain an analogous argument for $\hat \alpha_{\psure,m_k}$.
\end{proof}

\begin{proof}[Proof of Theorem \ref{GSUREl}]\reply{
For $m=n$ and invertible matrices $A$  the projection $\Pi$ satisfies $\Pi={\rm id}$ and 
\begin{align*}
  \| x^*-\hat x_{\alpha}\|_2^2&=\|x^*\|_2^2-2\langle x^*,\hat x_{\alpha}\rangle+\|\hat x_{\alpha}\|_2^2=\sum_{i=1}^m(x_i^*)^2-2\langle x^*,\hat x_{\alpha}\rangle+\sum_{i=1}^m \frac{\gamma_i^2}{(\gamma_i^2 + \alpha)^2} y_i^2,
\end{align*} 
where we used \eqref{eq:xhat}. Recall from \eqref{modelSpectral} that $y_i=\gamma_ix_i^*+\tilde\varepsilon_i.$ This yields
\begin{align*}
\langle x^*,\hat x_{\alpha}\rangle&=\langle V\Sigma^{-1}U^*(y-\varepsilon),V\Sigma_{\alpha}^+U^*y\rangle=\langle \Sigma^{-1}U^*(y-\varepsilon),\Sigma_{\alpha}^+U^*y\rangle\\
&=\langle \Sigma^{-1}(y_i-\tilde\varepsilon_i)_{i=1}^m,\Sigma_{\alpha}^+(y_i)_{i=1}^m\rangle=\langle \Sigma^{-1}(\gamma_ix_{i}^*)_{i=1}^m,\Sigma_{\alpha}^+(y_i)_{i=1}^m\rangle\\
&=\langle (x_{i}^*)_{i=1}^m,\Sigma_{\alpha}^+(y_i)_{i=1}^m\rangle=\sum_{i=1}^m\frac{\gamma_ix_i^*}{\gamma_i^2+\alpha}y_i.
\end{align*}
Hence,
\begin{align*}
\| x^*-\hat x_{\alpha}\|_2^2=\sum_{i=1}^m(x_i^*)^2-2\sum_{i=1}^m\frac{x_i^*\gamma_i}{\gamma_i^2+\alpha}y_i+\sum_{i=1}^m \frac{\gamma_i^2}{(\gamma_i^2 + \alpha)^2} y_i^2.
\end{align*}
Recall from \eqref{GSUREspectral} that 
\begin{align} \label{GSURErep}
\gsure(\alpha,y) &= \sum_{i=1}^m \left( \frac{1}{\gamma_i} - \frac{\gamma_i}{\gamma_i^2 + \alpha} \right)^2 y_i^2 - \sigma^2 \sum_{i=1}^m \frac{1}{\gamma_i^2}  + 2 \sigma^2 \sum_{i=1}^m \frac{1}{\gamma_i^2 + \alpha}.
\end{align}
We obtain
\begin{align*}
\gsure(\alpha,y)-\| x^*-\hat x_{\alpha}\|_2^2&=
\sum_{i=1}^m\Bigl(\frac{1}{\gamma_i^2}-\frac{2}{\gamma_i^2+\alpha}\Bigr)(y_i^2-\sigma^2)-\|x^*\|_2^2+2\sum_{i=1}^m
\frac{\gamma_i^2(x_i^*)^2}{\gamma_i^2+\alpha}+2\sum_{i=1}^m\frac{\gamma_ix_i^*}{\gamma_i^2+\alpha}\tilde\varepsilon_i\\
&=\sum_{i=1}^m\Bigl(\frac{1}{\gamma_i^2}-\frac{2}{\gamma_i^2+\alpha}\Bigr)(y_i^2-\mathbb{E}[y_i^2])+2\sum_{i=1}^m\frac{\gamma_ix_i^*}{\gamma_i^2+\alpha}\tilde\varepsilon_i
\\
	&=2\alpha\sum_{i=1}^m\frac{x_i^*}{\gamma_i(\gamma_i^2+\alpha)}\tilde{\varepsilon}_i+\sum_{i=1}^m\frac{\alpha^2-\gamma_i^4}{\gamma_i^2(\gamma_1^2+\alpha)^2}(\tilde\varepsilon_i^2-\sigma^2)\\
	&=:GSl_1(\alpha)+GSl_2(\alpha),
	\end{align*}
	where $GSl_1(m,\alpha)$ and $GSl_2(m,\alpha)$ are defined in an obvious manner. Obviously,
	\begin{align*}
	\frac{\alpha}{\gamma_i^2+\alpha}\leq\frac{\alpha}{\gamma_{i+1}^2+\alpha}\quad\text{and}\quad0\leq\frac{\alpha}{\gamma_i^2+\alpha}\leq1
	\end{align*}
	by assumption \eqref{CondGamma}. Therefore
  \begin{align*}
    \sup_{\alpha\in[0,\infty)}|GSl_1(\alpha)|&=\sup_{\alpha\in[0,\infty)}\Biggl|2\alpha\sum_{i=1}^m\frac{x_i^*}{\gamma_i(\gamma_i^2+\alpha)}\tilde{\varepsilon}_i\Biggr|
		\leq\sup_{0\leq c_1\leq\ldots\leq c_m\leq1}\Biggl|2\sum_{i=1}^m c_i\frac{x_i^*}{\gamma_i}\tilde{\varepsilon}_i\Biggr|\\
		&=\sup_{1\leq j\leq m}\Biggl| 2\sum_{i=j}^m\frac{x_i^*}{\gamma_i}\tilde{\varepsilon}_i\Biggr|=O_\mathbb{P}\biggl(\sqrt{\sum_{i=1}^m\frac{(x_i^{\ast})^2}{\gamma_i^2}}\biggr),
	\end{align*}
where the last estimate follows from Kolmogorov's maximal inequality as in \eqref{eq:Kolmogorov} .	Now, since
\begin{align*}
c_m\sqrt{\sum_{i=1}^m\frac{(x_i^{\ast})^2}{\gamma_i^2}}\leq c_m\sqrt{\max_{1\leq i\leq m}|x_i^{\ast}|\sum_{i=1}^m\frac{1}{\gamma_i^2}}=O(\sqrt{c_m}).
\end{align*}
Next  we derive a corresponding  estimate for  the term $GSl_2(\alpha)$. Observe that
$0\leq\alpha/(\gamma_{i}^2+\alpha)\leq \alpha/(\gamma_{i+1}^2+\alpha)\leq1$ and $1\geq\gamma_{i}^4/(\gamma_{i}^2+\alpha)^2\geq\gamma_{i+1}^4/(\gamma_{i+1}^2+\alpha)^2\geq0$ for any $\alpha\geq0$ and any $1\leq i\leq m$
by ordering of the singular values.
This implies
  \begin{align*}
    \sup_{\alpha\in[0,\infty)}|GSl_2(\alpha)|&=\sup_{\alpha\in[0,\infty)}\biggl|\sum_{i=1}^m\frac{\alpha^2-\gamma_i^4}{\gamma_i^2(\gamma_i^2+\alpha)^2}(\tilde\varepsilon_i^2-\sigma^2)\biggr|\\
		&\leq \sup_{1\geq c_1\geq\ldots\geq c_m\geq0}\biggl|\sum_{i=1}^m\frac{c_i}{\gamma_i^2}(\tilde\varepsilon_i^2-\sigma^2)\biggr|+\sup_{0\leq c_1\leq\ldots\leq c_m\leq1}\biggl|\sum_{i=1}^m\frac{c_i}{\gamma_i^2}(\tilde\varepsilon_i^2-\sigma^2)\biggr|\\
		&\leq\sup_{1\leq j\leq m}\Biggl| \sum_{i=1}^j\frac{1}{\gamma_i^2}(\tilde\varepsilon_i^2-\sigma^2)\Biggr|+\sup_{1\leq j\leq m}\Biggl| \sum_{i=j}^m\frac{1}{\gamma_i^2}(\tilde\varepsilon_i^2-\sigma^2)\Biggr|=O_{\mathbb{P}}\left(\sqrt{\sum_{i=1}^m\frac{1}{\gamma_i^4}}\right),
	\end{align*}		
	by a further application of Kolmogorov's maximal inequality as in \eqref{eq:Kolmogorov}.
Notice that	$\sqrt{c_m}\leq1/\sqrt{m}$ and $1/\sqrt{m}\leq d_m\leq1$
Therefore, since
\begin{align*}
c_m\sqrt{\sum_{i=1}^m\frac{1}{\gamma_i^4}}=O\left(d_m\right),
\end{align*}
the claim of the theorem follows.
}	
\end{proof}

\begin{proof}[Proof of Theorem \ref{GSURER}]
\reply{
For full rank matrices $A\in \mathbb{R}^{m\times m}$  we have from \eqref{GSURErep}
  \begin{align*}
	  \gsure(\alpha,y)-\msee(\alpha)=\sum_{i=1}^m \left( \frac{1}{\gamma_i} - \frac{\gamma_i}{(\gamma_i^2 + \alpha)} \right)^2 \bigl(y_i^2
		-\mathbb{E}[y_i^2]\bigr)
		=\sum_{i=1}^m \left( \frac{1}{\gamma_i} - \frac{\gamma_i}{(\gamma_i^2 + \alpha)} \right)^2 \check{\varepsilon}_i.
	\end{align*}}
As in the proof of Theorem \ref{PSURER} we set $\check{\varepsilon}_i:=y_i^2-\mathbb{E}[y_i^2].$	Recall that the random variables $\check{\varepsilon}_i$ are centered, independent with Var$[\check{\varepsilon}_i]=4\gamma_i^2{x_i^*}^2\sigma^2+2\sigma^4$. We find
   \begin{align*}
	  \gsure(\alpha,y)-\msee(\alpha)=\frac{1}{\gamma_m^2}\sum_{i=1}^m\frac{\gamma_m^2}{\gamma_i^2} \frac{\alpha^2}{(\gamma_i^2+\alpha)^2} \check{\varepsilon_i}.
	\end{align*}
With the same arguments as in the proofs of Theorems	\ref{PSUREl} and \ref{PSURER} we obtain
  
\begin{align*}	
	\sup_{\alpha\in[0,\infty)}\Bigl|\gsure(\alpha,y)-\msee(\alpha)\Bigr|\leq  \sup_{0\leq c_1\leq c_2\leq\ldots\leq1}\Bigl|\frac{1}{\gamma_m^2}\sum_{i=1}^m c_i \check{\varepsilon_i}\Bigr|
	\leq\max_{1\leq j\leq m} \Bigl|\frac{1}{\gamma_m^2}\sum_{i=j}^m \check{\varepsilon_i}\Bigr|.
\end{align*}	
 Again, an application of Kolmogorov's maximal inequality yields
 	\begin{align*}	
	\sup_{\alpha\in[0,\infty)}\Bigl|\gsure(\alpha,y)-\msee(\alpha)\Bigr|=O_\mathbb{P}\Biggl(\biggl(4\sum_{i=1}^m\gamma_i^2{x_i^*}^2\sigma^2+2m\sigma^4\biggr)^{\frac{1}{2}}\Biggr)
\end{align*}
and the first claim of the theorem follows with $\cond(A) = \gamma_1/\gamma_m = 1/\gamma_m$.
{Moreover, in a similar manner as in the proofs of the previous theorems, we find
\begin{align*}
  \mathbb{E}\biggl(\sup_{\alpha\in[0,\infty)}\Bigl|\frac{1}{m \, \cond(A)^2}\bigr(\gsure(\alpha,y)-\msee(\alpha)\bigr)\Bigr|\biggr)^2\leq \mathbb{E}\sup_{1\leq j\leq m}\biggl|\frac{1}{\gamma_m^2}S_j\biggr|^2
\end{align*}
and by the $L^p$ maximal inequality the second claim now follows as
\begin{align*}
  \mathbb{E}\sup_{1\leq j\leq m}\biggl|\frac{1}{\gamma_m^2}S_j\biggr|^2\leq\frac{1}{\gamma_m^4} \mathbb{E}S_m^2=O\bigl(m/\gamma_m^4\bigr).
\end{align*}}
\end{proof}

\section{Consistent LASSO Solver} \label{sec:ADMM}

We want to solve \eqref{eq:VarReg} with $R(x) = \|x\|_1$ for a large number of different values of $\alpha$ but need to ensure that the results are comparable and consistent. For this, we rely on an implementation of the scaled version of ADMM \cite{BoPaChPeEc11} that carries out the iterations for all $\alpha$ simultaneously, with the same penalty parameter $\rho$ for all $\alpha$ and a stop criterion based on the maximal primal and dual residuum over all $\alpha$. Online adaptation of $\rho$ is also performed based on primal and dual residua for all $\alpha$. While ensuring the consistency of the results, this leads to sub-optimal performance for individual $\alpha$'s which has to be countered by using a large number of iterations to obtain high accuracies.
\begin{Algorithm}[All-At-Once ADMM]
Given $\alpha_1,\ldots,\alpha_{N_\alpha}$, $\rho > 0$ (penalty parameter), $\tau > 1$, $\mu > 1$ (adaptation parameters), $K \in \mathbb{N}$ (max. iterations) and $\varepsilon \geqslant 0$ (stopping tolerance), initialize $X^0, Z^0, U^0 \in \R^{n \times N_{\alpha}}$ by $0$, and $Y = y \otimes \mathds{1}_{N_\alpha}^T$, $\Lambda = [\alpha_1,\ldots,\alpha_{N_\alpha}] \otimes \mathds{1}_{n}$, where $\mathds{1}_q$ denotes an all-one column vector in $\R^q$. Further, let $\odot$ denote the component-wise multiplication between matrices (Hadamard product).\\

For $k=1,\ldots,K$ do:
\begin{align*}
X^{k+1} &= (A^* A + \rho I)^{-1} (A^* Y + \rho(Z^k-U^k)) & (x-update)\\[3pt]
Z^{k+1} &= \sign \left(X^{k+1} + U^{k}\right) \odot  \max\left(X^{k+1} + U^{k} - \Lambda/\rho,0\right) & (z-update)\\[3pt]
U^{k+1} &= U^k + X^{k+1} - Z^{k+1} & (u-update)\\[3pt]
r_i^{k+1} &= X_{(\cdot,i)}^{k+1}-Z_{(\cdot,i)}^{k+1} \hspace{1.7em} \quad \forall \quad i=1,\ldots,N_\alpha & (\text{primal residuum})\\[3pt]
s_i^{k+1} &= - \rho (Z_{(\cdot,i)}^{k+1}-Z_{(\cdot,i)}^k) \quad \forall \quad i=1,\ldots,N_\alpha & (\text{dual residuum})\\[3pt]
(U^{k+1},\rho) &= \begin{cases}
(U^{k+1}/\tau,\tau \rho) &\text{if} \hspace{1em} \# \left\lbrace i \; \Big| \; \norm{r_i^{k+1}}_2 > \mu \norm{s_i^{k+1}}_2 \right\rbrace > N_\alpha/2 \\[6pt]
(\tau U^{k+1},\rho/\tau) &\text{if} \hspace{1em} \# \left\lbrace i \; \Big| \; \norm{s_i^{k+1}}_2 > \mu \norm{r_i^{k+1}}_2 \right\rbrace > N_\alpha/2 \\[6pt]
(U^{k+1},\rho) &\text{else}.
\end{cases} & (\rho-\text{adaptation})\\[3pt]
\epsilon^{pri}_i &= \varepsilon \left(\sqrt{n} + \max(\norm{X_{(\cdot,i)}^{k+1}}_2,\norm{Z_{(\cdot,i)}^{k+1}}_2) \right) \quad \forall \quad i=1,\ldots,N_\alpha & (\text{primal stop tol})\\[3pt] 
\epsilon^{dual}_i &= \varepsilon \left( \sqrt{n} +  \rho \norm{U_{(\cdot,i)}^{k+1}}_2 \right) \hspace{6.2em} \quad \forall \quad i=1,\ldots,N_\alpha & (\text{dual stop tol})\\[3pt] 
\text{stop if} &\quad \norm{r_i^{k+1}}_2 < \epsilon_i^{pri} \; \wedge \; \norm{s_i^{k+1}}_2 < \epsilon_i^{dual} \quad \forall \quad i=1,\ldots,N_\alpha
\end{align*}  
\end{Algorithm}
The algorithm returns both $X_{(\cdot,i)}^{k+1}$ and $Z_{(\cdot,i)}^{k+1}$ as approximations of the solution to \eqref{eq:VarReg} with $R(x) = \|x\|_1$ and $\alpha = \alpha_i$ of which we use $Z_{(\cdot,i)}^{k+1}$ for our purposes as it is exactly sparse due to the soft-thresholding step (\textit{z-update}). In the computations, we furthermore initialized $\rho = 1$ and used $\tau = 2$, $\mu = 1.1$, $\varepsilon  = 10^{-14}$ and $K = 10^4$.

\bigskip

{\bf Acknowledgements.}
The work of N. Bissantz, H. Dette and K. Proksch 
has been supported  by the
Collaborative Research Center ``Statistical modeling of nonlinear
dynamic processes'' (SFB 823, Projects A1, C1, C4) of the German Research Foundation (DFG).


\bibliographystyle{amsplain} 
\bibliography{literature}

\end{document}